\documentclass[10pt]{article}

\usepackage{pifont}
\usepackage{amsmath}
\usepackage{cases}
\usepackage{amsfonts}
\usepackage{latexsym}
\usepackage{amssymb}
\usepackage{stmaryrd}
\usepackage{overpic}
\usepackage{color}

\usepackage{abstract}
\usepackage{indentfirst}
\usepackage{booktabs}
\usepackage{url}
\usepackage{verbatim}
\usepackage{diagbox}
\usepackage{multirow} 
\usepackage{graphicx}  
\usepackage{float}  
\usepackage{subfigure} 
\usepackage{subfig}
\usepackage{bm}
\usepackage[numbers,sort&compress]{natbib}
\usepackage{epstopdf}
\usepackage[ruled,linesnumbered]{algorithm2e}

\SetKwInOut{Input}{input}
\SetKwInOut{Output}{output}
\DontPrintSemicolon
\SetAlgoNlRelativeSize{-1}
\SetAlCapSkip{0.3em}

\makeatletter

\newcommand{\Rmnum}[1]{\expandafter\@slowromancap\romannumeral #1@}
\makeatother

\DeclareFontEncoding{FMS}{}{}
\DeclareFontSubstitution{FMS}{futm}{m}{n}
\DeclareFontEncoding{FMX}{}{}
\DeclareFontSubstitution{FMX}{futm}{m}{n}
\DeclareSymbolFont{fouriersymbols}{FMS}{futm}{m}{n}
\DeclareSymbolFont{fourierlargesymbols}{FMX}{futm}{m}{n}
\DeclareMathDelimiter{\VERT}{\mathord}{fouriersymbols}{152}{fourierlargesymbols}{147}

\newtheorem{theorem}{Theorem}[section]
\newtheorem{assumption}{Assumption}
\newtheorem{corollary}[theorem]{Corollary}

\newtheorem{example}[theorem]{Example}
\newtheorem{lemma}[theorem]{Lemma}

\newtheorem{proposition}[theorem]{Proposition}
\newtheorem{remark}{Remark}
\newenvironment{proof}[1][Proof]{\textbf{#1.} }
{\ \rule{0.75em}{0.75em}\smallskip}

\begin{document}

\begin{center}
\large\bf Structure-Oriented Randomized Neural Networks for Poisson--Nernst--Planck and Poisson--Nernst--Planck--Navier--Stokes Systems
\end{center}

\vspace{2mm}

\begin{center}
{\large\sc Yunlong Li}\footnote{School of Mathematics and Statistics, Xi'an Jiaotong University, Xi'an, Shaanxi 710049, P.R. China. E-mail: {\tt 4122107033@stu.xjtu.edu.cn}},\quad
{\large\sc Fei Wang}\footnote{School of Mathematics and Statistics \& State Key Laboratory of Multiphase Flow in Power Engineering, Xi'an Jiaotong University, Xi'an, Shaanxi 710049, China. The work of this author was partially supported by the Major Research Plan of
			the National Natural Science Foundation of China (Grant No.\ 92470115). Email: {\tt feiwang.xjtu@xjtu.edu.cn}} \quad

\end{center}

\vspace{2mm}

\begin{quote} 
\noindent{}{\bf Abstract}: 
We develop a structure-oriented randomized neural network framework, termed SO-RaNN, for the Poisson-Nernst-Planck (PNP) system and the Poisson-Nernst-Planck-Navier-Stokes (PNP-NS) system. The decoupled linearized subproblems are solved iteratively by randomized neural networks in a space-time framework. For the concentration variables, a pointwise cut-off is used to enforce positivity at the value level, and discrete mass-scaling factors are computed at selected correction instants and interpolated in time, so as to ensure exact mass matching at those instants and to promote approximate mass preservation between them. To introduce an auxiliary discrete dissipation mechanism, we further employ an SAV-type post-processing correction, which yields monotonicity of the SAV auxiliary variable under the ideal SAV update. For the PNP-NS system, a structure-preserving randomized neural network (SP-RaNN) is used for the velocity field, so that the velocity approximation satisfies the incompressibility constraint pointwise by construction. On the theoretical side, we derive residual-based estimates for the raw, uncorrected RaNN solvers of the linearized subproblems, formulate a conditional local-in-time convergence result for the raw outer Picard iteration of the PNP system, and analyze the value-level positivity correction together with the mass-correction and SAV post-processing steps. For the PNP-NS system, we establish an approximation result for the SP-RaNN space and provide a conditional error statement for the corresponding linearized Oseen-type problem. Numerical experiments demonstrate approximation accuracy in the source-driven manufactured tests and illustrate the intended value-level positivity correction, selected-time mass matching, computed free-energy curves based on the final gauge-fixed potential, and divergence-free approximation in benchmark tests.

{\bf Keywords:} structure-oriented randomized neural network, Poisson-Nernst-Planck, Poisson-Nernst-Planck-Navier-Stokes, space-time

{\bf Mathematics Subject Classification.} 65M06, 68T07, 41A46
\end{quote}

\vspace{2mm}

\section{Introduction}

The Poisson--Nernst--Planck (PNP) system is a coupled continuum model for ionic transport under an electric field(\cite{Nernst1889,Planck1890}). It consists of the Poisson equation and Nernst-Planck equations, where the Poisson equation determines the electrostatic potential generated by the distribution of charges, and Nernst-Planck equations describe the flux and evolution of each ionic species. The PNP system has been used in a wide range of applications, including semiconductors (\cite{Gajewski1986,Markowich1990,Jerome1996}), electrochemistry (\cite{Bazant2004}), and biology (\cite{Biler1994,Eisenberg1998}). When fluid motion is present, the PNP equations are coupled with the Navier--Stokes equations through the Coulomb force, leading to the Poisson-Nernst-Planck-Navier-Stokes (PNP-NS) system. This model is used to describe electrokinetic flows and has applications in biological ion channels and electrochemical systems (\cite{Choi2006,Schmuck2009}).

For sufficiently smooth positive solutions under suitable boundary conditions, the PNP system satisfies positivity of ionic concentrations, mass conservation, and free-energy dissipation. In the PNP-NS system, the incompressibility constraint for the velocity field is an additional structure. Numerically preserving these structures is important for stability and physical fidelity. Some structure-preserving methods have already been proposed for the PNP system (\cite{Liu2014,He2019,Hu2020,Ding2020,Shen2021,Huang2021,Fu2022,Dong2024,Bonilla2025}) and the PNP-NS system (\cite{Dehghan2023,Zhou2023,Pan2024,Yu2025}). For the PNP system, He et al. (\cite{He2019}) constructed an unconditionally stable semi-implicit linearized difference scheme based on the reformulation of the PNP system. Huang and Shen (\cite{Huang2021}) proposed an approach to construct structure-preserving schemes based on a suitable variable transformation and the scalar auxiliary variable (SAV, \cite{Shen2018}) method, which is flexible, easy to implement, and capable of achieving high-order accuracy with a computational complexity comparable to that of a semi-implicit scheme. Fu and Xu (\cite{Fu2022}) proposed a class of high-order accurate and structure-preserving finite element schemes based on the log-density formulation. Bonilla and Guti\'errez-Santacreu (\cite{Bonilla2025}) proposed two physics-based stabilized finite element methods. For the PNP-NS system, Zhou et al. (\cite{Zhou2023}) and Pan et al. (\cite{Pan2024}) developed efficient and structure-preserving schemes based on the SAV approach. Yu et al. (\cite{Yu2025}) proposed a decoupled and structure-preserving scheme, and carried out the error analysis for the fully discretized scheme.

In recent years, neural networks (NN) methods have attracted considerable attention due to their powerful approximation capability (\cite{Barron1993,Chen1995}). An increasing number of studies have focused on using neural network methods to solve PDEs (\cite{E2017,Sirignano2018,George2019,Lu2021,Jin2021,Sheng2021,Moseley2023}). These NN methods offer several attractive features such as mesh-free representation, high-dimensional capability and space-time formulation, but training NNs involves solving non-convex, nonlinear optimization problems, which may limit the accuracy and efficiency of NN-based approaches in practical PDE computations.

Randomized Neural Network (RaNN) methods (\cite{Pao1994,Igelnik1995,Huang2006}) fix the parameters in hidden layers and only train the parameters in the output layer, which results in a linear least-squares (LS) optimization. Therefore, RaNN methods balance the approximation error and the optimization error, and have demonstrated potential in solving PDEs 
(\cite{Dong2021,Dong2023,Wang2024,Chen2022,Xu2025,Shang2023,Shang2025,Sun2024,Sun2025,Dang2023,Dang2024,Li2025}). Dong et al. (\cite{Dong2021}) developed an efficient method that combines domain decomposition and local extreme learning machines (ELM, the single hidden-layer RaNN), enforcing continuity across shared boundaries. Chen et al. (\cite{Chen2022}) extended the local ELM method to the overlapping domain decomposition, which is called the random feature method. Li and Wang (\cite{Li2025}) developed local RaNN with finite difference methods for interface problems. Shang et al. (\cite{Shang2025}) integrated RaNNs with overlapping Schwarz domain decomposition, and constructed effective overlapping Schwarz preconditioners for solving the resulting linear systems. Dang and Wang (\cite{Dang2024}) proposed the adaptive-growth Randomized Neural Network framework, which significantly enhances the approximation capability of the RaNN, and established a unified error analysis. Some studies have combined RaNN methods with weak formulations, which is more effective for low regularity problems. Shang et al. (\cite{Shang2023}) integrated RaNNs with the Petrov-Galerkin method to solve both linear and nonlinear PDEs. Sun et al. (\cite{Sun2024}) proposed RaNN-DG method that couples local RaNNs by discontinuous Galerkin (DG) formulations. Dang and Wang (\cite{Dang2023}) combined local RaNNs with the hybridized discontinuous Petrov-Galerkin method for Stokes-Darcy problems. These studies indicate the potential of RaNNs for achieving accurate approximations with reduced training costs. A related and important issue is the choice of hidden-feature distributions. Since the hidden-layer parameters are fixed during the LS solve, the quality of the sampled feature space has a direct influence on the accuracy and robustness of RaNN-type solvers. In this direction, Zhang et al. (\cite{Zhang2024}) proposed transferable neural networks for PDEs by constructing neural feature spaces from a function-approximation viewpoint, where hidden neurons are re-parameterized and tuned to obtain transferable feature distributions. Yang and Wang (\cite{Yang2026}) proposed an adaptive-distribution RaNN framework, in which low-dimensional parameters of the hidden-feature sampling distribution are learned while the output-layer coefficients are still determined by a linear LS solve once the feature distribution is fixed.

Motivated by the difficulty of handling positivity, mass conservation, dissipation-related behavior, and, for the PNP-NS system, incompressibility within a RaNN framework, we propose a structure-oriented randomized neural network (SO-RaNN) method for the PNP and PNP-NS systems. The main contributions of this work are summarized as follows.

(i) We propose a space-time SO-RaNN framework for the PNP and PNP-NS systems, where the nonlinear coupled systems are treated by Picard-type linearization and reduced to a sequence of decoupled linear least-squares RaNN subproblems.

(ii) We develop a structure-oriented correction pipeline, including value-level positivity cut-off, selected-time mass correction, an SAV-type auxiliary update, an additional Nernst--Planck solve with the SAV-scaled potential, and a final gauge-fixed Poisson least-squares update.

(iii) For the PNP-NS system, we combine the above correction framework with an SP-RaNN velocity representation. 
As a result, the velocity approximation satisfies the incompressibility constraint pointwise by construction.

(iv) We provide a modular analysis that separates the raw RaNN approximation from the correction steps and the 
final Poisson least-squares update. The analysis includes residual-based estimates for the raw linearized subproblem solvers, a conditional local-in-time convergence result for the raw PNP Picard iteration, and a conditional error statement for the SP-RaNN approximation of the linearized Oseen-type subproblem.

The remainder of the paper is organized as follows. Section 2 reviews RaNNs and the SP-RaNN construction, and presents the PNP and PNP-NS models together with the proposed algorithms. Section 3 contains the theoretical analysis. For the PNP system, we study the raw RaNN solvers for the linearized subproblems, establish a conditional local convergence result for the raw outer Picard iteration, and analyze the correction steps. For the PNP-NS system, we discuss the divergence-free approximation and the corresponding conditional residual-based estimate for the linearized Oseen-type subproblem. Section 4 presents numerical experiments. Finally, Section 5 concludes the paper and discusses future directions.

\section{RaNN for the PNP and PNP-NS system}
In this section, we first review the construction of fully connected RaNNs and the divergence-free SP-RaNN space (\cite{Li2026}), then introduce the PNP and PNP-NS systems considered in this work and describe the proposed SO-RaNN algorithms for these models.

\subsection{Randomized Neural Networks}\label{RaNN}
We first consider a fully connected neural network $\Psi:\ \mathbb{R}^{n_0}\rightarrow \mathbb{R}^{n_D}$, which can be expressed recursively as follows:
\begin{align*}
    &\Psi_{0}(\mathbf{x}) = \mathbf{x},\\
    &\Psi_{k}(\mathbf{x}) = \rho(W_{k}\Psi_{k-1}+b_{k}), \quad k=1,...,D-1,\\
    &\Psi = \Psi_{D} = W_{D}\Psi_{D-1},
\end{align*}
where $D$ is the depth, $n_l,\ l=0,...,D$ is the width of each layer, $\rho$ is the activation function, $W_k \in \mathbb{R}^{n_k \times n_{k-1}},\ k=1,...,D$ and $b_k \in \mathbb{R}^{n_k}, k=1,...,D-1$ are the weights and biases in the $k$-th layer. For a standard neural network, the weights $W_k,\ k=1,...,D$ and the biases $b_k,\ k=1,...,D-1$ are trained by minimizing a loss function, which is usually nonlinear and nonconvex. Consequently, the training process could be computationally expensive and sensitive to the optimization strategy.

In order to reduce the difficulty of the training process, RaNN methods fix the hidden-layer parameters ($W_k,\ b_k,\ k=1,...,D-1$) after random initialization and only train the output-layer weights $W_D$. With the hidden features fixed, training reduces to a linear least-squares problem for the output coefficients. Since the parameters of the hidden layers are fixed, the neurons in the last hidden layer can be regarded as basis functions. For a scalar output, the RaNN approximation can therefore be written as a linear combination of the fixed hidden features:
\begin{equation}\label{RaNN-linear}
    u_\rho=\sum^{n_{D-1}}_{i=1}\alpha_i \psi_i,
\end{equation}
where $\psi_i$ is the output of neurons in the last hidden layer. The common fully-connected NN and RaNN are presented in Figure \ref{RaNN-p}.

\begin{figure}[!htbp] 		
	\centering
	\includegraphics[scale=0.25]{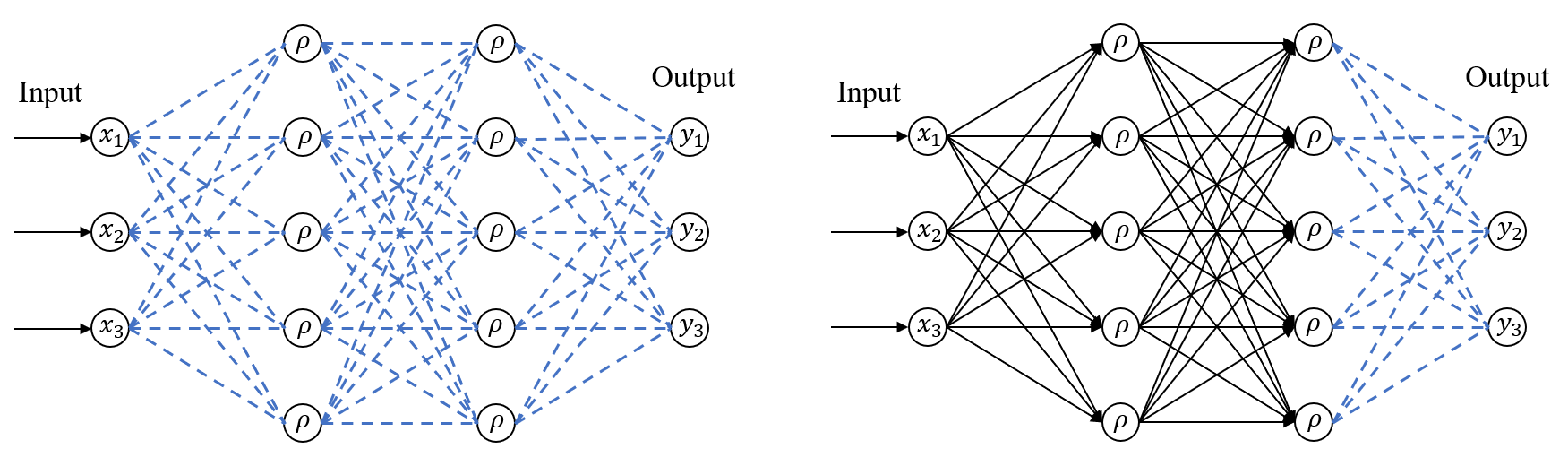}
	\caption{The structures of a two-hidden-layer NN (left) and a two-hidden-layer RaNN (right): the blue dash lines represent parameters that are learnable, and the black solid lines represent parameters that are randomly initialized and fixed.}
	\label{RaNN-p}
\end{figure}

Li and Wang (\cite{Li2026}) constructed the SP-RaNN space based on the equivalence characterization of divergence-free vector fields on simply connected domains. The resulting SP-RaNN $\mathbf{u}_{\rho}=\sum\limits_{j=1}^{m} \alpha_{j} \bm{\phi}^{j}$ satisfies the spatial divergence-free constraint pointwise by construction. For completeness, and to make the space-time implementation explicit, we recall the SP-RaNN basis functions used in this work. The hidden features may depend on time, whereas the divergence operator is always taken only with respect to the spatial variables.

For the three-dimensional time-dependent case, let
\begin{align*}
    &\bm{\phi}^{j}=(\phi_{1}^{j},\phi_{2}^{j},\phi_{3}^{j})^{\rm T},\\
    &\phi_{1}^j(\mathbf{x},t)=w^{j}_{32}\rho_3^j(h^{j}_{3}(\mathbf{x},t))-w^{j}_{23}\rho_2^j(h^{j}_{2}(\mathbf{x},t)),\\
    &\phi_{2}^j(\mathbf{x},t)=w^{j}_{13}\rho_1^j(h^{j}_{1}(\mathbf{x},t))-w^{j}_{31}\rho_3^j(h^{j}_{3}(\mathbf{x},t)),\\
    &\phi_{3}^j(\mathbf{x},t)=w^{j}_{21}\rho_2^j(h^{j}_{2}(\mathbf{x},t))-w^{j}_{12}\rho_1^j(h^{j}_{1}(\mathbf{x},t)),\\
    &h^{j}_i(\mathbf{x},t)=w^{j}_{i1}x+w^{j}_{i2}y+w^{j}_{i3}z+w_{it}^jt+b_i^j,\quad {\rm for}\ j=1,\cdots,m.
\end{align*}
Equivalently, the above construction can be interpreted as a curl-type construction. 
Indeed, if the vector-potential features are chosen as
\[
A^j=(\Theta_1^j(h_1^j),\Theta_2^j(h_2^j),\Theta_3^j(h_3^j)),
\qquad (\Theta_i^j)'=\rho_i^j,
\]
then taking the spatial curl of \(A^j\) gives the displayed basis functions. 
Therefore, \(\nabla\cdot\bm\phi^j=0\) pointwise.

For the two-dimensional time-dependent case, let
\begin{align*}
    &\bm{\phi}^{j}=(\phi_{1}^{j},\phi_{2}^{j})^{\rm T},\\
    &\phi_{1}^j(\mathbf{x},t)=w^{j}_{2}\rho^j(h^{j}(\mathbf{x},t)),\\
    &\phi_{2}^j(\mathbf{x},t)=-w^{j}_{1}\rho^j(h^{j}(\mathbf{x},t)),\\
    &h^{j}(\mathbf{x},t)=w^{j}_{1}x+w^{j}_{2}y+w_t^jt+b^j,\quad {\rm for}\ j=1,\cdots,m.
\end{align*}
Indeed,
\[
\nabla\cdot\bm{\phi}^j
=\partial_x\left(w_2^j\rho^j(h^j)\right)+\partial_y\left(-w_1^j\rho^j(h^j)\right)
=w_2^jw_1^j(\rho^j)'(h^j)-w_1^jw_2^j(\rho^j)'(h^j)=0.
\]

The simplest SP-RaNNs are shown in Figure \ref{SP-RaNN-p}. The weights between the input layer and the first hidden layer are randomly chosen and fixed, while the weights connecting the first hidden layer to the second hidden layer are determined accordingly. The biases in the second hidden layer are set to zero, and the activation functions are chosen as identity maps. The network structure shown in Figure \ref{SP-RaNN-p} can realize the previous derivation and thereby satisfy the spatial divergence-free condition $\nabla\cdot\mathbf u_\rho=0$ pointwise in space-time.

\begin{figure}[!htbp] 		
	\centering
	\includegraphics[scale=0.25]{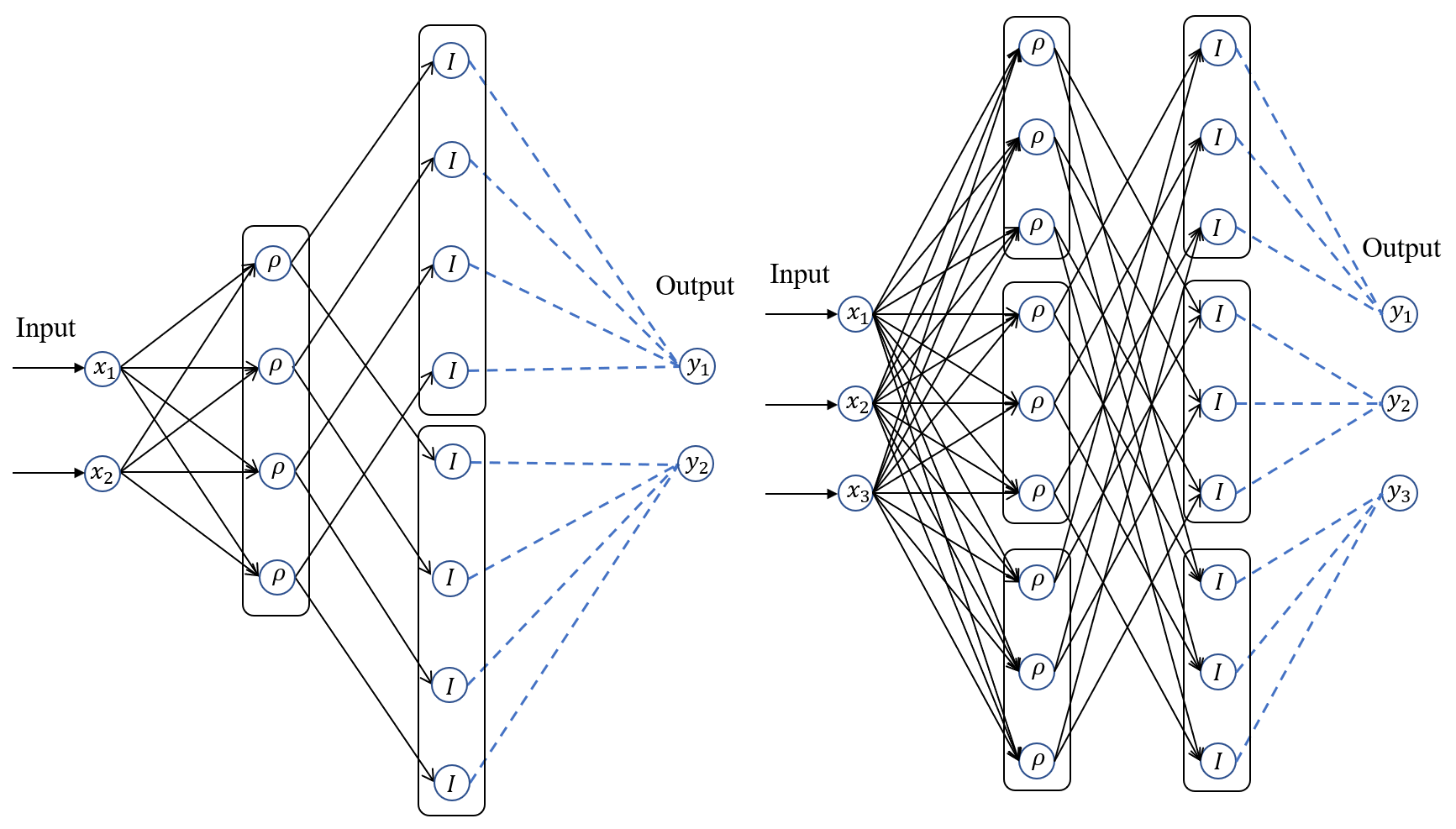}
	\caption{The structure of two-dimensional and three-dimensional SP-RaNNs: the blue dash lines represent tunable weights, while the black solid lines represent weights and biases that are fixed after initialization.}
	\label{SP-RaNN-p}
\end{figure}

\subsection{PNP Equations and PNP-NS Equations}\label{systems}

The PNP system \eqref{PNP-model} describes ionic transport under an electrostatic field through the coupling of the Nernst--Planck equations and the Poisson equation. In the Nernst--Planck equations, each ionic species evolves under diffusion driven by concentration gradients and electro-migration driven by the electric field. The Poisson equation determines the electrostatic potential generated by the charge distribution, thereby coupling the ionic concentrations with the electric field.

\begin{subequations}\label{PNP-model}
    \begin{align}
        \frac{\partial c_i}{\partial t} &= D_i \nabla \cdot(\nabla c_i + z_i c_i \nabla \phi),\ i=1,...,N,\ {\rm in}\ \Omega\times I,\\
        -\epsilon_p^2 \Delta \phi &= \sum\limits_{i=1}^{N}z_i c_i,\ {\rm in}\ \Omega\times I,
    \end{align}
\end{subequations}
where $\Omega$ is the spatial domain, $I$ is the time interval, $c_i(\mathbf{x},t)$ is the concentration of the $i$-th species, $\phi(\mathbf{x},t)$ is the electrostatic potential, $D_i>0$ is the diffusion coefficients, $z_i\in\mathbb{R}$ is the valence, and $\epsilon_p >0$ is the permittivity parameter. The initial condition is given by 
$$c_i(\mathbf{x},0)=c_{i,0}(\mathbf{x}),\ i=1,...,N.$$
On the boundary, we impose the no-flux condition:
\begin{equation}\label{no-flux-boundary}
    D_i (\nabla c_i + z_i c_i \nabla \phi)\cdot \mathbf{n}=0,\ i=1,...,N.
\end{equation}
If the homogeneous Neumann boundary condition is imposed for the Poisson equation, then the no-flux condition reduces to $\partial_{\mathbf{n}}c_i=0$.

In electrokinetic systems with fluid motion, ionic transport is further affected by advection, and the ion dynamics become coupled with the fluid flow. In this case, the concentrations are transported by diffusion, electro-migration, and advection, while the fluid is driven by viscosity, inertia, and the Coulomb force induced by the charge distribution. This bidirectional interaction between ion transport, electrostatic potential, and fluid motion is described by the PNP-NS system \eqref{PNPNS-model}:
\begin{subequations}\label{PNPNS-model}
    \begin{align}
        \frac{\partial c_i}{\partial t} + (\mathbf{u}\cdot \nabla)c_i &= D_i\nabla \cdot(\nabla c_i + z_i c_i \nabla \phi),\ i=1,...,N,\ {\rm in}\ \Omega\times I, \\
        -\epsilon_p^2 \Delta \phi &= \sum\limits_{i=1}^{N}z_i c_i,\ {\rm in}\ \Omega\times I,\\
        \frac{\partial \mathbf{u}}{\partial t} + (\mathbf{u}\cdot \nabla)\mathbf{u} - \nu \Delta \mathbf{u} + \nabla p &= - \left(\sum\limits_{i=1}^{N}z_i c_i\right)\nabla \phi,\ {\rm in}\ \Omega\times I, \\
        \nabla \cdot \mathbf{u} &= 0,\ {\rm in}\ \Omega\times I,
    \end{align}
\end{subequations}
where $\Omega$ is the spatial domain, $I$ is the time interval, $c_i(\mathbf{x},t)$ is the density of the $i$-th species, $\phi(\mathbf{x},t)$ is the electrostatic potential generated by the charged particles, $\mathbf{u}$ is the velocity field, $p$ is the pressure, the constants $D_i>0$ are the diffusion coefficients and $z_i\in\mathbb{R}$ are the valences of the ionic species, $\epsilon_p >0$ is the permittivity, and $\nu>0$ is the kinematic viscosity. The initial condition is given by
$$\mathbf{u}(\mathbf{x},0)=\mathbf{u}_0(\mathbf{x}),\ c_i(\mathbf{x},0)=c_{i,0}(\mathbf{x}),\ i=1,...,N.$$
We consider the blocking boundary conditions for the Nernst-Planck equation:
\begin{equation}
    (\mathbf{u}c_i-D_i(\nabla c_i+z_i c_i\nabla\phi))\cdot \mathbf{n}=0,\ i=1,...,N.
\end{equation}
We impose the homogeneous Neumann boundary condition for the Poisson equation and the no-slip boundary condition for the Navier--Stokes equations. Under the homogeneous blocking condition together with $\mathbf{u}=0$ and $\partial_{\mathbf{n}}\phi=0$, the boundary condition
\[
(\mathbf{u}c_i-D_i(\nabla c_i+z_i c_i\nabla\phi))\cdot \mathbf{n}=0
\]
reduces to $\partial_{\mathbf{n}}c_i=0$.

\begin{remark}
    Since the Poisson equation is equipped with Neumann boundary conditions, we assume the standard compatibility condition for the charge density whenever required. Under this condition, the electrostatic potential $\phi$ is determined only up to an additive constant. Throughout this paper, we fix this gauge by imposing $\int_\Omega \phi d\mathbf{x}=0$. Similarly, in the PNP-NS system, the pressure $p$ is also determined only up to an additive constant, and we use the normalization $\int_\Omega p d\mathbf{x}=0$. In the numerical implementation, these normalization conditions are incorporated into the least-squares system.
\end{remark}

For sufficiently smooth positive solutions under the corresponding homogeneous no-flux/blocking, Neumann, and no-slip boundary conditions, the PNP and PNP-NS systems formally satisfy the following free-energy laws. For convenience, we define $L_1(c_i,\phi)=\int_{\Omega} \sum_{i=1}^N D_i c_i|\nabla({\rm log}c_i + z_i \phi)|^2d\mathbf{x}$, and $L_2(c_i,\phi,\mathbf{u})=\int_{\Omega} \sum_{i=1}^N D_i c_i|\nabla({\rm log}c_i + z_i \phi)|^2+\nu|\nabla \mathbf{u}|^2d\mathbf{x}$.
\begin{equation}\label{PNP-e}
    E_1(t)=\int_{\Omega} \sum_{i=1}^N c_i {\rm log} c_i+\frac{\epsilon_p^2}{2}|\nabla \phi|^2d\mathbf{x},
\end{equation}
\begin{equation}\label{PNP-edt}
    \frac{dE_1(t)}{dt}=-\int_{\Omega} \sum_{i=1}^N D_i c_i|\nabla({\rm log}c_i + z_i \phi)|^2d\mathbf{x}, 
\end{equation}
\begin{equation}\label{PNPNS-e}
    E_2(t)=\int_{\Omega} \sum_{i=1}^N c_i {\rm log} c_i+\frac{\epsilon_p^2}{2}|\nabla \phi|^2+\frac{1}{2}|\mathbf{u}|^2d\mathbf{x},
\end{equation}
\begin{equation}\label{PNPNS-edt}
    \frac{dE_2(t)}{dt}=-\int_{\Omega} \sum_{i=1}^N D_i c_i|\nabla({\rm log}c_i + z_i \phi)|^2+\nu|\nabla \mathbf{u}|^2d\mathbf{x},
\end{equation}

\subsection{SO-RaNN for the PNP system}\label{RaNN-PNP}

For simplicity, in this section we present the two-ion case with initial data and homogeneous Neumann conditions for both the concentrations and the electrostatic potential. This corresponds to the homogeneous no-flux condition when $\partial_{\mathbf n}\phi=0$:
\begin{subequations}\label{PNP-model-2}
    \begin{align}
        \frac{\partial c_1}{\partial t} &= D_1 \nabla \cdot(\nabla c_1 + z_1 c_1 \nabla \phi),\ {\rm in}\ \Omega\times I, \\
        \frac{\partial c_2}{\partial t} &= D_2 \nabla \cdot(\nabla c_2 + z_2 c_2 \nabla \phi),\ {\rm in}\ \Omega\times I, \\
        -\epsilon_p^2 \Delta \phi &= z_1 c_1 + z_2 c_2,\ {\rm in}\ \Omega\times I, \\
        c_1(\mathbf{x},0) &=g_{1}(\mathbf{x}),\ {\rm on}\ \Omega\times \{0\},\\
        c_2(\mathbf{x},0) &=g_{2}(\mathbf{x}),\ {\rm on}\ \Omega\times \{0\},\\
        \frac{\partial c_1}{\partial \mathbf{n}} &=\frac{\partial c_2}{\partial \mathbf{n}}=\frac{\partial \phi}{\partial \mathbf{n}}=0,\ {\rm on}\ \partial \Omega \times I.
    \end{align}
\end{subequations}

For time-dependent problems, traditional numerical methods require time discretization, which leads
to error accumulation and makes it difficult to perform long-time simulations. Therefore, we employ a space-time RaNN approach, in which time and spatial variables are treated simultaneously, the time-dependent problem is reformulated as a space-time problem on \(Q=\Omega\times I\), with initial and boundary conditions imposed through the least-squares residual. We use three RaNNs $\bar{c}_{1,\rho},\ \bar{c}_{2,\rho},\ \bar{\phi}_{\rho}$ to approximate $c_1,\ c_2$ and $\phi$ separately. For convenience, we assume all networks have the same number of basis functions. Then we can define:
\begin{equation}\label{linear-combination1}
    \bar{c}_{i,\rho}=\sum_{k=1}^m \alpha^{i}_{k} \psi^{i}_{k},\ i=1,2,\ \bar{\phi}_{\rho}=\sum_{k=1}^m \beta_k \psi^3_k.
\end{equation}

The PNP system \eqref{PNP-model-2} is nonlinear, but the RaNN is in the form of a linear combination, so one may either solve a nonlinear least-squares problem directly or consider the nonlinear iterative method. Here we use an iteration method, specifically the Picard iteration \eqref{PNP-ite}. 

\begin{subequations}\label{PNP-ite}
    \begin{align}
        \frac{\partial c^{n+1}_1}{\partial t} &= D_1 \nabla \cdot\left(\nabla c^{n+1}_1+z_1 c^{n+1}_1\nabla \phi^{n}\right),\ {\rm in}\ \Omega\times I, \label{PNP-ite1}\\
        \frac{\partial c^{n+1}_2}{\partial t} &= D_2 \nabla \cdot\left(\nabla c^{n+1}_2+z_2 c^{n+1}_2\nabla \phi^{n}\right),\ {\rm in}\ \Omega\times I, \label{PNP-ite2}\\
        -\epsilon_p^2 \Delta \phi^{n+1} &= z_1 c^{n+1}_1 + z_2 c^{n+1}_2,\ {\rm in}\ \Omega\times I, \label{PNP-ite3}\\
        c^{n+1}_1(\mathbf{x},0) &=g_{1}(\mathbf{x}),\ {\rm on}\ \Omega\times \{0\}, \label{PNP-ite-ini1}\\
        c^{n+1}_2(\mathbf{x},0) &=g_{2}(\mathbf{x}),\ {\rm on}\ \Omega\times \{0\}, \label{PNP-ite-ini2}\\
        \frac{\partial c^{n+1}_1}{\partial \mathbf{n}} &=\frac{\partial c^{n+1}_2}{\partial \mathbf{n}}=\frac{\partial \phi^{n+1}}{\partial \mathbf{n}}=0,\ {\rm on}\ \partial \Omega \times I. \label{PNP-ite-boundary}
    \end{align}
\end{subequations}

In particular, we initialize the RaNNs to obtain $c^0_{1},\ c^0_{2},\ \phi^0$, and then perform the Picard iteration \eqref{PNP-ite}. At iteration step $n+1$, with the potential $\bar\phi_\rho^n$ fixed from
the previous iteration, we first solve the two linear least-squares problems
corresponding to the Nernst-Planck equations \eqref{PNP-ite1} and \eqref{PNP-ite2}. This gives
the raw concentration outputs
$\bar c_{1,\rho}^{n+1}$ and $\bar c_{2,\rho}^{n+1}$.
We then apply the pointwise cut-off $c_{i,\rho}^{+,n+1}:=\sigma_\delta(\bar c_{i,\rho}^{n+1}),\ \sigma_\delta(s)=\max\{s,\delta\},\quad i=1,2.$ The positivity-corrected concentrations $c_{1,\rho}^{+,n+1}$ and $c_{2,\rho}^{+,n+1}$ are used as the charge-density input in the subsequent gauge-fixed Poisson least-squares fitting step associated with \eqref{PNP-ite3}, from which we obtain the raw potential $\bar\phi_\rho^{n+1}$. They are also used as the value fields passed to the next Picard iteration. All differential operators in the Nernst-Planck residuals are still evaluated from the raw outputs $\bar c_{i,\rho}^{n+1}$, so that the non-smooth cut-off is used only at the value level. 

After the nonlinear iteration converges, we obtain the raw RaNN outputs $\bar{c}_{i,\rho}$ and $\bar{\phi}_\rho$, together with the positivity-corrected value fields $c^+_{i,\rho}$. We then apply a discrete-time mass correction to obtain corrected concentrations $\tilde{c}_{i,\rho}$ and use an SAV-type post-processing step to define an intermediate SAV-scaled potential \(\tilde\phi_\rho\). This step is designed to enforce positivity at the value level, match the prescribed masses at selected correction instants, and introduce monotonicity of the SAV auxiliary variable under the ideal SAV update. With $\tilde{\phi}_\rho$ fixed, we perform one additional linearized Nernst--Planck solve, apply the value-level cut-off and the final mass correction, and finally compute the final potential $\phi_\rho$ by a gauge-fixed Poisson LS step using the corrected concentrations.

In each iteration step, we obtain a linear LS problem \eqref{PNP-LS} by discretizing equations \eqref{PNP-ite} at $N_c$ collocation points including $N_1$ interior points, $N_2$ boundary points and $N_3$ initial points. In order to reduce storage space and computational load, we solve three decoupled linear LS problems \eqref{PNP-LS1}, \eqref{PNP-LS2} and \eqref{PNP-LS3}, which correspond to \eqref{PNP-ite1}, \eqref{PNP-ite2} and \eqref{PNP-ite3} (with boundary and initial conditions). Within each raw Picard step, these decoupled LS systems are solved sequentially as part of the same outer Picard iteration; no additional inner nonlinear iteration is introduced.

\begin{subequations}\label{PNP-LS}
    \begin{gather}
    A^1X_1=F^1,\ B^1X_1=G^1,\ D^1X_1=H^1,\label{PNP-LS1}\\
    A^2X_2=F^2,\ B^2X_2=G^2,\ D^2X_2=H^2,\label{PNP-LS2}\\
    A^3X_3=F^3,\ B^3X_3=G^3.\label{PNP-LS3}
    \end{gather}
\end{subequations}

We update parameters $\alpha^1_{k},\ \alpha^2_{k},\ \beta_{k},\ k=1,...,m$ in each iteration, and obtain the approximation $c^+_{i,\rho},\ \bar{\phi}_{\rho}$ after the iteration convergence. Then we sample time points $t_{j},\ j=1,...,N_t$ in $I$, and apply the post-processing steps to match the prescribed masses at the selected correction instants and to introduce an SAV auxiliary-variable monotonicity mechanism under the ideal SAV update:

(i) Positivity correction during the iteration. At each Picard step, we introduce the positivity cut-off map $\sigma_{\delta}(x)={\rm max}(x,\delta),\ \delta>0$ to cut off the output $\bar{c}_{i,\rho}$, and set $c^+_{i,\rho}=\sigma_{\delta}(\bar{c}_{i,\rho})$. Thus positivity is enforced at the value level in every iteration step. To avoid differentiating the non-smooth cut-off, we evaluate the differential operators from the raw network output $\bar{c}_{i,\rho}$ rather than from $\sigma_{\delta}(\bar{c}_{i,\rho})$. The cut-off is used only at the value level to enforce positivity.

(ii) Mass correction after convergence. After the nonlinear iteration converges, we sample time points $t_j,\ j=1,...,N_t$, and compute mass $m^j_1=\int_{\Omega}c^+_{1,\rho}(\mathbf{x},t_j)d\mathbf{x},\ m^j_2=\int_{\Omega}c^+_{2,\rho}(\mathbf{x},t_j)d\mathbf{x},\ j=1,...,N_t$, let $\gamma^j_1 = \frac{m^0_1}{m^j_1},\ \gamma^j_2 = \frac{m^0_2}{m^j_2}$, where $m^0_1,\ m^0_2$ are initial mass, so $\gamma^j_i >0$. Interpolating $\{\gamma^j_i\}$ in time yields a continuous scaling function $\gamma_i(t),\ i=1,2$, and the mass-corrected concentration is defined by  $\tilde{c}_{i,\rho}=\gamma_i(t)c^+_{i,\rho}$. This construction enforces exact mass matching at the selected correction instants $t_j$, while between two adjacent correction instants the mass is only approximately preserved. For mass-conservative homogeneous problems, the target mass is the initial mass. For a source-driven problem with a nonzero net source, this target should be replaced by the prescribed or exact mass at the corresponding correction instant.

(iii) SAV correction after convergence. Based on the mass-corrected
concentrations $\tilde c_{i,\rho}$ and the raw potential $\bar\phi_\rho$,
we define, at the selected correction instants $t_j$,
\[
\mathcal E_1^{j}
:= E_1(\tilde c_{1,\rho}^{j},\tilde c_{2,\rho}^{j},\bar\phi_\rho^{j}),
\qquad
\mathcal D_1^{j}
:= L_1(\tilde c_{1,\rho}^{j},\tilde c_{2,\rho}^{j},\bar\phi_\rho^{j}).
\]
The constant $C_0$ is chosen so that $\mathcal E_1^{j}+C_0>0$
for all correction instants. We initialize
\[
R_1^0=\mathcal E_1^0+C_0
\]
and update the auxiliary SAV variable by
\[
\frac{R_1^{j+1}-R_1^j}{\Delta t_j}
=
-\frac{R_1^{j+1}}{\mathcal E_1^{j+1}+C_0}\,
\mathcal D_1^{j+1}.
\]
Equivalently,
\[
R_1^{j+1}
=
\frac{R_1^j}
{1+\Delta t_j
\mathcal D_1^{j+1}/(\mathcal E_1^{j+1}+C_0)}.
\]
Since $\mathcal D_1^{j+1}\ge 0$ and $\mathcal E_1^{j+1}+C_0>0$,
the ideal SAV update gives $R_1^{j+1}\le R_1^j$.
We then set
\[
\xi^{j+1}=
\frac{R_1^{j+1}}{\mathcal E_1^{j+1}+C_0},
\]
interpolate $\{\xi^j\}$ in time, and define
$\tilde\phi_\rho=\xi(t)\bar\phi_\rho$.
The monotone quantity is the auxiliary sequence $R_1^j$.
We do not claim direct monotonicity of the original free energy
$E_1(\tilde c_{1,\rho},\tilde c_{2,\rho},\tilde\phi_\rho)$ unless an additional
consistency estimate between $R_1^j$ and the original energy is established.

In implementation, when the computed $\xi^j$ deviates excessively from one due to extremely small concentrations, we optionally replace it by $1$ as a numerical safeguard, see Remark \ref{xi}. This truncation is introduced only for robustness and is not included in the dissipation proof of Section~\ref{theory}. We formalize the preceding discussion into the Algorithm \ref{PNP-algorithm}.

\begin{algorithm}
\caption{SO-RaNNs for the PNP system}\label{PNP-algorithm}
\SetKwData{Left}{left}\SetKwData{This}{this}\SetKwData{Up}{up}
\SetKwFunction{Union}{Union}\SetKwFunction{FindCompress}{FindCompress}
\SetKwInOut{Input}{input}\SetKwInOut{Output}{output}
\Input{Computational domain $\Omega\times I$; Gauss-Legendre integration points $\mathbf{x}_{\rm int}$ and weights $w_{\rm int}$; iteration error $e=1$; threshold value $\eta$ and $\eta_{\phi}$; maximum number $M_{\rm ite}$ of iterations}
\Output{RaNN approximations $\tilde{c}_{1,\rho},\ \tilde{c}_{2,\rho}$ and the final potential $\phi_{\rho}$}
Randomly sample points in $\Omega\times I,\ \partial\Omega \times I$ and $\Omega \times \{0\}$ according to a uniform distribution; uniformly choose $t_j$ on $I$\;
Initialize RaNNs $\bar{c}_{1,\rho},\ \bar{c}_{2,\rho},\ \bar{\phi}_{\rho},$ and $e^0=1,\ n=0$, assemble the LS matrices\;
\While{$e^n>\eta$ and $n< M_{\rm ite}$}{solve the LS system associated with \eqref{PNP-ite1} to get $\bar{c}^{n+1}_{1,\rho}$ from $\bar{\phi}^n_{\rho}$\;
solve the LS system associated with \eqref{PNP-ite2} to get $\bar{c}^{n+1}_{2,\rho}$ from $\bar{\phi}^n_{\rho}$\;
obtain $c^{+,n+1}_{1,\rho}$ and $c^{+,n+1}_{2,\rho}$ from $\bar{c}^{n+1}_{1,\rho}$ and $\bar{c}^{n+1}_{2,\rho}$\;
solve the gauge-fixed Poisson LS system associated with \eqref{PNP-ite3} to obtain $\bar{\phi}^{n+1}_{\rho}$ from the charge density generated by $c^{+,n+1}_{1,\rho}$ and $c^{+,n+1}_{2,\rho}$\;
calculate iteration error $e^{n+1}:=\max\{\Vert c^{+,n+1}_{1,\rho}-c^{+,n}_{1,\rho}\Vert_0,\ \Vert c^{+,n+1}_{2,\rho}-c^{+,n}_{2,\rho}\Vert_0,\ \Vert \bar{\phi}^{n+1}_{\rho}-\bar{\phi}^{n}_{\rho}\Vert_0\}$\;
$n=n+1$\;}
calculate $\gamma^j_1,\ \gamma^j_2$, then obtain $\gamma_1(t),\ \gamma_2(t)$ and $\tilde{c}_{1,\rho},\ \tilde{c}_{2,\rho}$\;
compute $\mathcal E_1^j$ and $\mathcal D_1^j$ at the selected correction instants, choose $C_0$ such that $\mathcal E_1^j+C_0>0$, set $R_1^0=\mathcal E_1^0+C_0$, update $R_1^{j+1}=R_1^j/(1+\Delta t_j\mathcal D_1^{j+1}/(\mathcal E_1^{j+1}+C_0))$, and set $\xi^{j+1}=R_1^{j+1}/(\mathcal E_1^{j+1}+C_0)$\;
optionally apply the numerical safeguard: \If{$|\xi^j-1|>\eta_{\phi}$}{$\xi^j=1$}
obtain $\xi(t)$ and $\tilde{\phi}_{\rho}$\;
solve the LS system associated with \eqref{PNP-ite1} to get $\bar{c}_{1,\rho}$ from $\tilde{\phi}_{\rho}$\;
solve the LS system associated with \eqref{PNP-ite2} to get $\bar{c}_{2,\rho}$ from $\tilde{\phi}_{\rho}$\;
obtain $c^{+}_{1,\rho}$ and $c^{+}_{2,\rho}$ from $\bar{c}_{1,\rho}$ and $\bar{c}_{2,\rho}$\;
calculate the final scaling factors $\gamma^j_1,\ \gamma^j_2$, obtain $\gamma_1(t),\ \gamma_2(t)$ and update the final corrected concentrations $\tilde{c}_{1,\rho},\ \tilde{c}_{2,\rho}$\;
solve the gauge-fixed Poisson LS system associated with \eqref{PNP-ite3} to obtain the final potential $\phi_\rho$ from the corrected concentrations $\tilde{c}_{1,\rho}$ and $\tilde{c}_{2,\rho}$\;
\end{algorithm}

\begin{remark}\label{points}
    We determine the number $N_1,\ N_2,\ N_3$ of sampling points (interior, boundary and initial points) by the measurement ratio. For instance, $\Omega\times I=[a,b]\times[c,d]\times[0,T]$, set measurement-ratio parameter $n_e$, which means that we collect $n_e$ points at each unit length. Then $N_1=(b-a)(d-c)Tn^3_e,\ N_2=(2(b-a)+2(d-c))Tn^2_e,\ N_3=(b-a)(d-c)n^2_e$. We need to ensure that the system \eqref{PNP-LS} and \eqref{PNPNS-LS} are overdetermined by selecting suitable $n_e$ and $m$.
\end{remark}

\begin{remark}\label{xi}
    Under extreme conditions where the ion concentrations are close to zero, the denominator in the SAV factor may become very small, which can lead to an excessively small value of $\xi^j$. To enhance numerical robustness, we optionally introduce a safeguard parameter $\eta_\phi:$ if $|\xi^j-1|>\eta_\phi$, we reset $\xi^j=1$. This reset is an implementation-level stabilization device. The dissipation statement in Section 3.2.3 applies to the ideal SAV update before this safeguard is activated. When the safeguard is used, we do not claim the same unconditional dissipation property without additional analysis.
\end{remark}

\begin{remark}\label{long-time}
    When we consider long-time simulation, the NN-based space-time approach requires a large number of sampling points and more neurons, so it is a natural idea to implement temporal blocking. We divide $I$ into $I_l,\ l=1,...,N_l$, employ Algorithm \ref{PNP-algorithm} in each time block, the initial conditions are given by prior approximations when $l>1$. This temporal-blocking strategy is consistent with the local-in-time nature of Theorem~\ref{outer-picard}, although a full block-by-block stability analysis is not pursued here.
\end{remark}

\begin{remark}\label{compatibility}
For the homogeneous Neumann Poisson equation, the continuous compatibility condition is
\[
    \int_\Omega \sum_{i=1}^N z_i c_i(x,t)\,dx=0 .
\]
In the continuous Picard analysis, this condition is imposed through the admissible class, so the zero-mean Neumann Poisson operator is well defined.

In the implemented RaNN method, however, the Nernst--Planck subproblems are solved by least squares. Hence the no-flux boundary condition, the species mass identities, and the resulting charge-neutrality condition are not imposed as hard constraints at the raw-output level. The raw charge defect is not ignored: by Proposition~\ref{prop:raw-charge-defect}, it is controlled by the concentration approximation errors. The positivity cut-off is also not designed to enforce charge neutrality; it only introduces the additional controlled value-level perturbation estimated in Proposition~\ref{prop:charge-cutoff}.

Consequently, the implemented Poisson step should not be identified with a classical compatible Neumann Poisson solve unless the computed charge density satisfies the compatibility condition. In the numerical algorithm of this paper we keep the computed charge density unchanged and interpret the Poisson step as a gauge-fixed least-squares potential fit to the computed charge density.
\end{remark}
\subsection{SO-RaNN for the PNP-NS system}\label{RaNN-PNPNS}

For simplicity, we again present the two-ion case under homogeneous boundary conditions:
\begin{subequations}\label{PNPNS-model-2}
    \begin{align}
        \frac{\partial c_1}{\partial t} + (\mathbf{u}\cdot \nabla)c_1 &= D_1\nabla \cdot(\nabla c_1 + z_1 c_1 \nabla \phi),\ {\rm in}\ \Omega\times I, \\
        \frac{\partial c_2}{\partial t} + (\mathbf{u}\cdot \nabla)c_2 &= D_2\nabla \cdot(\nabla c_2 + z_2 c_2 \nabla \phi),\ {\rm in}\ \Omega\times I, \\
        -\epsilon_p^2 \Delta \phi &= z_1 c_1 + z_2 c_2,\ {\rm in}\ \Omega\times I, \label{PNPNS3}\\
        \frac{\partial \mathbf{u}}{\partial t} + (\mathbf{u}\cdot \nabla)\mathbf{u} - \nu \Delta \mathbf{u} + \nabla p&= - (z_1 c_1+z_2 c_2)\nabla \phi,\ {\rm in}\ \Omega\times I, \\
        \nabla \cdot \mathbf{u} &= 0,\ {\rm in}\ \Omega\times I, \label{PNPNS2-df}\\
        c_1(\mathbf{x},0) &=g_{1}(\mathbf{x}),\ {\rm on}\ \Omega\times \{0\},\\
        c_2(\mathbf{x},0) &=g_{2}(\mathbf{x}),\ {\rm on}\ \Omega\times \{0\},\\
        \mathbf{u}(\mathbf{x},0) &=\mathbf{g}_{3}(\mathbf{x}),\ {\rm on}\ \Omega\times \{0\},\\
        \frac{\partial c_1}{\partial \mathbf{n}} &=\frac{\partial c_2}{\partial \mathbf{n}}=\frac{\partial \phi}{\partial \mathbf{n}}=0,\ {\rm on}\ \partial \Omega \times I,\\
        \mathbf{u}&=\mathbf{0},\ {\rm on}\ \partial \Omega \times I.
    \end{align}
\end{subequations}

Similar to section \ref{RaNN-PNP}, we use four RaNNs $\bar{c}_{1,\rho},\ \bar{c}_{2,\rho},\ \bar{\phi}_{\rho},\ p_{\rho}$ to approximate $c_1,\ c_2,\ \phi,\ p$ separately, and use a SP-RaNN $\mathbf{u}_{\rho}$ to approximate $\mathbf{u}$. For convenience, we assume all networks have the same number of basis functions. Then we can define:
\begin{equation}\label{linear-combination2}
    \bar{c}_{i,\rho}=\sum_{k=1}^m \alpha^{i}_{k} \psi^{i}_{k},\ i=1,2,\ \bar{\phi}_{\rho}=\sum_{k=1}^m \alpha^{\phi}_k \psi^3_k,\ p_\rho=\sum_{k=1}^m \alpha^{p}_k \psi^4_k,\ \mathbf{u}_{\rho}=\sum_{k=1}^m \alpha^{\mathbf{u}}_k \bm{\Psi}_k.
\end{equation}
Here $\bm{\Psi}_k$ denotes the \(k\)-th vector-valued divergence-free SP-RaNN basis function. The divergence-free equation is not included as a separate least-squares residual, since the SP-RaNN velocity approximation $\mathbf{u}_{\rho}$ is divergence-free pointwise by construction (\cite{Li2026}). Here we still show the Picard's iteration formulation \eqref{PNPNS-ite}:
\begin{subequations}\label{PNPNS-ite}
    \begin{align}
        \frac{\partial c^{n+1}_1}{\partial t} + (\mathbf{u}^{n}\cdot \nabla)c^{n+1}_1 &= D_1 \nabla \cdot(\nabla c^{n+1}_1 + z_1 c^{n+1}_1 \nabla \phi^{n}),\ {\rm in}\ \Omega\times I, \label{PNPNS-ite1}\\
        \frac{\partial c^{n+1}_2}{\partial t} + (\mathbf{u}^{n}\cdot \nabla)c^{n+1}_2 &= D_2 \nabla \cdot(\nabla c^{n+1}_2 + z_2 c^{n+1}_2 \nabla \phi^{n}),\ {\rm in}\ \Omega\times I, \label{PNPNS-ite2}\\
        -\epsilon_p^2 \Delta \phi^{n+1} &= z_1 c^{n+1}_1 + z_2 c^{n+1}_2,\ {\rm in}\ \Omega\times I, \label{PNPNS-ite3}\\
        \frac{\partial \mathbf{u}^{n+1}}{\partial t} + (\mathbf{u}^{n}\cdot \nabla)\mathbf{u}^{n+1} - \nu \Delta \mathbf{u}^{n+1} + \nabla p^{n+1} &= - (z_1 c^{n+1} _1+z_2 c^{n+1} _2)\nabla \phi^{n+1},\ {\rm in}\ \Omega\times I, \label{PNPNS-ite4}\\
        c^{n+1}_1(\mathbf{x},0) &=g_{1}(\mathbf{x}),\ {\rm on}\ \Omega\times \{0\},\\
        c^{n+1}_2(\mathbf{x},0) &=g_{2}(\mathbf{x}),\ {\rm on}\ \Omega\times \{0\},\\
        \mathbf{u}^{n+1}(\mathbf{x},0) &=\mathbf{g}_{3}(\mathbf{x}),\ {\rm on}\ \Omega\times \{0\},\\
        \frac{\partial c^{n+1}_1}{\partial \mathbf{n}} &=\frac{\partial c^{n+1}_2}{\partial \mathbf{n}}=\frac{\partial \phi^{n+1}}{\partial \mathbf{n}}=0,\ {\rm on}\ \partial \Omega \times I,\\
        \mathbf{u}^{n+1}&=\mathbf{0},\ {\rm on}\ \partial \Omega \times I.
    \end{align}
\end{subequations}

Similarly, we initialize RaNNs and substitute \eqref{linear-combination2} into equations \eqref{PNPNS-ite}. After discretizing equations \eqref{PNPNS-ite} at collocation points, in each iteration step, we solve four linear LS problems \eqref{PNPNS-LS1}, \eqref{PNPNS-LS2}, \eqref{PNPNS-LS3}, \eqref{PNPNS-LS4}, which correspond to \eqref{PNPNS-ite1}, \eqref{PNPNS-ite2}, \eqref{PNPNS-ite3} and \eqref{PNPNS-ite4} (with boundary and initial conditions).

\begin{subequations}\label{PNPNS-LS}
    \begin{gather}
    A^1X_1=F^1,\ B^1X_1=G^1,\ D^1X_1=H^1,\label{PNPNS-LS1}\\
    A^2X_2=F^2,\ B^2X_2=G^2,\ D^2X_2=H^2,\label{PNPNS-LS2}\\
    A^3X_3=F^3,\ B^3X_3=G^3,\label{PNPNS-LS3}\\
    A^4X_4=F^4,\ B^4X_4=G^4,\ D^4X_4=H^4,\label{PNPNS-LS4}
    \end{gather}
\end{subequations}

As in the PNP case, the positivity cut-off is performed at each nonlinear iteration step on the raw concentration outputs, while the mass correction and the SAV correction are applied after convergence of the outer iteration, and the SAV-scaled potential is used in an additional post-processing iteration. The positivity-corrected concentrations $c^{+,n+1}_{i,\rho}$ are then used in the subsequent Poisson update, in the velocity-pressure update, and in the following Picard iteration, whereas the differential operators in the PDE residuals are still evaluated from the raw outputs.

We update the parameters $\alpha_k^1$, $\alpha_k^2$, $\alpha_k^\phi$, $\alpha_k^p$, and $\alpha_k^u$, $k=1,\ldots,m$, in each iteration, and after convergence obtain the positivity-corrected concentrations $c^+_{i,\rho}$ together with the raw outputs $\bar\phi_\rho$, $\mathbf{u}_\rho$, and $p_\rho$. We then sample time points $t_j$, $j=1,\ldots,N_t$, and apply the post-processing steps to correct the concentrations in mass and to introduce a monotonicity mechanism for the SAV auxiliary variable under the ideal SAV update.

(i) Positivity correction during the iteration. At each Picard step, we introduce the positivity cut-off map $\sigma_\delta(x)=\max(x,\delta)$, $\delta>0$, and set $c^+_{i,\rho}=\sigma_\delta(\bar c_{i,\rho})$. Thus positivity is enforced at the value level in every iteration step. To avoid differentiating the non-smooth cut-off, we evaluate the differential operators from the raw network output $\bar c_{i,\rho}$ rather than from $\sigma_\delta(\bar c_{i,\rho})$.

(ii) Mass correction after convergence. Compute mass $m^j_1=\int_{\Omega}c^+_{1,\rho}(\mathbf{x},t_j)d\mathbf{x},\ m^j_2=\int_{\Omega}c^+_{2,\rho}(\mathbf{x},t_j)d\mathbf{x},\ j=1,...,N_t$, let $\gamma^j_1 = \frac{m^0_1}{m^j_1},\ \gamma^j_2 = \frac{m^0_2}{m^j_2}$, where $m^0_1,\ m^0_2$ are initial mass, so $\gamma^j_i >0$. Interpolating $\{\gamma_i^j\}$ in time yields continuous scaling functions $\gamma_i(t)$, $i=1,2$, and we define the corrected concentrations by $\tilde c_{i,\rho}=\gamma_i(t)c^+_{i,\rho}$. This construction enforces exact mass matching at the selected correction instants $t_j$, while between two adjacent correction instants the mass is approximately preserved. For mass-conservative homogeneous problems, the target mass is the initial mass. For a source-driven problem with a nonzero net source, this target should be replaced by the prescribed or exact mass at the corresponding correction instant.

(iii) SAV correction after convergence. Based on the mass-corrected
concentrations $\tilde c_{i,\rho}$, the raw potential $\bar\phi_\rho$,
and the velocity $\mathbf u_\rho$, we define, at the selected correction
instants $t_j$,
\[
\mathcal E_2^{j}
:=
E_2(\tilde c_{1,\rho}^{j},\tilde c_{2,\rho}^{j},
\bar\phi_\rho^{j},\mathbf u_\rho^{j}),
\qquad
\mathcal D_2^{j}
:=
L_2(\tilde c_{1,\rho}^{j},\tilde c_{2,\rho}^{j},
\bar\phi_\rho^{j},\mathbf u_\rho^{j}).
\]
The constant $C_0$ is chosen sufficiently large so that
\[
\mathcal E_2^{j}+C_0>0
\]
for all correction instants. We initialize the auxiliary SAV variable by
\[
R_2^0=\mathcal E_2^0+C_0
\]
and update it through the ideal SAV relation
\begin{equation}\label{PNPNS-diss-d}
\frac{R_2^{j+1}-R_2^j}{\delta t_j}
=
-\frac{R_2^{j+1}}{\mathcal E_2^{j+1}+C_0}
\mathcal D_2^{j+1},
\end{equation}
where $\delta t_j=t_{j+1}-t_j$. Equivalently,
\begin{equation}\label{PNPNS-e-correction}
R_2^{j+1}
=
\frac{R_2^j}
{1+\delta t_j\,\mathcal D_2^{j+1}/(\mathcal E_2^{j+1}+C_0)}.
\end{equation}
Since $\mathcal D_2^{j+1}\ge 0$ and $\mathcal E_2^{j+1}+C_0>0$,
the ideal SAV update gives
\[
R_2^{j+1}\le R_2^j .
\]
We then define
\[
\xi^{j+1}
=
\frac{R_2^{j+1}}{\mathcal E_2^{j+1}+C_0},
\]
interpolate the values $\{\xi^j\}$ in time to obtain $\xi(t)$, and update
the potential by
\[
\tilde\phi_\rho=\xi(t)\bar\phi_\rho
=
\sum_{k=1}^m \xi(t)\alpha_k^\phi\psi_k^3 .
\]
The monotone quantity in this step is the auxiliary SAV sequence $R_2^j$.
We do not claim direct monotonicity of the original free energy
$E_2(\tilde c_{1,\rho},\tilde c_{2,\rho},\tilde\phi_\rho,\mathbf u_\rho)$
unless an additional consistency estimate between $R_2^j$ and the original
energy is established.

In conclusion, for all $(\mathbf{x},t)\in\Omega\times I$, the SP-RaNN
velocity satisfies the spatial incompressibility constraint $\nabla\cdot\mathbf u_\rho=0$ pointwise. Since
$\gamma_i(t)>0$, the mass-corrected concentrations remain positive; more
precisely,
\[
\tilde c_{i,\rho}(\mathbf{x},t)
\ge
\delta\,\gamma_i(t),
\]
and hence $\tilde c_{i,\rho}$ is positive whenever the interpolated scaling
factor $\gamma_i(t)$ is positive. At each selected correction instant $t_j$,
\[
\int_\Omega \tilde c_{i,\rho}(\mathbf{x},t_j)\,d\mathbf{x}=m_i^0.
\]
For the ideal SAV update, the auxiliary SAV sequence $R_2^j$ is nonincreasing. This monotonicity concerns the SAV auxiliary variable, not necessarily the original free energy evaluated at the fully post-processed fields. We formalize the preceding discussion into the Algorithm \ref{PNPNS-algorithm}, where $\mathbf{u}^n_{\rho}$ denotes the converged velocity field from the outer Picard iteration before the final post-processing update.

\begin{algorithm}
\caption{SO-RaNNs for the PNP-NS system}
\label{PNPNS-algorithm}
\SetKwData{Left}{left}\SetKwData{This}{this}\SetKwData{Up}{up}
\SetKwFunction{Union}{Union}\SetKwFunction{FindCompress}{FindCompress}
\SetKwInOut{Input}{input}\SetKwInOut{Output}{output}
\Input{Computational domain $\Omega\times I$; Gauss-Legendre integration points $\mathbf{x}_{\rm int}$ and weights $w_{\rm int}$; iteration error $e=1$; threshold value $\eta$ and $\eta_{\phi}$; maximum number $M_{\rm ite}$ of iterations}
\Output{RaNN approximations $\tilde{c}_{1,\rho},\ \tilde{c}_{2,\rho}$, the final potential $\phi_{\rho}$, the divergence-free velocity $\mathbf{u}_{\rho}$, and the pressure approximation $p_{\rho}$}
Randomly sample points in $\Omega\times I,\ \partial\Omega \times I$ and $\Omega \times \{0\}$ according to a uniform distribution; uniformly choose $t_j$ on $I$\;
Initialize RaNNs $\bar{c}_{1,\rho},\ \bar{c}_{2,\rho},\ \bar{\phi}_{\rho},\ p_\rho,\ \mathbf{u}_\rho$, and $e^0=1,\ n=0$, assemble the LS matrices\;
\While{$e^n>\eta$ and $n< M_{\rm ite}$}{solve the LS system associated with \eqref{PNPNS-ite1} to get $\bar{c}^{n+1}_{1,\rho}$ from $\bar{\phi}^n_{\rho}$ and $\mathbf{u}^{n}_{\rho}$\;
solve the LS system associated with \eqref{PNPNS-ite2} to get $\bar{c}^{n+1}_{2,\rho}$ from $\bar{\phi}^n_{\rho}$ and $\mathbf{u}^{n}_{\rho}$\;
obtain $c^{+,n+1}_{1,\rho}$ and $c^{+,n+1}_{2,\rho}$ from $\bar{c}^{n+1}_{1,\rho}$ and $\bar{c}^{n+1}_{2,\rho}$\;
solve the gauge-fixed Poisson LS system associated with \eqref{PNPNS-ite3} to obtain the raw potential $\bar{\phi}^{n+1}_{\rho}$ from the charge density generated by the positivity-corrected concentrations $c^{+,n+1}_{1,\rho}$ and $c^{+,n+1}_{2,\rho}$\;
solve the LS system associated with \eqref{PNPNS-ite4} to get $p^{n+1}_{\rho}$ and $\mathbf{u}^{n+1}_{\rho}$ from $\mathbf{u}^{n}_{\rho}$, $c^{+,n+1}_{1,\rho}$, $c^{+,n+1}_{2,\rho}$ and $\bar{\phi}^{n+1}_{\rho}$\;
calculate iteration error $e^{n+1}:=\max\{\Vert c^{+,n+1}_{1,\rho}-c^{+,n}_{1,\rho}\Vert_0,\ \Vert c^{+,n+1}_{2,\rho}-c^{+,n}_{2,\rho}\Vert_0,\ \Vert \bar{\phi}^{n+1}_{\rho}-\bar{\phi}^{n}_{\rho}\Vert_0,\ \Vert \mathbf{u}^{n+1}_{\rho}-\mathbf{u}^{n}_{\rho}\Vert_0,\ \Vert p^{n+1}_{\rho}-p^{n}_{\rho}\Vert_0\}$\;
$n=n+1$\;}
After convergence of the outer Picard iteration, denote the last raw velocity iterate by $\bar{\mathbf u}_\rho:=\mathbf u_\rho^n$.\;
calculate $\gamma^j_1,\ \gamma^j_2$, then obtain $\gamma_1(t),\ \gamma_2(t)$ and $\tilde{c}_{1,\rho},\ \tilde{c}_{2,\rho}$\;
compute $\mathcal E_2^j$ and $\mathcal D_2^j$ at the selected correction instants, choose $C_0$ such that $\mathcal E_2^j+C_0>0$, set $R_2^0=\mathcal E_2^0+C_0$, update $R_2^{j+1}=R_2^j/(1+\Delta t_j\mathcal D_2^{j+1}/(\mathcal E_2^{j+1}+C_0))$, and set $\xi^{j+1}=R_2^{j+1}/(\mathcal E_2^{j+1}+C_0)$\;
optionally apply the numerical safeguard: \If{$|\xi^j-1|>\eta_{\phi}$}{$\xi^j=1$}
obtain $\xi(t)$ and $\tilde{\phi}_{\rho}$\;
solve the LS system associated with \eqref{PNPNS-ite1} to get $\bar{c}_{1,\rho}$ from $\tilde{\phi}_{\rho}$ and $\bar{\mathbf u}_{\rho}$\;
solve the LS system associated with \eqref{PNPNS-ite2} to get $\bar{c}_{2,\rho}$ from $\tilde{\phi}_{\rho}$ and $\bar{\mathbf u}_{\rho}$\;
obtain $c^{+}_{1,\rho}$ and $c^{+}_{2,\rho}$ from $\bar{c}_{1,\rho}$ and $\bar{c}_{2,\rho}$\;
calculate the final scaling factors $\gamma^j_1,\ \gamma^j_2$, obtain $\gamma_1(t),\ \gamma_2(t)$ and update the final corrected concentrations $\tilde{c}_{1,\rho},\ \tilde{c}_{2,\rho}$\;
solve the gauge-fixed Poisson LS system associated with \eqref{PNPNS-ite3} to obtain the final potential $\phi_{\rho}$ from the final corrected concentrations $\tilde{c}_{1,\rho}$ and $\tilde{c}_{2,\rho}$\;
solve the LS system associated with \eqref{PNPNS-ite4} to get $p_{\rho}$ and $\mathbf{u}_{\rho}$ from $\bar{\mathbf u}_{\rho}$, $\tilde{c}_{1,\rho}$, $\tilde{c}_{2,\rho}$ and $\phi_{\rho}$\;
\end{algorithm}

\begin{remark}\label{penalty}
Usually we multiply both sides of the equations that are corresponding to boundary and initial conditions by a positive constant $\lambda$, which is helpful to emphasize the boundary and initial conditions during training. In numerical experiments we choose $\lambda=100$. Please refer to \cite{Li2026} (Remark 5) for a detailed discussion.
\end{remark}

\begin{remark}
The gauge-fixed Poisson step in Algorithms~\ref{PNP-algorithm} and~\ref{PNPNS-algorithm} is implemented as a least-squares fit of the Poisson residual together with the gauge constraint. If the computed charge density is not exactly charge-neutral, this step should be interpreted as a gauge-fixed least-squares potential fit rather than as a compatible Neumann Poisson solve.
\end{remark}

\section{Theoretical Analysis}\label{theory}
This section provides a stepwise theoretical analysis of the proposed method. We first establish approximation and residual estimates for the raw RaNN subproblem solvers. We then study the local contraction of the raw Picard map
for the PNP system, analyze the stability of the structure-oriented correction steps, and finally discuss the divergence-free approximation and Oseen-type subproblem arising in the PNP--NS system.

The full method is organized through an outer Picard iteration. For the PNP system, we complement the subproblem estimates with a local-in-time contraction result for the corresponding raw Picard map. For the PNP-NS system, the analysis is restricted to the divergence-free approximation and the linearized Oseen-type subproblem, and is not intended as a full convergence theory for the coupled outer iteration.

\subsection{Approximation and residual estimates for the raw RaNN solvers}\label{raw-rann}
We begin with the raw neural-network outputs before any positivity cut-off, mass correction, or SAV update is applied. For each frozen-coefficient linearized subproblem, we introduce the corresponding residual operators,
establish a graph-norm equivalence, and derive residual-based estimates following \cite{Dang2024}.

The full numerical method is organized through an outer Picard iteration for the nonlinear coupled system. For the PNP system, after analyzing the frozen-coefficient linearized subproblems, we further establish in Section \ref{local-picard} a local-in-time convergence result for the corresponding raw outer Picard map at the PDE level. Thus, the results in Sections \ref{np-raw} and \ref{poisson-raw} should be interpreted as stepwise estimates for the raw subproblem solvers, while Section \ref{local-picard} provides the complementary local convergence statement for the raw outer iteration. By contrast, for the PNP-NS system, the present analysis remains local to the linearized subproblems and does not constitute a complete convergence theory for the coupled outer iteration.

\subsubsection{Linearized Nernst–Planck subproblem}\label{np-raw}
We first consider the linearized Nernst–Planck equation appearing in the decoupled iteration. Fix an iteration index $n$, let the electric potential and the auxiliary coefficients in the linearized equation be determined by the previous iterate. In this subsection, we analyze only the raw RaNN approximation for the linearized Nernst–Planck subproblem itself; the positivity cut-off, mass correction, and SAV-based post-processing are treated later in Section \ref{post-processing}. 

For the linearized equation \eqref{PNP-ite1} with boundary condition \eqref{PNP-ite-boundary} and initial condition \eqref{PNP-ite-ini1}, we fix iteration index $n$ and set $\psi:=\phi^n,\ u:=c_1^{n+1}$. $\psi$ is obtained from the Poisson equation at the previous iteration step, so we assume $\Omega \subset \mathbb{R}^d$ is bounded with $C^{1,1}$ boundary, and
\begin{equation}\label{assumption1}
    \psi \in C^{\alpha}([0,T];W^{1,\infty}(\Omega))\cap L^{\infty}(0,T;W^{2,\infty}(\Omega)),\ \alpha>\frac{1}{4}.
\end{equation}

Let $Q=\Omega \times I$ and $S=\partial \Omega \times I$. We introduce the solution space $V=\{u:u\in L^2(0,T;H^2(\Omega)),\ \partial_t u\in L^2(0,T;L^2(\Omega))\}$, the residual space $Y=L^2(0,T;L^2(\Omega))$, the initial-data space $Z=H^1(\Omega)$, and the boundary-data space $X=L^2(0,T;H^{\frac{1}{2}}(\partial \Omega))\cap H^{\frac{1}{4}}(0,T;L^2(\partial \Omega))$, and define the operators $\mathcal{G}$, $\mathcal{B}_0$, and $\mathcal{B}_1$ according to the linearized Nernst–Planck equation, the initial condition, and the boundary condition, respectively. Specifically,
$$\mathcal{G}u:=\partial_t u-\mathcal{A}_{\psi}u \in Y,\ \mathcal{B}_0u:=u(0)\in Z,\ \mathcal{B}_1u:=(\partial_{\bf n}u+z_1u\partial_\mathbf{n}\psi)|_{\partial \Omega}\in X,$$ where $\mathcal{A}_{\psi}u:=D_1(\Delta u+z_1(\nabla u\cdot \nabla \psi+u\Delta \psi))$. The boundary condition is written as $\mathcal{B}_1u=g$. Under assumption \eqref{assumption1}, by Sobolev embedding and product estimates
\begin{equation}\label{A-bound}
    \Vert \mathcal{A}_{\psi}u\Vert_{L^2(\Omega)}\leq C_{\psi}\Vert u\Vert_{H^2(\Omega)},
\end{equation}
where $C_{\psi}$ depends on $D_1,\ |z_1|,\ \Vert \psi \Vert_{W^{2,\infty}},\ \Omega$.

The key step is to show that this residual-based loss is equivalent to the natural norm on $V$. Once such a graph-norm equivalence is established, it provides the continuous residual--error link for the ideal loss. For the implemented pointwise collocation loss, an additional discrete norming assumption is needed, as stated below.

\textbf{Graph norm equivalence}. We have a standard consequence of the Lions-Magenes theory (\cite{Lions1972}) applied to the triple $H^2(\Omega)\hookrightarrow H^1(\Omega) \hookrightarrow L^2(\Omega)$ as follows:

\begin{lemma}[Time trace into $H^1$]\label{trace-lemma-1}
    If $u\in L^2(0,T;H^2(\Omega))$ and $\partial_t u\in L^2(0,T;L^2(\Omega))$, then $u$ admits a representative in $C([0,T];H^1(\Omega))$. Moreover, there exists $C^0_{\rm tr}(T,\Omega)>0$ such that
    \begin{equation}\label{time-trace-bound}
        \mathop {\rm sup}\limits_{t\in [0,T]} \Vert u(t) \Vert^2_{H^1(\Omega)}\leq C^0_{\rm tr}\left( \Vert u \Vert^2_{L^2(H^2)} + \Vert \partial_t u \Vert^2_{L^2(L^2)} \right)=C^0_{\rm tr}\Vert u \Vert^2_V.
    \end{equation}
In particular, $\Vert \mathcal{B}_0 u \Vert_Z\le \sqrt{C^0_{\rm tr}}\Vert u \Vert_V$.
\end{lemma}

Since the potential \(\psi=\phi^n\) is fixed during the \((n+1)\)-th linearized Nernst--Planck solve, the coefficients \(\nabla\psi\), \(\Delta\psi\), and \(\partial_n\psi\) are treated as given frozen coefficients. The following regularity estimate is therefore formulated for this frozen-coefficient Robin problem.

\begin{lemma}[Frozen-coefficient Robin regularity]
\label{robin-regularity-lemma}
Assume that $\Omega\subset\mathbb R^d$ is a bounded $C^{1,1}$ domain and that
\[
\psi\in C^\alpha([0,T];W^{1,\infty}(\Omega))
\cap L^\infty(0,T;W^{2,\infty}(\Omega)),\qquad \alpha>\frac14 .
\]
Set
\[
\beta_\psi=z_1\partial_{\mathbf n}\psi,\qquad
b_\psi=-D_1z_1\nabla\psi,\qquad
c_\psi=-D_1z_1\Delta\psi .
\]
Assume that the Robin coefficient $\beta_\psi$ satisfies the regularity hypotheses for non-autonomous Robin boundary conditions in \cite{Arendt2016}, and that the data
\[
f\in L^2(0,T;L^2(\Omega)),\qquad
g\in X,\qquad
u_0\in H^1(\Omega)
\]
satisfy the usual initial-boundary compatibility condition. Then the solution of
\[
\partial_tu-D_1\Delta u+b_\psi\cdot\nabla u+c_\psi u=f,
\]
\[
\partial_{\mathbf n}u+\beta_\psi u=g,\qquad u(0)=u_0,
\]
satisfies
\[
\|u\|_{H^1(0,T;L^2(\Omega))}
+
\|u\|_{L^2(0,T;H^2(\Omega))}
\le
C_\psi
\left(
\|f\|_{L^2(0,T;L^2(\Omega))}
+\|g\|_X
+\|u_0\|_{H^1(\Omega)}
\right).
\]
Here $C_\psi$ depends on $\Omega,T,D_1,z_1$ and the corresponding norms of $\psi$ and $\beta_\psi$.
\end{lemma}

\begin{lemma}[Boundedness of $\mathcal{B}_1$]\label{trace-lemma-2}
Assume that the Robin coefficient
\[
\beta_\psi=z_1\partial_{\mathbf n}\psi
\]
has the regularity required in Lemma~\ref{robin-regularity-lemma}. Then there exists a constant $C_\psi>0$ such that, for all $u\in V$,
\[
\|\mathcal B_1u\|_X
=
\|\partial_{\mathbf n}u+\beta_\psi u\|_X
\le C_\psi\|u\|_V .
\]
The constant $C_\psi$ depends on $\Omega$, $T$, $z_1$, and the corresponding norms of $\beta_\psi$. The required trace multiplier property is included in the admissibility assumptions on the Robin coefficient $\beta_\psi$.
\end{lemma}

\begin{remark}
In the implemented RaNN solver, the frozen potential $\psi$ is the raw RaNN output from the previous Picard step. If a smooth activation function is used, then $\psi$ is a finite linear combination of smooth functions on the bounded space-time domain. Consequently, the coefficient regularity conditions imposed above are naturally satisfied for each fixed RaNN approximation. These assumptions are stated only to make the continuous frozen-coefficient parabolic regularity estimate applicable. The constants in the estimate may depend on the corresponding network coefficients.
\end{remark}

\begin{theorem}[Graph norm equivalence for NP equation]\label{graph-norm-np-thm}
    Assume \eqref{assumption1} and the hypotheses of Lemma~\ref{robin-regularity-lemma} for the frozen potential $\psi$ under consideration. Then there exist constants $C_L,\ C_U>0$ such that for all $u\in V$,
    \begin{equation}\label{graph-norm}
        C_L \Vert u \Vert^2_V \leq \Vert \mathcal{G}u \Vert^2_Y + \Vert \mathcal{B}_0u \Vert^2_Z + \Vert \mathcal{B}_1u \Vert^2_X  \leq C_U \Vert u \Vert^2_V,
    \end{equation}
    where $C_L$ and $C_U$ depend on $D_1$, $|z_1|$, $\Omega$, $T$, the norms of the frozen potential $\psi$ appearing in \eqref{assumption1}, and the regularity constant in Lemma~\ref{robin-regularity-lemma}.
\end{theorem}
\begin{proof}
    \textbf{Upper bound}. By triangle inequality and \eqref{A-bound},
    $$ \Vert \mathcal{G}u \Vert_Y \leq \Vert \partial_t u\Vert_Y + \Vert \mathcal{A}_{\psi}u \Vert_Y \leq \Vert \partial_t u\Vert_Y + C_{\psi}\Vert u \Vert_V. $$
    Thus 
    $$\Vert \mathcal{G}u \Vert^2_Y\leq 2\Vert \partial_t u\Vert_Y^2 + 2(C_\psi)^2\Vert u\Vert_V^2 \le C \Vert u\Vert_V^2.$$
    Moreover, combining Lemmas \ref{trace-lemma-1} and \ref{trace-lemma-2} yields the right inequality in \eqref{graph-norm}.\\
    \textbf{Lower bound}. set
    \[
    f=\mathcal Gu,\qquad g=\mathcal B_1u,\qquad u_0=\mathcal B_0u .
    \]
    Since these data are induced by the same function $u\in V$, the required initial-boundary compatibility condition in Lemma~\ref{robin-regularity-lemma} is satisfied in the trace sense. Then \(u\) is the solution of the frozen-coefficient Robin problem covered by Lemma~\ref{robin-regularity-lemma}. Hence
    \[
    \|u\|_{H^1(0,T;L^2)}+\|u\|_{L^2(0,T;H^2)}
    \le C_\psi
    \left(
    \|\mathcal Gu\|_Y+\|\mathcal B_1u\|_X+\|\mathcal B_0u\|_Z
    \right).
    \]
    This gives the lower inequality in \eqref{graph-norm}.
\end{proof}

\textbf{Ideal residual estimate and practical collocation loss.} Denote by $u^*\in V$ the exact solution of
\[
\mathcal Gu^*=f,\qquad \mathcal B_0u^*=u_0,\qquad \mathcal B_1u^*=g.
\]
For $u\in V$, define the continuous graph-norm residual loss
\begin{equation}\label{loss}
\mathcal L(u):=\|\mathcal Gu-f\|_Y^2+\|\mathcal B_0u-u_0\|_Z^2+\|\mathcal B_1u-g\|_X^2.
\end{equation}
By Theorem~\ref{graph-norm-np-thm}, an ideal residual minimizer associated with $\mathcal L$ satisfies a graph-norm error estimate. In the implementation, however, $\mathcal L$ is replaced by a pointwise collocation loss. Since the graph-norm loss contains Sobolev trace norms, pointwise sampling is not by itself a consistent discretization of the full continuous graph-norm loss. Therefore, a graph-norm estimate for the implemented solver is stated below only under an explicit discrete norming assumption on the chosen collocation set.

\textbf{Approximation error}.
Theorem 4.12 in \cite{Dang2024} provides a universal approximation result for randomized shallow neural networks in $H^s(D)$. Choosing $s$ such that $H^s(Q)\hookrightarrow V$ continuously, we obtain the following approximation result in the $V$-norm.

\begin{assumption}[Activation admissibility]\label{activation function}
    The activation function $\rho$ satisfies:\\
    (A1) $\rho\in C^s(\mathbb{R})$, and $\Vert \rho^{(s_0)}\Vert_{L^{\infty}(\mathbb{R})}<\infty$ for all $s_0=0,1,\ldots,s$.\\
    (A2) $\widehat{\rho'}\neq 0,\ \rho^{(s_0)}\in L^1(\mathbb R),\ s_0=1,\ldots,s+1.$
\end{assumption}
\begin{theorem}\label{approximation-np-thm}
    Assume that the activation function $\rho$ satisfies Assumption \ref{activation function}, and that the exact solution $u^*$ belongs to $H^{s+\epsilon_C}(Q)$ for some $\epsilon_C>0$, where $s$ is chosen such that $H^s(Q)\hookrightarrow V$ continuously. Then, for any $\epsilon_A>0$, there exists a randomized shallow neural network
    $u_a\in \mathcal{N}_\rho$ such that
    $$\Vert u^*-u_a\Vert_V<\epsilon_A$$
    with probability at least $1-\delta_p$, provided that the number of neurons is sufficiently large.
\end{theorem} 
\begin{proof}
    By \cite[Theorem 4.12]{Dang2024}, for any $\epsilon'_A>0$, there exists $u_a\in \mathcal{N}_\rho$ such that 
    $$\Vert u^*-u_a\Vert_{H^s(Q)}< \epsilon'_A$$
    with probability at least $1-\delta_p$.
    Because $H^s(Q)\hookrightarrow V$ is continuous, there exists $C_{\rm emb}>0$ such that 
    $$\Vert v\Vert_V\le C_{\rm emb}\Vert v\Vert_{H^s(Q)}.$$
    Let $\epsilon'_A=\frac{\epsilon_A}{C_{\rm emb}}$. Then
    $$\Vert u^*-u_a\Vert_{H^s(Q)}< \epsilon'_A.$$
    Therefore,
$$
\|u^*-u_a\|_V
\le C_{\rm emb}\|u^*-u_a\|_{H^s(Q)}
< \epsilon_A.
$$
\end{proof}

\textbf{Statistical bound for the $L^2$-type empirical loss}. For the statistical component, we restrict attention to the practical collocation loss based on pointwise $L^2$-type sampling averages, the norms are evaluated by quadrature formulas based on interior and boundary collocation points: $\{\mathbf{x}^I_i\}^{N_{\rm in}}_{i=1}\subset Q$, $\{\mathbf{x}^O_i\}^{N_{\rm ini}}_{i=1} \subset \Omega$ and $\{\mathbf{x}^B_i\}^{N_{\rm b}}_{i=1} \subset S$, which are drawn i.i.d. from the uniform distributions on $Q,\ \Omega,\ S$ respectively, so we define the empirical loss function
\begin{equation}\label{loss_f}
    \hat{\mathcal{L}}_2(u):= \frac{|Q|}{N_{\rm in}}\sum^{N_{\rm in}}_{i=1}|\mathcal{G}u(\mathbf{x}^I_i)-f(\mathbf{x}^I_i)|^2+\frac{|\Omega|}{N_{\rm ini}}\sum^{N_{\rm ini}}_{i=1}|\mathcal{B}_0u(\mathbf{x}^O_i)-u_0(\mathbf{x}^O_i)|^2+\frac{|S|}{N_{\rm b}}\sum^{N_{\rm b}}_{i=1}|\mathcal{B}_1u(\mathbf{x}^B_i)-g(\mathbf{x}^B_i)|^2.
\end{equation}
To estimate the statistical component, we follow the argument of \cite[Theorem 4.14]{Dang2024}. Since the initial-data term in $\hat{\mathcal{L}}$ has the same form as an empirical average of a uniformly bounded quantity over i.i.d. samples, it can be treated in the same way as the interior and boundary components.
\begin{theorem}[Statistical bound for the $L^2$-type empirical loss]\label{statistical-np-thm}
Let
\[
\mathcal L_2(u)=\|\mathcal Gu-f\|_{L^2(Q)}^2+\|\mathcal B_0u-u_0\|_{L^2(\Omega)}^2+\|\mathcal B_1u-g\|_{L^2(S)}^2
\]
and let $\widehat{\mathcal L}_2(u)$ be its Monte Carlo collocation approximation. Under the stated uniform boundedness assumptions on the residual functions generated by the finite-dimensional RaNN trial space, with probability at least $1-\delta_s$,
\[
\sup_{u\in\mathcal N_\rho}|\mathcal L_2(u)-\widehat{\mathcal L}_2(u)|
\lesssim
C_N^2C_R^2M
\left(
\sqrt{\frac{\log(M/\delta_s)}{N_{\rm in}}}
+
\sqrt{\frac{\log(M/\delta_s)}{N_{\rm ini}}}
+
\sqrt{\frac{\log(M/\delta_s)}{N_{\rm b}}}
\right).
\]
Here $M$ denotes the number of trainable output coefficients. This theorem concerns only the $L^2$-type residual loss $\mathcal L_2$. It does not by itself imply convergence in the graph norm associated with the continuous loss $\mathcal L$. The graph-norm estimate for the implemented collocation solver requires the discrete graph-norm equivalence assumption stated below.
\end{theorem}

\textbf{Conditional estimate for the collocation-based RaNN solver}. We next derive a deterministic conditional estimate for the implemented least-squares solver. The estimate uses the continuous graph-norm equivalence
and the discrete graph-norm equivalence assumption on the finite-dimensional RaNN error space. The statistical bound above provides motivation for the collocation loss, but the following result is conditional on the norming
property of the chosen collocation set.

For the empirical least-squares problem, let \(\widehat{\mathcal L}_N(v)\)
denote the collocation loss with the prescribed right-hand side, initial data,
and boundary data:
\[
\widehat{\mathcal L}_N(v)
=
\|\mathcal Gv-f\|_{Y,N}^2
+
\|\mathcal B_0v-u_0\|_{Z,N}^2
+
\|\mathcal B_1v-g\|_{X,N}^2 .
\]
For an error function \(e\), we define the corresponding homogeneous empirical
residual loss by
\[
\widehat{\mathcal L}_N^0(e)
:=
\|\mathcal Ge\|_{Y,N}^2
+
\|\mathcal B_0e\|_{Z,N}^2
+
\|\mathcal B_1e\|_{X,N}^2 .
\]
Similarly, let
\[
\mathcal L^0(e)
:=
\|\mathcal Ge\|_{Y}^2
+
\|\mathcal B_0e\|_{Z}^2
+
\|\mathcal B_1e\|_{X}^2
\]
be the corresponding continuous homogeneous residual loss.
Since the exact solution \(u^*\) satisfies
\[
\mathcal Gu^*=f,\qquad
\mathcal B_0u^*=u_0,\qquad
\mathcal B_1u^*=g,
\]
we have, by linearity of the residual operators,
\[
\widehat{\mathcal L}_N(w_\rho)
=
\widehat{\mathcal L}_N^0(w_\rho-u^*)
\qquad
\text{for all } w_\rho\in\mathcal N_\rho .
\]

\begin{assumption}[Discrete graph-norm equivalence on the RaNN error space]
\label{ass:discrete-graph-equivalence}
Let
\[
\mathcal E_\rho(u^*)
:=
\{v_\rho-u^*: v_\rho\in\mathcal N_\rho\}.
\]
We assume that the collocation set is norming for the residuals generated by
\(\mathcal E_\rho(u^*)\). Namely, there exist constants
\(c_{\rho,N},C_{\rho,N}>0\) such that, for all \(e\in\mathcal E_\rho(u^*)\),
\[
c_{\rho,N}\mathcal L^0(e)
\le
\widehat{\mathcal L}_N^0(e)
\le
C_{\rho,N}\mathcal L^0(e).
\]
\end{assumption}

After the random hidden-layer parameters are fixed, the RaNN trial space \(\mathcal N_\rho\) is finite-dimensional. In the implementation, the continuous graph-norm residual loss is replaced by the pointwise collocation loss \(\widehat{\mathcal L}_N\). Since pointwise sampling does not automatically yield uniform control of the full graph-norm residual, we impose Assumption \ref{ass:discrete-graph-equivalence}, which states that the chosen collocation set is norming for the residuals generated by the finite-dimensional RaNN error space.

Under this discrete norm-equivalence assumption, the empirical least-squares minimizer can be compared with the ideal continuous residual minimizer on \(\mathcal N_\rho\). The resulting estimate is therefore conditional on the
constants \(c_{\rho,N}\) and \(C_{\rho,N}\).

\begin{corollary}[Conditional graph-norm estimate for the implemented RaNN solver]
\label{cor:NP-implemented-conditional}
Let $u^*\in V$ be the exact solution of the linearized Nernst--Planck
subproblem, and let $u_\rho\in\mathcal N_\rho$ be a minimizer of the empirical
collocation loss $\widehat{\mathcal L}_N$. Assume that the graph-norm
equivalence in Theorem \ref{graph-norm-np-thm} holds and that
Assumption \ref{ass:discrete-graph-equivalence} is satisfied. If
\[
\inf_{v_\rho\in\mathcal N_\rho}\|v_\rho-u^*\|_V\le \epsilon_A,
\]
then
\[
\|u_\rho-u^*\|_V
\le
C\left(\frac{C_{\rho,N}}{c_{\rho,N}}\right)^{1/2}\epsilon_A .
\]
In particular, if $C_{\rho,N}/c_{\rho,N}$ is uniformly bounded, then
\[
\|u_\rho-u^*\|_V
\le C\epsilon_A .
\]
\end{corollary}
\begin{proof}
Let \(v_\rho\in\mathcal N_\rho\) be arbitrary. Since \(u_\rho\) minimizes the
empirical least-squares loss,
\[
\widehat{\mathcal L}_N(u_\rho)
\le
\widehat{\mathcal L}_N(v_\rho).
\]
Using
\[
\widehat{\mathcal L}_N(w_\rho)
=
\widehat{\mathcal L}_N^0(w_\rho-u^*),
\qquad w_\rho\in\mathcal N_\rho,
\]
we obtain
\[
\widehat{\mathcal L}_N^0(u_\rho-u^*)
\le
\widehat{\mathcal L}_N^0(v_\rho-u^*).
\]
By Assumption~\ref{ass:discrete-graph-equivalence},
\[
\mathcal L^0(u_\rho-u^*)
\le
c_{\rho,N}^{-1}
\widehat{\mathcal L}_N^0(u_\rho-u^*)
\le
c_{\rho,N}^{-1}
\widehat{\mathcal L}_N^0(v_\rho-u^*)
\le
\frac{C_{\rho,N}}{c_{\rho,N}}
\mathcal L^0(v_\rho-u^*).
\]
Using the continuous graph-norm equivalence,
\[
\|u_\rho-u^*\|_V^2
\le
C\mathcal L^0(u_\rho-u^*)
\le
C\frac{C_{\rho,N}}{c_{\rho,N}}
\mathcal L^0(v_\rho-u^*)
\le
C\frac{C_{\rho,N}}{c_{\rho,N}}
\|v_\rho-u^*\|_V^2 .
\]
Taking the infimum over \(v_\rho\in\mathcal N_\rho\) yields
\[
\|u_\rho-u^*\|_V
\le
C\left(\frac{C_{\rho,N}}{c_{\rho,N}}\right)^{1/2}
\inf_{v_\rho\in\mathcal N_\rho}\|v_\rho-u^*\|_V .
\]
The desired estimate follows from the approximation assumption.
\end{proof}

\subsubsection{Poisson subproblem}\label{poisson-raw}
We now consider the compatible Neumann Poisson subproblem. Owing to its elliptic structure, the corresponding graph-norm equivalence is simpler than that for the Nernst--Planck equation. Once this equivalence is established, the same residual-based approximation framework applies to compatible data.

For the theoretical estimate, the charge density in the Poisson subproblem \eqref{PNP-ite3} with the boundary condition in \eqref{PNP-ite-boundary} is assumed to satisfy the Neumann compatibility condition. In the implemented SO-RaNN algorithm, however, the charge density generated from the current concentration iterates may fail to satisfy this condition exactly. In that case, the computed potential is interpreted as a gauge-fixed Poisson least-squares fit, as explained in Remark~\ref{compatibility}, rather than as a classical compatible Neumann Poisson solution.

Consequently, the estimates in this subsection apply only to compatible Neumann Poisson data. If the computed charge density has a nonzero spatial mean, the corresponding Poisson least-squares fit is outside the scope of this compatible Poisson error estimate. The nonzero mean charge is then regarded as a compatibility defect of the computed concentrations and is used only in the interpretation of the reported potential and related diagnostic quantities.

Define
\[
V:=L^2(0,T;H^2(\Omega)),\qquad
Y:=L^2(0,T;L^2(\Omega)),\qquad
X:=L^2(0,T;H^{1/2}(\partial\Omega)).
\]
Since the homogeneous Neumann Poisson problem determines the potential only up to an additive spatial constant, we impose the gauge condition through an additional residual. Define
\[
\mathcal{G}\phi:=-\epsilon_p^2\Delta\phi,\qquad
\mathcal{B}\phi:=\partial_\mathbf{n}\phi,\qquad
\mathcal{M}\phi(t):=\frac{1}{|\Omega|^{1/2}}\int_\Omega \phi(x,t)\,dx .
\]
The population loss is then
\[
\mathcal{L}(\phi):=
\|\mathcal{G}\phi-f\|_Y^2+\|\mathcal{B}\phi-g\|_X^2+\|\mathcal{M}\phi\|_{L^2(0,T)}^2 .
\]

The corresponding empirical loss is defined as its discrete counterpart under the same norm structure. We now establish the graph-norm equivalence for the Poisson subproblem:
\begin{theorem}[Graph norm equivalence for Poisson equation]\label{graph-norm-p-thm}
    Assume that $\Omega$ is a bounded $C^{1,1}$ domain. There exist constants $C_L,\ C_U>0$ such that for all $\phi\in V$,
    \begin{equation}\label{graph-norm-p}
        C_L\Vert \phi\Vert^2_V\le \Vert\mathcal{G}\phi\Vert^2_Y+\Vert \mathcal{B}\phi \Vert^2_X + \Vert \mathcal{M}\phi \Vert^2_{L^2(0,T)} \le C_U \Vert\phi \Vert^2_V,
    \end{equation}
    where $C_U$ and $C_L$ depend on $\Omega$ and $\epsilon_p$; in particular, the lower-bound constant may deteriorate as $\epsilon_p\to0$.
\end{theorem}
\begin{proof}
    \textbf{Upper bound}. By the boundedness of the Laplacian from \(H^2(\Omega)\) to \(L^2(\Omega)\),
    \[
    \|\mathcal{G}\phi\|_Y^2
    =
    \epsilon_p^4\|\Delta\phi\|_{L^2(0,T;L^2(\Omega))}^2
    \le
    C\|\phi\|_{V}^2 .
    \]
    By the trace theorem,
    \[
    \|\mathcal{B}\phi\|_X^2
    =
    \|\partial_n\phi\|_{L^2(0,T;H^{1/2}(\partial\Omega))}^2
    \le
    C\|\phi\|_{V}^2 .
    \]
    Moreover, for a.e. \(t\in(0,T)\),
    \[
    |\mathcal{M}\phi(t)|
    =
    \left|\frac{1}{|\Omega|^{1/2}}\int_\Omega \phi(x,t)\,dx\right|
    \le
    \|\phi(t)\|_{L^2(\Omega)}
    \le
    \|\phi(t)\|_{H^2(\Omega)}.
    \]
    Integrating in time gives
    \[
    \|\mathcal{M}\phi\|_{L^2(0,T)}^2
    \le
    \|\phi\|_{V}^2 .
    \]
    Combining these three estimates yields the upper bound.\\
    \textbf{Lower bound}. Fix \(\phi\in V\). For a.e. \(t\in(0,T)\), write
    \[
    \phi(\cdot,t)=w(\cdot,t)+\bar\phi(t),
    \qquad
    \bar\phi(t):=\frac{1}{|\Omega|}\int_\Omega \phi(x,t)\,dx ,
    \]
    where \(w(\cdot,t)\) has zero spatial mean. Since
    \[
    \Delta w(\cdot,t)=\Delta\phi(\cdot,t),
    \qquad
    \partial_{\mathbf n}w(\cdot,t)=\partial_{\mathbf n}\phi(\cdot,t),
    \]
    the standard \(H^2\)-regularity estimate for the Neumann problem with zero mean
    gives
    \[
    \|w(\cdot,t)\|_{H^2(\Omega)}
    \le
    C_\Omega
    \left(
    \|\Delta\phi(\cdot,t)\|_{L^2(\Omega)}
    +
    \|\partial_{\mathbf n}\phi(\cdot,t)\|_{H^{1/2}(\partial\Omega)}
    \right).
    \]
    Moreover,
    \[
    \|\bar\phi(t)\|_{H^2(\Omega)}
    \le C_\Omega |\mathcal M\phi(t)|.
    \]
    Using \(\phi=w+\bar\phi\), integrating in time, and recalling that
    \(\mathcal G\phi=-\epsilon_p^2\Delta\phi\) and
    \(\mathcal B\phi=\partial_{\mathbf n}\phi\), we obtain
    \[
    \|\phi\|_V^2
    \le
    C
    \left(
    \|\mathcal G\phi\|_Y^2
    +
    \|\mathcal B\phi\|_X^2
    +
    \|\mathcal M\phi\|_{L^2(0,T)}^2
    \right).
    \]
    This proves the lower bound.
\end{proof}

For the Neumann Poisson subproblem, the data \((f,g)\) are assumed to satisfy
the compatibility condition
\[
\int_\Omega f(x,t)\,dx
+
\epsilon_p^2\int_{\partial\Omega}g(x,t)\,ds
=0
\qquad\text{for a.e. }t\in(0,T).
\]
In the PNP application with homogeneous Neumann data, \(g=0\) and
\(f=\sum_{i=1}^N z_i c_i\), so this condition reduces to
\[
\int_\Omega \sum_{i=1}^N z_i c_i(x,t)\,dx=0 .
\]
Under this condition, the zero-mean Neumann Poisson solution exists and is unique.

Once the graph-norm equivalence is available, the compatible Poisson subproblem can be treated in the same conditional residual-estimate framework as the linearized Nernst--Planck subproblem. The pointwise collocation loss used in the implementation is only a computable surrogate. Therefore, the graph-norm estimate below is stated under a discrete norm-equivalence assumption for compatible Poisson residuals, rather than as an unconditional consequence of the $L^2$-type statistical bound. This estimate should not be read as a convergence theorem for an incompatible gauge-fixed least-squares potential fit generated by a charge density with nonzero spatial mean.

For the error variable \(e=\phi-\phi^*\), define
\[
\mathcal L_P^0(e)
:=
\|\mathcal Ge\|_Y^2
+
\|\mathcal Be\|_X^2
+
\|\mathcal Me\|_{L^2(0,T)}^2 ,
\]
and let \(\widehat{\mathcal L}_{P,N}^0(e)\) be the corresponding homogeneous
empirical residual loss.

Let
\[
\mathcal E_\rho^P(\phi^*)
:=
\{\varphi_\rho-\phi^*: \varphi_\rho\in\mathcal N_\rho\}.
\]

\begin{corollary}[Conditional graph-norm estimate for compatible Neumann Poisson data]
Assume that the Neumann compatibility condition holds and that the zero-mean
exact solution \(\phi^*\in V\) exists. Let \(\phi_\rho\in\mathcal N_\rho\) be
the empirical collocation minimizer for the Poisson subproblem. Assume that there exist constants \(c_{\rho,N}^{P},C_{\rho,N}^{P}>0\) such
that
\[
c_{\rho,N}^{P}\mathcal L_P^0(e)
\le
\widehat{\mathcal L}_{P,N}^0(e)
\le
C_{\rho,N}^{P}\mathcal L_P^0(e)
\]
for all \(e\in\mathcal E_\rho^P(\phi^*)\). If
\[
\inf_{\varphi_\rho\in\mathcal N_\rho}
\|\varphi_\rho-\phi^*\|_V
\le \epsilon_A,
\]
then
\[
\|\phi_\rho-\phi^*\|_V
\le
C\left(\frac{C_{\rho,N}^{P}}{c_{\rho,N}^{P}}\right)^{1/2}
\epsilon_A .
\]
\end{corollary}

We next record a simple consequence of the global charge-neutrality compatibility condition. This estimate explains how the Neumann compatibility defect enters the raw least-squares approximation. In the continuous problem, the charge density is mean-free. In the raw RaNN discretization, however, the concentrations are obtained from least-squares subproblem solves and are not constrained to satisfy exact charge neutrality. The following estimate shows that the resulting raw charge defect is controlled by the concentration approximation errors.

\begin{proposition}[Raw charge defect and Neumann compatibility]
\label{prop:raw-charge-defect}
Let \(c_i\), \(i=1,\ldots,N\), be exact concentrations satisfying
\[
\sum_{i=1}^N z_i\int_\Omega c_i(x,t)\,dx=0
\quad \text{for a.e. } t\in(0,T).
\]
Let \(\bar c_{i,\rho}\) be the corresponding raw RaNN approximations, and define
\[
\mathcal Q_\rho(t)
:=
\int_\Omega \sum_{i=1}^N z_i \bar c_{i,\rho}(x,t)\,dx .
\]
Then, for a.e. \(t\in(0,T)\),
\[
|\mathcal Q_\rho(t)|
\le
\sum_{i=1}^N |z_i|
\|\bar c_{i,\rho}(\cdot,t)-c_i(\cdot,t)\|_{L^1(\Omega)} .
\]
In particular,
\[
|\mathcal Q_\rho(t)|
\le
|\Omega|^{1/2}
\sum_{i=1}^N |z_i|
\|\bar c_{i,\rho}(\cdot,t)-c_i(\cdot,t)\|_{L^2(\Omega)},
\]
and
\[
\|\mathcal Q_\rho\|_{L^2(0,T)}
\le
|\Omega|^{1/2}
\sum_{i=1}^N |z_i|
\|\bar c_{i,\rho}-c_i\|_{L^2(Q)} .
\]
\end{proposition}

\begin{proof}
Using the charge neutrality of the exact solution, we have
\[
\mathcal Q_\rho(t)
=
\sum_{i=1}^N z_i
\int_\Omega
\bigl(\bar c_{i,\rho}(x,t)-c_i(x,t)\bigr)\,dx .
\]
The first estimate follows from the triangle inequality. The second follows
from Hölder's inequality, and the \(L^2(0,T)\) estimate follows by integration
in time.
\end{proof}

\subsubsection{A local convergence result for the outer Picard iteration}\label{local-picard}
The residual-based estimates established in Sections \ref{np-raw} and \ref{poisson-raw} are local to the linearized subproblems obtained after freezing the coefficients from the previous iterate. We now show that, for the PNP system, these linearized subproblems induce a locally contractive outer Picard map at the PDE level. To keep the argument consistent with the analysis above, we state the result for the exact subproblem solves and then interpret the implemented raw RaNN iteration as an inexact perturbation of this map.

Let $p>d$ and define
\[
\mathcal W_T^p
:=
\left\{
\psi\in L^\infty(0,T;W^{2,p}(\Omega)):
\int_\Omega \psi(\mathbf{x},t)\,d\mathbf{x}=0
\ {\rm for\ a.e.}\ t\in(0,T)
\right\},
\]
with norm
\[
\|\psi\|_{\mathcal W_T^p}
:=
\|\psi\|_{L^\infty(0,T;W^{2,p}(\Omega))}.
\]
We also define
\[
\mathcal X_T^p
:=
L^\infty(0,T;L^p(\Omega))
\cap
L^2(0,T;W^{1,p}(\Omega)).
\]
For $R>0$, set
\[
\mathbb B_R^p
:=
\{\psi\in\mathcal W_T^p:
\|\psi\|_{\mathcal W_T^p}\le R\}.
\]
Since $p>d$, the embedding
\[
W^{2,p}(\Omega)\hookrightarrow W^{1,\infty}(\Omega)
\]
holds, and therefore the drift field $\nabla\psi$ is uniformly bounded for
$\psi\in\mathbb B_R^p$.
For a given \(\psi\in\mathbb B_R^p\), let \(S_i(\psi)\), \(i=1,2\), denote the
solution of the linearized Nernst--Planck subproblem with homogeneous no-flux
boundary condition and initial datum \(g_i\). Set
\[
M_i^0:=\int_\Omega g_i(x)\,dx,\qquad i=1,2,
\]
and assume the charge-neutral compatibility condition
\[
z_1M_1^0+z_2M_2^0=0.
\]
Define the admissible class
\[
\mathcal A_T^p
:=
\left\{
(c_1,c_2)\in(\mathcal X_T^p)^2:
z_1\int_\Omega c_1(x,t)\,dx
+
z_2\int_\Omega c_2(x,t)\,dx
=0
\ \text{for a.e. }t\in(0,T)
\right\}.
\]
Since the no-flux condition preserves the species masses,
\[
\int_\Omega S_i(\psi)(x,t)\,dx=M_i^0,
\]
the pair \((S_1(\psi),S_2(\psi))\) belongs to \(\mathcal A_T^p\). Hence the
zero-mean Neumann Poisson operator \(P\) is applicable, and the raw Picard map
\[
\mathcal T(\psi):=P(S_1(\psi),S_2(\psi))
\]
is well defined on \(\mathbb B_R^p\).

\begin{theorem}[Conditional local contraction of the raw Picard map]\label{outer-picard}
Let $\Omega\subset\mathbb R^d$ be a bounded domain with $C^{1,1}$ boundary,
and let $p>d$. Let $g_i$ be sufficiently regular so that the assumptions below
are meaningful. Assume that for every $\psi\in\mathbb B_R^p$ the following
properties hold.

(i) The linearized Nernst-Planck subproblem defining $S_i(\psi)$ is uniquely
solvable in $\mathcal{X}_T^p$, preserves the species masses,
\[
\int_\Omega S_i(\psi)(x,t)\,dx=M_i^0
\qquad\text{for a.e. }t\in(0,T),
\]
and satisfies
\[
\sum_{i=1}^2\|S_i(\psi)\|_{\mathcal X_T^p}
+
\sum_{i=1}^2\|S_i(\psi)\|_{L^\infty(Q)}
\le C_R.
\]

(ii) The solution map of the linearized Nernst-Planck subproblem is locally
Lipschitz in the sense that, for all
$\psi,\tilde\psi\in\mathbb B_R^p$,
\[
\sum_{i=1}^2
\|S_i(\psi)-S_i(\tilde\psi)\|_{L^\infty(0,T;L^p(\Omega))}
\le
C_R T^{1/2}
\|\psi-\tilde\psi\|_{\mathcal W_T^p}.
\]

(iii) The Neumann Poisson operator $P:\mathcal A_T^p\to\mathcal W_T^p$ satisfies the
$W^{2,p}$ estimate
\[
\|P(c_1,c_2)-P(\tilde c_1,\tilde c_2)\|_{\mathcal W_T^p}
\le
C_P
\sum_{i=1}^2
\|c_i-\tilde c_i\|_{L^\infty(0,T;L^p(\Omega))}
\]
for all $(c_1,c_2),(\tilde c_1,\tilde c_2)\in\mathcal A_T^p$.

(iv) For $0<T\le T_0$, the map $\mathcal T$ maps $\mathbb B_R^p$ into itself:
\[
\mathcal T(\mathbb B_R^p)\subset \mathbb B_R^p.
\]
Assume that \(C_R\) and \(C_P\) can be chosen uniformly for
\(0<T\le T_0\). Then there exists $T_*>0$, depending on
$R,\Omega,p,D_i,z_i,\epsilon_p$ and the initial data, such that for every
$0<T\le T_*$, the map $\mathcal T$ is a contraction on $\mathbb B_R^p$.
More precisely,
\[
\|\mathcal T(\psi)-\mathcal T(\tilde\psi)\|_{\mathcal W_T^p}
\le
q
\|\psi-\tilde\psi\|_{\mathcal W_T^p},
\qquad
q:=C_PC_RT^{1/2}<1.
\]
Consequently, $\mathcal T$ has a unique fixed point
$\psi^*\in\mathbb B_R^p$, and the Picard iteration
$\psi^{n+1}=\mathcal T(\psi^n)$ converges in $\mathcal W_T^p$ to $\psi^*$.
The associated concentration sequence \(c_i^{n+1}=S_i(\psi^n)\) converges to \(c_i^*=S_i(\psi^*)\) in
\(L^\infty(0,T;L^p(\Omega))\), with the convergence estimate following from assumption (ii).
\end{theorem}
\begin{proof}
For $\psi\in \mathbb B_R^p$, assumption (i) implies that
\[
\int_\Omega S_i(\psi)(x,t)\,dx=M_i^0.
\]
Hence, by $z_1M_1^0+z_2M_2^0=0$,
\[
(S_1(\psi),S_2(\psi))\in\mathcal A_T^p.
\]
Therefore $\mathcal T(\psi)=P(S_1(\psi),S_2(\psi))$ is well defined. Let $\psi,\tilde\psi\in\mathbb B_R^p$ and define
\[
c_i:=S_i(\psi),\qquad
\tilde c_i:=S_i(\tilde\psi),\qquad
e_i:=c_i-\tilde c_i,\quad i=1,2.
\]
By assumption (ii),
\[
\sum_{i=1}^2
\|e_i\|_{L^\infty(0,T;L^p(\Omega))}
\le
C_R T^{1/2}
\|\psi-\tilde\psi\|_{\mathcal W_T^p}.
\]
Let
\[
\Phi:=\mathcal T(\psi)-\mathcal T(\tilde\psi).
\]
Then $\Phi$ satisfies
\[
-\epsilon_p^2\Delta\Phi=z_1e_1+z_2e_2,
\]
with homogeneous Neumann boundary condition and the zero-mean constraint.
By assumption (iii),
\[
\|\Phi\|_{\mathcal W_T^p}
\le
C_P
\sum_{i=1}^2
\|e_i\|_{L^\infty(0,T;L^p(\Omega))}
\le
C_PC_R T^{1/2}
\|\psi-\tilde\psi\|_{\mathcal W_T^p}.
\]
Choose $T_*>0$ such that
\[
q:=C_PC_R T_*^{1/2}<1.
\]
Then for every $0<T\le T_*$,
\[
\|\mathcal T(\psi)-\mathcal T(\tilde\psi)\|_{\mathcal W_T^p}
\le
q\|\psi-\tilde\psi\|_{\mathcal W_T^p}.
\]
Thus $\mathcal T$ is a contraction on $\mathbb B_R^p$. Since
$\mathcal T(\mathbb B_R^p)\subset\mathbb B_R^p$ by assumption (iv), the
Banach fixed-point theorem yields the unique fixed point and the convergence
of the Picard iteration.
\end{proof}

The admissibility condition above is imposed at the continuous PDE level in order to make the homogeneous Neumann Poisson operator well defined. The raw Picard map analyzed in Theorem~\ref{outer-picard} is therefore an exact-subproblem map on the charge-neutral admissible class. In the implemented RaNN method, the Nernst--Planck and Poisson subproblems are solved in a least-squares sense. Hence the no-flux boundary condition and the species
mass identities are satisfied only approximately, and the numerical charge neutrality is not expected to hold exactly. The implemented iteration should therefore be interpreted as an inexact perturbation of the admissible
continuous Picard map, as in Corollary~\ref{outer-picard-rho}.

\begin{corollary}\label{outer-picard-rho}
Under the assumptions of Theorem~\ref{outer-picard}, let
$\mathcal T_\rho:\mathbb B_R^p\to\mathbb B_R^p$ denote the inexact Picard map
induced by the implemented raw RaNN solvers. Assume $\psi_\rho^0\in\mathbb B_R^p$, and there exists
$\eta_\rho\ge 0$ such that
\[
\|\mathcal T_\rho(\psi)-\mathcal T(\psi)\|_{\mathcal W_T^p}
\le \eta_\rho,
\qquad \forall \psi\in\mathbb B_R^p.
\]
Let $\psi^*\in\mathbb B_R^p$ be the fixed point of $\mathcal T$, and let $\psi_\rho^{n+1}=\mathcal T_\rho(\psi_\rho^n)$.
Then, for all $n\ge 0$,
\[
\|\psi_\rho^{n+1}-\psi^*\|_{\mathcal W_T^p}
\le
q\|\psi_\rho^n-\psi^*\|_{\mathcal W_T^p}
+\eta_\rho,
\]
where $q\in(0,1)$ is the contraction constant in Theorem~\ref{outer-picard}.
Consequently,
\[
\|\psi_\rho^n-\psi^*\|_{\mathcal W_T^p}
\le
q^n\|\psi_\rho^0-\psi^*\|_{\mathcal W_T^p}
+
\frac{\eta_\rho}{1-q},
\qquad n\ge 0.
\]
In particular, the inexact Picard iteration converges to an
$O(\eta_\rho)$-neighborhood of the exact fixed point.
\end{corollary}

\begin{proof}
By the triangle inequality and the fixed-point property
$\mathcal T(\psi^*)=\psi^*$, we have
\[
\begin{aligned}
\|\psi_\rho^{n+1}-\psi^*\|_{\mathcal W_T^p}
&=
\|\mathcal T_\rho(\psi_\rho^n)-\mathcal T(\psi^*)\|_{\mathcal W_T^p} \\
&\le
\|\mathcal T_\rho(\psi_\rho^n)-\mathcal T(\psi_\rho^n)\|_{\mathcal W_T^p}
+
\|\mathcal T(\psi_\rho^n)-\mathcal T(\psi^*)\|_{\mathcal W_T^p} \\
&\le
\eta_\rho
+
q\|\psi_\rho^n-\psi^*\|_{\mathcal W_T^p}.
\end{aligned}
\]
Iterating this inequality gives
\[
\|\psi_\rho^n-\psi^*\|_{\mathcal W_T^p}
\le
q^n\|\psi_\rho^0-\psi^*\|_{\mathcal W_T^p}
+
\sum_{k=0}^{n-1}q^k\eta_\rho
\le
q^n\|\psi_\rho^0-\psi^*\|_{\mathcal W_T^p}
+
\frac{\eta_\rho}{1-q}.
\]
This proves the result.
\end{proof}

The above corollary applies only to the raw RaNN Picard iteration. The positivity cut-off, the mass correction, the SAV correction, and the final Poisson LS update are post-processing or value-level correction steps and are analyzed separately below.

\subsection{Stability of the structure-oriented correction steps}\label{post-processing}
In this subsection, we analyze the structure-oriented correction and least-squares fitting steps used in the proposed method: the positivity cut-off, the discrete-time mass correction, the SAV auxiliary-variable correction, and the final potential update. The purpose is not to establish a full convergence theory for the fully post-processed output, but to show that each operation is stable in the sense relevant to the implementation and preserves the intended structure at the appropriate level.

\subsubsection{Positivity cut-off}
We begin with the positivity-enforcing cut-off applied to the concentration outputs. Since this operation acts pointwise on the value field, it is natural to analyze it in terms of value-error stability rather than in terms of the full $V$-norm.

\begin{proposition}[Value-level stability of the positivity cut-off]
\label{cut-off}
Let $Q=\Omega\times I$ and let
\[
P_\delta(s):=\max\{s,\delta\}, \qquad \delta>0 .
\]
For a raw approximation $u_\rho$, define the value-level cut-off
\[
u_\rho^+ := P_\delta(u_\rho).
\]
Assume that the exact concentration satisfies
\[
u^*\ge 0 \qquad \text{a.e. in } Q.
\]
Then, pointwise a.e. in $Q$, one has the following alternative:
\[
\text{either}\quad |u_\rho^+-u_\rho|\le 2\delta,
\qquad
\text{or}\quad
|u_\rho^+-u^*|\le |u_\rho-u^*|.
\]
More precisely,
\[
|u_\rho^+-u^*|
\le
|u_\rho-u^*|
+
\delta\,\mathbf 1_{\{-\delta\le u_\rho<\delta\}} .
\]
Consequently,
\[
\|u_\rho^+-u^*\|_{L^2(Q)}
\le
\|u_\rho-u^*\|_{L^2(Q)}
+
\delta\,|\{-\delta\le u_\rho<\delta\}|^{1/2},
\]
and in particular,
\[
\|u_\rho^+-u^*\|_{L^2(Q)}
\le
\|u_\rho-u^*\|_{L^2(Q)}
+
\delta |Q|^{1/2}.
\]
\end{proposition}

\begin{proof}
Fix a point \((x,t)\in Q\), and write \(s=u_\rho(x,t)\) and \(c=u^*(x,t)\).
Since \(c\ge0\), the scalar cut-off \(P_\delta(s)=\max\{s,\delta\}\) satisfies
\[
|P_\delta(s)-c|
\le
|s-c|
+
\delta\,\mathbf 1_{\{-\delta\le s<\delta\}} .
\]
Indeed, if \(s\ge\delta\), equality holds. If \(-\delta\le s<\delta\), then \(P_\delta(s)=\delta\). If \(c\ge \delta\),
then
\[
|\delta-c|=c-\delta\le c-s=|s-c|.
\]
If \(0\le c<\delta\), then
\[
|\delta-c|=\delta-c\le \delta\le |s-c|+\delta .
\]
Thus the stated bound holds in the transition region. If \(s<-\delta\), then \(P_\delta(s)=\delta\), and
because \(c\ge0\), one has
\[
|\delta-c|\le |s-c|.
\]
Returning to \(s=u_\rho(x,t)\) and \(c=u^*(x,t)\) gives the pointwise estimate.
Taking the \(L^2(Q)\)-norm and using the triangle inequality yields
\[
\|u_\rho^+-u^*\|_{L^2(Q)}
\le
\|u_\rho-u^*\|_{L^2(Q)}
+
\delta |\{-\delta\le u_\rho<\delta\}|^{1/2}.
\]
The final estimate follows from
\[
|\{-\delta\le u_\rho<\delta\}|\le |Q|.
\]
\end{proof}

The cut-off map \(P_\delta\) is non-smooth, and the residual derivatives in the implementation are still evaluated from the raw network output. Therefore, Proposition \ref{cut-off} should be interpreted only as a value-level stability statement, not as a full \(V\)-norm residual estimate for \(u_\rho^+\). It shows that the cut-off introduces at most an \(O(\delta)\) perturbation near zero, while strongly negative values are moved toward the physically admissible region without increasing the pointwise error with respect to the nonnegative exact
concentration.

Combined with Proposition \ref{prop:raw-charge-defect}, the following result shows that the charge defect after cut-off consists of the raw least-squares compatibility defect plus a controlled value-level cut-off perturbation.

\begin{proposition}[Charge deviation induced by the value-level cut-off]
\label{prop:charge-cutoff}
Fix a time \(t\in[0,T]\), and suppress the time variable in the notation. Let $\bar c_i$ be the raw RaNN approximation of the $i$-th concentration,
and let
\[
c_i^+ = \sigma_\delta(\bar c_i):=\max\{\bar c_i,\delta\},
\qquad \delta>0 .
\]
Assume that the exact concentration satisfies $c_i\ge 0$ a.e. in $\Omega$.
Then, pointwise in $\Omega$, the following alternatives hold:
\[
\text{either}\quad |c_i^+ - \bar c_i|\le 2\delta,
\qquad
\text{or}\quad |c_i^+ - c_i|\le |\bar c_i-c_i|.
\]
More precisely,
\[
|c_i^+ - c_i|
\le
|\bar c_i-c_i|
+\delta\,\mathbf 1_{\{0\le \bar c_i<\delta\}}
+2\delta\,\mathbf 1_{\{-\delta\le \bar c_i<0\}} .
\]
Consequently, if the exact concentrations satisfy the charge-neutrality condition
\[
\sum_{i=1}^N z_i\int_\Omega c_i\,dx=0,
\]
then the total charge of the cut-off concentrations satisfies
\[
\left|
\sum_{i=1}^N z_i\int_\Omega c_i^+\,dx
\right|
\le
\sum_{i=1}^N |z_i|\|\bar c_i-c_i\|_{L^1(\Omega)}
+
\sum_{i=1}^N |z_i|
\left(
\delta |\{0\le \bar c_i<\delta\}|
+
2\delta |\{-\delta\le \bar c_i<0\}|
\right).
\]
In particular,
\[
\left|
\sum_{i=1}^N z_i\int_\Omega c_i^+\,dx
\right|
\le
\sum_{i=1}^N |z_i|\|\bar c_i-c_i\|_{L^1(\Omega)}
+
2\delta |\Omega|\sum_{i=1}^N |z_i|.
\]
\end{proposition}

\begin{proof}
Applying the scalar estimate from Proposition~\ref{cut-off} pointwise to each
species gives
\[
|c_i^+ - c_i|
\le
|\bar c_i-c_i|
+\delta\,\mathbf 1_{\{0\le \bar c_i<\delta\}}
+2\delta\,\mathbf 1_{\{-\delta\le \bar c_i<0\}} .
\]
Using the charge neutrality of the exact solution,
\[
\sum_{i=1}^N z_i\int_\Omega c_i\,dx=0,
\]
we obtain
\[
\left|
\sum_{i=1}^N z_i\int_\Omega c_i^+\,dx
\right|
=
\left|
\sum_{i=1}^N z_i\int_\Omega (c_i^+-c_i)\,dx
\right|
\le
\sum_{i=1}^N |z_i|\int_\Omega |c_i^+-c_i|\,dx .
\]
Substituting the pointwise bound gives the first estimate. The simplified bound
follows from
\[
|\{0\le \bar c_i<\delta\}|+|\{-\delta\le \bar c_i<0\}|
\le |\Omega|.
\]
\end{proof}

Proposition~\ref{prop:charge-cutoff} shows that the additional charge defect introduced by the positivity
cut-off is controlled by the size of the cut-off perturbation. More precisely,
if
\[
\mathcal Q_\rho(t)
=
\int_\Omega \sum_{i=1}^N z_i\bar c_{i,\rho}(x,t)\,dx,
\qquad
\mathcal Q_\rho^+(t)
=
\int_\Omega \sum_{i=1}^N z_i c_{i,\rho}^+(x,t)\,dx,
\]
then
\[
|\mathcal Q_\rho^+(t)|
\le
|\mathcal Q_\rho(t)|
+
|\Omega|^{1/2}
\sum_{i=1}^N |z_i|
\|c_{i,\rho}^+(\cdot,t)-\bar c_{i,\rho}(\cdot,t)\|_{L^2(\Omega)} .
\]
Hence, if the raw RaNN approximation has a small charge-neutrality defect
and the cut-off perturbation is small, then the cut-off concentration also
has a small charge-neutrality defect.

\subsubsection{Discrete-time mass correction}
We next consider the mass correction step. In the algorithm, the scaling factors are computed at selected correction instants and then interpolated in time. Therefore, the natural property to prove is exact mass matching at those correction instants, together with a stability estimate for the corrected value field.

\begin{proposition}[Exact mass matching at the correction instants]\label{mass-cor}
    Let $t_0<t_1<\cdots<t_J$ be the selected correction instants. Assume that $u_\rho^+\ge \delta>0$ a.e. in $\Omega\times I$. Define
    \[ 
    M^j_\rho:=\int_\Omega u^+_\rho(x,t_j)\,dx,
    \qquad
    M^j:=\int_\Omega u^*(x,t_j)\,dx=\int_\Omega u^*(x,0)\,dx.
    \]
    Then $M_\rho^j\ge \delta |\Omega|>0$, and hence the scaling factor
    \[ 
    \gamma^j:=\frac{M^j}{M^j_\rho}
    \]
    is well defined. Let $\gamma(t)$ be an interpolation of $\{\gamma^j\}_{j=0}^J$, and define
    \[ 
    \tilde u_\rho(x,t):=\gamma(t)u^+_\rho(x,t).
    \]
    Then
    \[ 
    \int_\Omega \tilde u_\rho(x,t_j)\,dx=\int_\Omega u^*(x,t_j)\,dx,
    \qquad j=0,\ldots,J.
    \]
\end{proposition}
\begin{proof}
The lower bound $M_\rho^j\ge \delta|\Omega|$ follows from the cut-off, so $\gamma^j$ is well defined. By construction,
\[
\int_\Omega \tilde u_\rho(x,t_j)\,dx
=
\gamma^j\int_\Omega u_\rho^+(x,t_j)\,dx
=
\frac{M^j}{M_\rho^j}M_\rho^j
=
M^j.
\]
\end{proof}

If $\gamma \in L^{\infty}(0,T)$, then the corrected value field satisfies an $L^2$-stability estimate of the form
$$\Vert \tilde u_\rho-u^*\Vert_{L^2(Q)}\le \Vert \gamma\Vert_{L^{\infty}(0,T)} \Vert u^+_\rho-u^*\Vert_{L^2(Q)}+\Vert\gamma-1\Vert_{L^{\infty}(0,T)}\Vert u^*\Vert_{L^2(Q)}.$$
Since $\gamma(t)$ is obtained by interpolation from the discrete values $\gamma^j$, the construction enforces exact mass matching only at the selected correction instants, while between two adjacent correction instants the mass is only approximately preserved. The same argument applies to the final mass correction after the additional Nernst-Planck solve driven by the SAV-scaled potential.

\subsubsection{Monotonicity of the SAV auxiliary variable}
In this subsection, we analyze the ideal SAV update defined by the discrete auxiliary-variable relation. Any additional truncation or reset used in implementation for numerical robustness is not included in the proposition below.

\begin{proposition}[Monotonicity of the SAV auxiliary variable]
Let $\mathcal E^{j+1}$ and $\mathcal D^{j+1}$ denote the discrete energy
and dissipation quantities evaluated from the fields used in the ideal SAV
update, and assume that $\mathcal D^{j+1}\ge0$ and
$\mathcal E^{j+1}+C_0>0$. Let the auxiliary variable $R^j$ be updated by
\[
\frac{R^{j+1}-R^j}{\Delta t_j}
=
-\frac{R^{j+1}}{\mathcal E^{j+1}+C_0}\mathcal D^{j+1}.
\]
If $R^j\ge0$, then
\[
R^{j+1}
=
\frac{R^j}
{1+\Delta t_j\mathcal D^{j+1}/(\mathcal E^{j+1}+C_0)}
\le R^j.
\]
Consequently, the SAV auxiliary sequence is nonincreasing.
\end{proposition}

This proposition concerns only the ideal SAV auxiliary update. It does not imply monotonicity of the reported physical free energy based on the final potential unless an additional consistency estimate between the auxiliary variable and the reported free-energy functional is available. It also does not cover implementation-level safeguards on \(\xi^j\) or the final Poisson update used in the reported free-energy evaluation.

\begin{proposition}[Value-field deviation induced by post-processing]
    Let $\tilde{c}_{i,\rho}=\gamma_i(t)c^+_{i,\rho}$ and $\tilde{\phi}_{\rho}=\xi(t)\bar{\phi}_{\rho}$. If $\gamma_i,\ \xi\in L^{\infty}(0,T)$, then 
    $$\Vert \tilde{c}_{i,\rho}-\bar{c}_{i,\rho}\Vert_{L^2(Q)}\le \Vert \gamma_i\Vert_{L^{\infty}}\Vert c^+_{i,\rho}-\bar{c}_{i,\rho}\Vert_{L^2(Q)}+\Vert \gamma_i-1\Vert_{L^{\infty}}\Vert \bar{c}_{i,\rho}\Vert_{L^2(Q)},$$
    and
    $$\Vert \tilde{\phi}_{\rho}-\bar{\phi}_{\rho}\Vert_{L^2(Q)}\le \Vert \xi-1\Vert_{L^{\infty}}\Vert \bar{\phi}_{\rho}\Vert_{L^2(Q)}.$$
\end{proposition}
\begin{proof}
    For the concentration variable,
    \[
    \tilde c_{i,\rho}-\bar c_{i,\rho}
    =
    \gamma_i(t)(c_{i,\rho}^+-\bar c_{i,\rho})
    +
    (\gamma_i(t)-1)\bar c_{i,\rho}.
    \]
    Taking the \(L^2(Q)\)-norm and using the triangle inequality gives the first estimate. The estimate for the potential follows directly from
    \[
    \tilde\phi_\rho-\bar\phi_\rho=(\xi(t)-1)\bar\phi_\rho.
    \]
\end{proof}

\subsection{Divergence-free approximation and Oseen-type analysis for the PNP-NS system}\label{pnpns-error}
In this subsection, we analyze the velocity approximation and the corresponding linearized Navier–Stokes subproblem appearing in the PNP-NS system. In contrast to the previous subsection, the key additional structure here is incompressibility. Since the SP-RaNN construction satisfies the divergence-free constraint pointwise by construction, we restrict the analysis to a divergence-free trial space. This avoids the mismatch between the numerical ansatz and the abstract function space, and is also consistent with the structure-preserving design of the algorithm.

\subsubsection{Approximation property of the SP-RaNN}
We first show that the divergence-free randomized neural network is capable of approximating divergence-free velocity fields. The key idea is to represent the target velocity through a stream function in two dimensions or a vector potential in three dimensions, and then approximate the potential by a standard RaNN. The divergence-free approximation is then recovered by differentiation.

\begin{theorem}\label{sp-rann-app}
Assume that the target velocity field is spatially divergence-free and admits a sufficiently regular potential representation. The existence of such a representation depends on the topology of $\Omega$ and on the relevant boundary-flux compatibility conditions; here it is imposed as an explicit assumption. In two dimensions, assume that there exists a stream function \(v(\cdot,t)\) such that
\[
\mathbf u(\cdot,t)=\nabla^\perp v(\cdot,t),
\]
with \(v\in H^s(Q)\). In three dimensions, assume that there exists a vector
potential \(\mathbf v(\cdot,t)\) such that
\[
\mathbf u(\cdot,t)=\nabla\times \mathbf v(\cdot,t),
\]
with \(\mathbf v\in H^s(Q)^3\).
Then there exists an SP-RaNN approximation \(\mathbf u_\rho\) such that
\[
\nabla\cdot \mathbf u_\rho=0
\quad\text{pointwise in }Q,
\]
and
\[
\|\mathbf u-\mathbf u_\rho\|_{H^{s-1}(Q)}
\le
C_{\rm sp}
\inf_{\mathbf w_\rho\in\mathcal N_\rho^{\rm pot}}
\|\mathbf w-\mathbf w_\rho\|_{H^s(Q)} .
\]
Here $\mathbf{w}=v$ in two dimensions, interpreted as a scalar potential,
and $\mathbf{w}=\mathbf{v}$ in three dimensions, interpreted as a vector
potential. The space $\mathcal{N}^{\rm pot}_\rho$ denotes the corresponding
scalar- or vector-potential RaNN space. The constant $C_{\rm sp}>0$ is the
operator norm of the spatial map $v\mapsto\nabla^\perp v$ in two dimensions or
$\mathbf v\mapsto\nabla\times\mathbf v$ in three dimensions from $H^s(Q)$
to $H^{s-1}(Q)$; in particular, it depends only on $s$, the dimension, the
domain and the norm convention, but is independent of the network width, the
random hidden-layer parameters and the output-layer coefficients. In particular, under the approximation theorem of \cite{Dang2024}, $\mathbf{u}_\rho$ approximates $\mathbf{u}$ in $H^{s-1}(Q)$ with high probability as the network width increases.
\end{theorem}
\begin{proof}
    The estimates below use the boundedness of spatial first-order derivatives from \(H^s(\Omega\times I)\) to \(H^{s-1}(\Omega\times I)\).

    \textbf{The two-dimensional case}. By the assumed potential representation, $\mathbf{u}=\nabla^{\perp} v$. From Theorem 4.12 in \cite{Dang2024}, choose $v_\rho=\sum_{i=1}^m \alpha_i\psi_i\in\mathcal N_\rho^{\rm pot}$, we have $\Vert v_\rho-v\Vert_{H^s(Q)}\le \epsilon$. Set $\mathbf u_\rho=\nabla^\perp v_\rho$, we have
\[
\Vert \mathbf{u}-\mathbf{u}_\rho \Vert_{H^{s-1}(Q)}
=
\Vert \nabla^\perp(v-v_\rho)\Vert_{H^{s-1}(Q)}
\le
C_{\rm sp}\Vert v-v_\rho\Vert_{H^{s}(Q)}.
\]

    \textbf{The three-dimensional case}. By the assumed potential representation, $\mathbf{u}=\nabla \times \mathbf{v}$. For $\mathbf{v}=(v_1,\ v_2,\ v_3)$, we can use $v_{\rho,1}=\sum_{i=1}^m \alpha^1_i \psi^1_i,\ v_{\rho,2}=\sum_{i=1}^m \alpha^2_i \psi^2_i,\ v_{\rho,3}=\sum_{i=1}^m \alpha^3_i \psi^3_i$ to approximate $v_1,\ v_2,\ v_3$ separately. It is easy to construct $\mathbf{v}_{\rho}=\sum_{i=1}^{3m}\alpha_i \bm{\psi}_i$ to approximate $\mathbf{v}$ directly, where $\bm{\psi}=(\psi^1,\psi^2,\psi^3)^T$. Let 
    $$\alpha=(\alpha^1_1,...,\alpha^1_m,\alpha^2_1,...,\alpha^2_m,...,\alpha^3_1,...,\alpha^3_m)^T,$$
    $$\psi^1=(\psi^1_1,...,\psi^1_m,\underbrace{0,...,0}_{2m})^T,\ \psi^2=(\underbrace{0,...,0}_{m},\psi^2_1,...,\psi^2_m,\underbrace{0,...,0}_{m})^T,\ \psi^3=(\underbrace{0,...,0}_{2m},\psi^3_1,...,\psi^3_m)^T.$$
    It is clear that the RaNN with vector-valued activation function $\mathbf{v}_{\rho}$ can approximate $\mathbf{v}$, and the approximation ability depends on the approximation ability of the scalar-valued network (Theorem 4.12 in \cite{Dang2024}). Set $\mathbf u_\rho=\nabla\times\mathbf v_\rho$, we have
\[
\Vert \mathbf{u}-\mathbf{u}_\rho \Vert_{H^{s-1}(Q)}
=
\Vert \nabla\times(\mathbf v-\mathbf v_\rho)\Vert_{H^{s-1}(Q)}
\le
C_{\rm sp}\Vert \mathbf v-\mathbf v_\rho\Vert_{H^{s}(Q)}.
\]
Taking the infimum over the corresponding scalar- or vector-potential RaNN space gives the stated estimate.
\end{proof}

The theorem above is stated in terms of potentials only for the purpose of analysis. In the actual implementation, the SP-RaNN basis is used directly, so that the velocity approximation satisfies the spatial divergence-free constraint pointwise by construction (See \cite{Li2026}).

\subsubsection{Oseen-type analysis in a divergence-free trial space}
We next consider the velocity-pressure subproblem in the linearized PNP-NS iteration. Since the SP-RaNN construction enforces the divergence-free constraint pointwise, we formulate the analysis in a divergence-free trial space. On the other hand, the boundary condition is not built into the neural ansatz and is therefore imposed through the residual-based loss rather than as a hard constraint in the function space.

For simplicity, and consistently with the PNP--NS model considered here, we state the conditional Oseen estimate for homogeneous no-slip boundary data.

Define the velocity and pressure spaces
$$V_u = \left\{\mathbf{u}\in L^2(I;H^2(\Omega)^d)\cap H^1(I;L^2(\Omega)^d): \nabla \cdot \mathbf{u}=0 \right\},$$
$$V_p:=L^2(I;H^1(\Omega)),\quad \mathcal{M}_p p(t):=\frac{1}{|\Omega|^{1/2}}\int_\Omega p(x,t)\,dx.$$
We equip $V_u\times V_p$ with the norm $\Vert (u,p)\Vert_{V_u\times V_p}^2:=\Vert\partial_t u\Vert_{L^2(I;L^2)}^2+\Vert u\Vert_{L^2(I;H^2)}^2+\Vert \nabla p\Vert_{L^2(I;L^2)}^2 + \|\mathcal{M}_p p\|_{L^2(I)}^2$.

Let
\[
\mathcal G_O(\mathbf u,p)
:=
\partial_t\mathbf u
+
(\mathbf u^n\cdot\nabla)\mathbf u
-
\nu\Delta\mathbf u
+
\nabla p
-
\mathbf F,
\]
where \(\mathbf F\) denotes the prescribed forcing term in the linearized
velocity equation. Let
\[
\mathcal B_0\mathbf u:=\mathbf u(\cdot,0)-\mathbf u_0,
\qquad
\mathcal B_1\mathbf u:=\mathbf u|_{\partial\Omega},
\]
and
\[
\mathcal M_p p(t)
:=
|\Omega|^{-1/2}\int_\Omega p(x,t)\,dx .
\]
The corresponding continuous residual loss is
\[
\mathcal L_O(\mathbf u,p)
:=
\|\mathcal G_O(\mathbf u,p)\|_Y^2
+
\|\mathcal B_0\mathbf u\|_Z^2
+
\|\mathcal B_1\mathbf u\|_X^2
+
\|\mathcal M_p p\|_{L^2(I)}^2 .
\]
where, for the Oseen problem, we use the residual spaces
\[
Y:=L^2(I;L^2(\Omega)^d),\qquad
Z:=H^1(\Omega)^d,
\]
and
\[
X:=
L^2(I;H^{3/2}(\partial\Omega)^d)
\cap
H^{3/4}(I;L^2(\partial\Omega)^d).
\]
For an error pair \((\mathbf e,r)\), where $\mathbf{e}$ is $\mathbf{u}_\rho-\mathbf{u}^*$, $r$ is $p_\rho-p^*$, define the homogeneous Oseen residual loss
\[
\mathcal L_O^0(\mathbf e,r)
:=
\|\partial_t\mathbf e
+
(\mathbf u^n\cdot\nabla)\mathbf e
-
\nu\Delta\mathbf e
+
\nabla r\|_Y^2
+
\|\mathbf e(\cdot,0)\|_Z^2
+
\|\mathbf e|_{\partial\Omega}\|_X^2
+
\|\mathcal M_p r\|_{L^2(I)}^2 .
\]

\begin{proposition}[Conditional ideal residual estimate for the divergence-free Oseen subproblem]
\label{ns-app}
Let $\mathbf u^n$ be a given divergence-free convection field with sufficient regularity, and let \((\mathbf u^*,p^*)\in V_u\times V_p\) be the exact gauge-fixed solution of the linearized Oseen-type subproblem. Assume that the homogeneous Oseen residual satisfies the graph-norm equivalence
\[
C_L\|(\mathbf e,r)\|_{V_u\times V_p}^2
\le
\mathcal L_O^0(\mathbf e,r)
\le
C_U\|(\mathbf e,r)\|_{V_u\times V_p}^2 .
\]
for all \((\mathbf e,r)\in V_u\times V_p\). Let
\[
(\mathbf u_\rho^{\rm id},p_\rho^{\rm id})
\in
\arg\min_{(\mathbf v_\rho,q_\rho)\in
V_\rho^{\rm div}\times Q_\rho}
\mathcal L_O(\mathbf v_\rho,q_\rho)
\]
be an ideal minimizer of the continuous residual loss over the
divergence-free SP-RaNN velocity space and the RaNN pressure space. If
\[
\inf_{(\mathbf v_\rho,q_\rho)\in
V_\rho^{\rm div}\times Q_\rho}
\|(\mathbf v_\rho-\mathbf u^*,q_\rho-p^*)\|_{V_u\times V_p}
\le
\epsilon_A,
\]
then
\[
\|(\mathbf u_\rho^{\rm id}-\mathbf u^*,
p_\rho^{\rm id}-p^*)\|_{V_u\times V_p}
\le
C\epsilon_A .
\]
\end{proposition}

\begin{proof}
Let \((\mathbf v_\rho,q_\rho)\in
V_\rho^{\rm div}\times Q_\rho\) satisfy
\[
\|(\mathbf v_\rho-\mathbf u^*,q_\rho-p^*)\|_{V_u\times V_p}
\le 2\epsilon_A .
\]
Since \((\mathbf u^*,p^*)\) solves the exact linearized Oseen problem, for any
\((\mathbf w,q)\),
the data residual of \((\mathbf w,q)\) equals the homogeneous residual of
\((\mathbf w-\mathbf u^*,q-p^*)\). Hence the ideal minimizer satisfies
\[
\mathcal L_O^0(\mathbf u_\rho^{\rm id}-\mathbf u^*,
p_\rho^{\rm id}-p^*)
\le
\mathcal L_O^0(\mathbf v_\rho-\mathbf u^*,q_\rho-p^*) .
\]
Using the graph-norm equivalence,
\[
\|(\mathbf u_\rho^{\rm id}-\mathbf u^*,
p_\rho^{\rm id}-p^*)\|_{V_u\times V_p}^2
\le
C\mathcal L_O^0(\mathbf u_\rho^{\rm id}-\mathbf u^*,
p_\rho^{\rm id}-p^*)
\]
\[
\le
C\mathcal L_O^0(\mathbf v_\rho-\mathbf u^*,q_\rho-p^*)
\le
C\|(\mathbf v_\rho-\mathbf u^*,q_\rho-p^*)\|_{V_u\times V_p}^2 .
\]
Taking the infimum over
\((\mathbf v_\rho,q_\rho)\in V_\rho^{\rm div}\times Q_\rho\)
gives the result.
\end{proof}

The estimate above is an ideal continuous-residual statement for the linearized Oseen-type subproblem. In the implementation, the continuous residual loss is replaced by a pointwise collocation least-squares loss, so a rigorous passage to the full Oseen graph norm would require additional sampling, quadrature, or
inverse-stability estimates. Thus, the Oseen analysis should be understood as a conditional, stepwise estimate for the linearized velocity-pressure solve. The SP-RaNN ansatz enforces \(\nabla\cdot\mathbf u_\rho=0\) pointwise by
construction, while the no-slip boundary condition and pressure normalization are imposed through residual terms. A full convergence analysis of the coupled PNP--NS outer iteration is left for future work.

\begin{remark}
    The graph-norm equivalence assumed in Proposition \ref{ns-app} is a regularity assumption for the underlying linearized Oseen-type PDE, rather than a consequence of the SP-RaNN ansatz. The SP-RaNN construction is used only to enforce the divergence-free constraint at the representation level. Related regularity results for Stokes and Oseen systems can be found, for example, in Raymond~\cite{Raymond2007}, Shibata--Shimada~\cite{Shibata2007}, and related work of Kajiwara~\cite{Kajiwara2022}. A verification of the specific space-time residual graph-norm equivalence used here is not pursued in this paper.
\end{remark}

\section{Numerical Examples}
In this section, we apply the proposed SO-RaNN method to the PNP and PNP-NS systems. The numerical tests are divided into two groups. Source-driven manufactured examples are used primarily to assess approximation accuracy against known exact solutions. Homogeneous examples, which are more closely aligned with the intrinsic structure of the original PNP/PNP-NS systems, are used to illustrate the intended structure-oriented behavior, including value-level positivity correction, mass matching at selected correction instants, computed physical free-energy curves based on the final potential, and divergence-free approximation. In the present experiments, the positivity cut-off, mass correction, and SAV-based correction are designed as a unified correction framework. However, in the source-driven Examples \ref{pnp-accuracy} and \ref{pnpns-accuracy}, we only apply the positivity and mass corrections to the concentration variables, while no SAV-based correction is imposed on the electric potential, since the corresponding source terms alter the original energy law. Accordingly, these examples are used mainly for accuracy assessment, while the reported positivity and mass-matching behaviors should be understood as partial correction effects rather than illustrations of the full correction mechanism. All algorithms are implemented in Python, and the least-squares systems are solved using \texttt{np.linalg.lstsq}. Numerical experiments are carried out on a workstation equipped with a 13th Gen Intel(R) Core(TM) i7-13700KF CPU at 3.40 GHz and 32 GB of RAM.

In the PNP experiments (Examples \ref{pnp-accuracy}–\ref{pnp-benchmark3}), we use three RaNNs to approximate $c_1,\ c_2$ and $\phi$, with initialization ranges $r_1,\ r_2$ and $r_3$, respectively. In the PNP-NS experiments (Examples \ref{pnpns-accuracy}–\ref{pnpns-benchmark1}), we use four RaNNs to approximate $c_1,\ c_2,\ \phi$ and $p$, with initialization ranges $r_1,\ r_2,\ r_3$ and $r_4$, and one SP-RaNN to approximate $\mathbf{u}$ with initialization range $r_5$. The time-dependent scaling functions $\gamma_i(t)$ are obtained by linear interpolation from the discrete values $\{\gamma_i^j\}$. We select the activation function $\rho(x)=e^{-0.5x^2}$, and choose $\delta=10^{-10}$ for the positivity cut-off map $\sigma_{\delta}$, $\eta=10^{-6},\ \eta_{\phi}=0.3$ and $M_{\rm ite}=30$ as the threshold values and maximum number of iterations in Algorithm \ref{PNP-algorithm} and \ref{PNPNS-algorithm}. The number of basis functions of RaNNs and SP-RaNNs is denoted by $m$, the numbers of collocation points are determined from the measurement-ratio parameter $n_e$ according to Remark \ref{points}, and $20^{d_{\rm st}}$ ($d_{\rm st}$ includes dimensions of time and space) Gauss-Legendre quadrature points are used to calculate numerical integration. For reproducibility, we fix \texttt{torch.manual\_seed(42)} and \texttt{np.random.seed(42)}.

The weights and biases in the hidden layers are generated randomly from the uniform distributions $\mathcal{U}(-r,r)$. For RaNN methods, the initialization strategy determines the basis functions, and thus has a significant impact on the accuracy. Dang and Wang (\cite{Dang2024}) proposed the frequency-based parameter initialization strategy for $r$ of weights, and adjust bias according to $\Omega\times I$. For brevity, we only apply the initialization strategy for bias to concentrate more basis functions within $\Omega\times I$. Specifically, after initializing weight matrix $W\in \mathbb{R}^{m\times d}$, we generate a random matrix $B\in \mathbb{R}^{m\times d}$, then let $\mathbf{b}=-(W \odot B)\cdot \mathbf{1}_{d\times 1}$, where $\odot$ denotes element-wise multiplication. When the time-block strategy is used, the initialization is performed separately within each time block.

\subsection{PNP system}

\begin{example}[Two-component PNP Equations, accuracy test]\label{pnp-accuracy}
In this example, we use the following augmented equations with exact solutions as the test problem
\begin{subequations}\label{PNP-f}
    \begin{align}
        \frac{\partial c_i}{\partial t} &= D_i \nabla \cdot(\nabla c_i + z_i c_i \nabla \phi) + f_i,\ i=1,2,\ {\rm in}\ \Omega\times I, \\
        -\epsilon_p^2 \Delta \phi &= \sum\limits_{i=1}^{N}z_i c_i + g,\ {\rm in}\ \Omega\times I,
    \end{align}
\end{subequations}
where $\Omega\times I=[-1,1]^2\times [0,T]$, $D_1=D_2=1$, $z_1=1,\ z_2=-1$, $\epsilon_p=1$, exact solutions are as follows
\begin{equation}
    \left\{
    \begin{array}{rrll}
        c_1 &=& {\rm sin}(\pi x){\rm sin}(\pi y){\rm sin}(t) + 1.1,  \\
        c_2 &=& -{\rm sin}(\pi x){\rm sin}(\pi y){\rm sin}(t) + 1.1, \\
        \phi &=& \frac{1}{\pi^2}{\rm sin}(\pi x){\rm sin}(\pi y){\rm sin}(t).
    \end{array}
    \right.
\end{equation}
The source terms $f_1,\ f_2,\ g$, boundary and initial conditions can be calculated with the exact solution.
\end{example}

In this source-driven manufactured test, the mass-scaling target at each correction instant is taken as the exact mass computed from the manufactured solution, rather than the initial mass.

First, we set $T=1$, choose $n_e=10,\ r_1=r_2=r_3=2,\ \lambda=100$, the results for different $m$ are summarized in the Table \ref{pnp-ex1-table}, where $e$ is the $L^2$ error at $t=1$, $t$ (CPU) denotes the required CPU time (seconds). Table \ref{pnp-ex1-table} shows that the approximation error decreases as the network width increases in the reported setting. For this source-driven test, the concentration correction steps do not deteriorate the observed approximation accuracy and, in some of the reported configurations, lead to slightly smaller errors. Since the source term alters the original energy law, no SAV-based correction is imposed on the electric potential in this example.

To assess the computational efficiency of the proposed RaNN solver in this manufactured setting, we also implemented a classical finite-difference solver for the same manufactured source--boundary problem. The finite-difference scheme uses second-order central differences in space, a BDF1 start-up followed by a BDF2 time-stepping procedure, explicit extrapolation for the drift term, and a sparse direct solver for the manufactured Poisson equation with exact-solution-induced Neumann boundary data and a mean-zero gauge constraint. The concentration and potential boundary data in this finite-difference comparison are taken from the same manufactured exact solution as in the RaNN least-squares formulation. For a representative run with $h=1.25\times 10^{-2},\ \Delta t=2.5\times 10^{-4}$, the finite-difference method yields $e_{c_1}=e_{c_2}=1.36\times10^{-4}$ and $e_\phi=5.81\times10^{-5}$, with a total CPU time of $37.715$ seconds. By comparison, the SO-RaNN method with $m=800$ in Table \ref{pnp-ex1-table} requires a smaller CPU time ($12.89$ seconds), while achieving smaller errors. This representative comparison suggests that, for this test problem and this particular implementation setting, the SO-RaNN method can be computationally competitive.

Furthermore, we consider the long-time simulation. Set $T=50$, as stated in the Remark \ref{long-time}, we divide $I=[0,50]$ into $I_l=[l-1,l],\ l=1,...,50$, then repeat the algorithm \ref{PNP-algorithm}. The results are shown in Figure \ref{pnp-accuracy-p3} (RaNN) and Figure \ref{pnp-accuracy-p4} (SO-RaNN). Interestingly, compared to Figure \ref{pnp-accuracy-p3}, the results in Figure \ref{pnp-accuracy-p4} remain stable over the tested horizon. The results suggest improved observed robustness over the tested horizon for this manufactured test.

\begin{table}[H]
\centering
\begin{tabular}{cccccccc}
\toprule
Method&$m$&$e_{c_1}$&$e_{c_2}$&$e_{\phi}$&$t$ (CPU)&ite\\
\midrule
\multirow{4}{*}{RaNN}&200 &2.41E-02	&2.79E-02	&1.04E-02	&1.99   &10\\
&400 &1.00E-03	&1.10E-03	&5.41E-04	&4.19  &10\\
&800 &1.11E-05	&1.59E-05	    &1.26E-05	&12.67  &10\\
&1600 &3.47E-07	&8.34E-06	    &2.25E-07	&34.90  &10\\
\midrule
\multirow{4}{*}{SO-RaNN}&200 &2.39E-02	&2.02E-02	&1.04E-02	&2.11	&10\\
&400 &9.92E-04	&7.66E-04	    &5.41E-04	&4.31	&10\\
&800 &7.54E-06	    &1.58E-05	        &1.27E-05	&12.89	&10\\
&1600 &3.02E-07	    &8.34E-07	        &2.26E-07	&35.34	&10\\
\bottomrule
\end{tabular}
\caption{$L^2$ errors at $t=1$ for RaNN and SO-RaNN method with different $m$, where $\ n_e=10,\ r_1=r_2=r_3=2,\ \lambda=100$ in Example \ref{pnp-accuracy}.}
\label{pnp-ex1-table}
\end{table}

\begin{figure}[!htbp] 		
	\centering
	\includegraphics[scale=0.5]{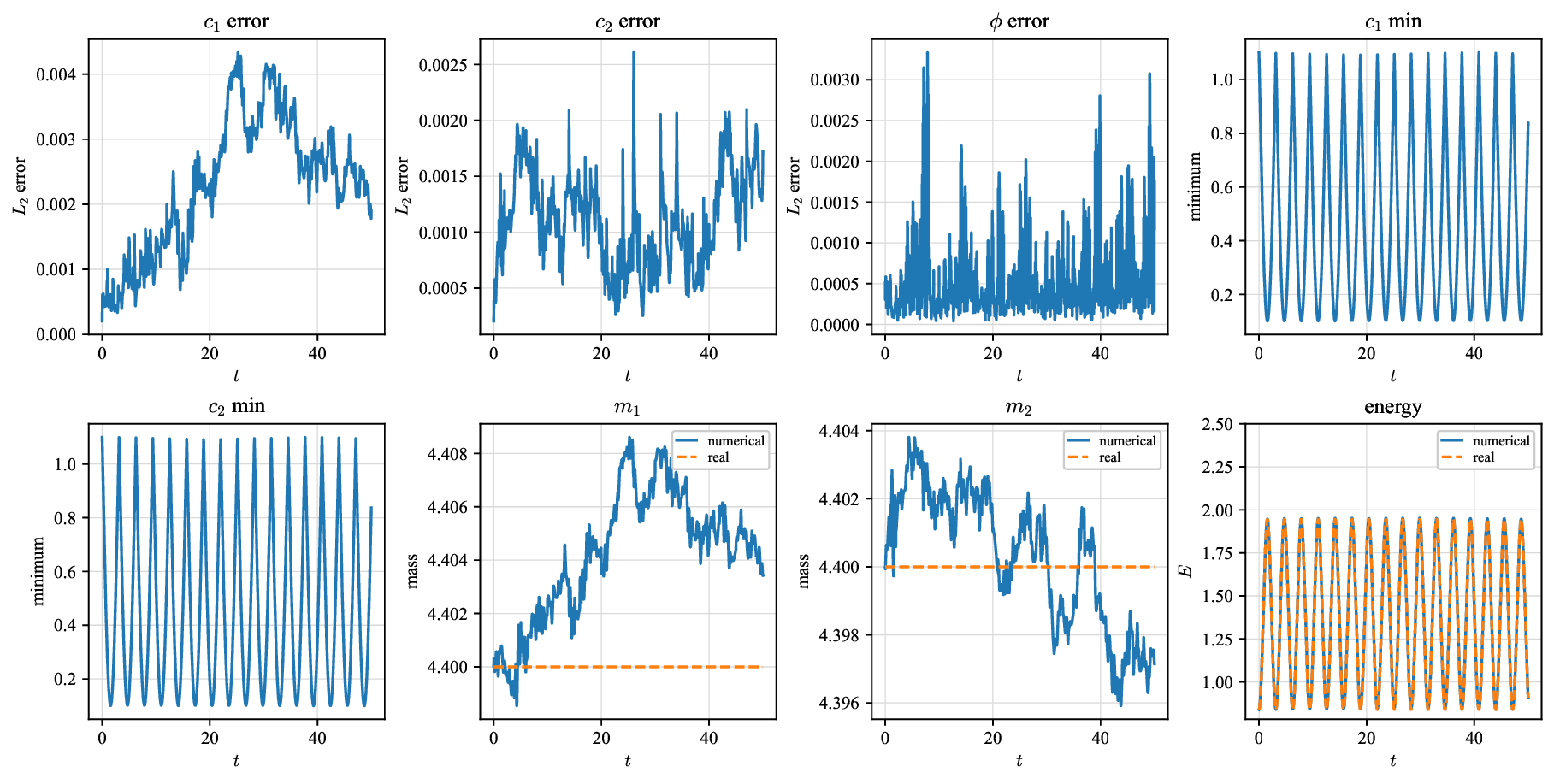}
	\caption{For the RaNN method, we show $L^2$ errors for $c_1,\ c_2,\ \phi$ against the manufactured exact solution, the minimum values and masses of ions, and the free-energy functional for this source-driven augmented test, where $\ m=400,\ \lambda=100,\ n_e=10,\ r_1=r_2=r_3=2$ in Example \ref{pnp-accuracy}.}
	\label{pnp-accuracy-p3}
\end{figure}

\begin{figure}[!htbp] 		
	\centering
	\includegraphics[scale=0.5]{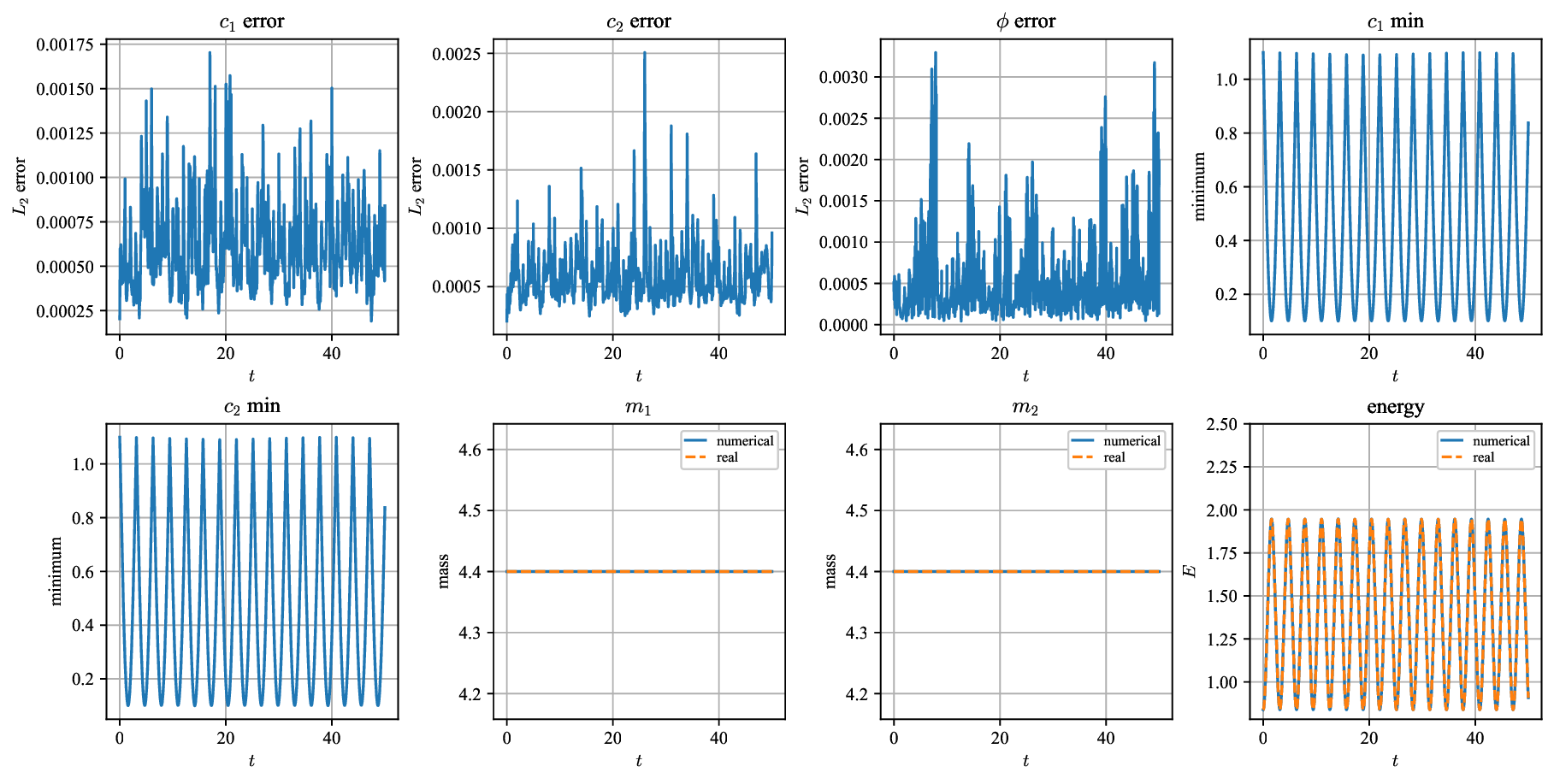}
	\caption{For the SO-RaNN method, we show $L^2$ errors for $c_1,\ c_2,\ \phi$ against the manufactured exact solution, the minimum values and masses of ions, and the free-energy functional for this source-driven augmented test, where $\ m=400,\ \lambda=100,\ n_e=10,\ r_1=r_2=r_3=2$ in Example \ref{pnp-accuracy}.}
	\label{pnp-accuracy-p4}
\end{figure}

\begin{example}[Two-component PNP Equations, benchmark test]\label{pnp-benchmark1}
In this example, we consider the equations \eqref{PNP-model-2} in the domain $\Omega\times I=[0,1]^3$, $D_1=D_2=1$, $z_1=1,\ z_2=-1$, the initial conditions are given by
\begin{equation}
    \left\{
    \begin{array}{rrll}
        c_1(x,y,0) &=& -4(x^2(1-x)^2+y^2(1-y)^2)+\frac{19}{15},  \\
        c_2(x,y,0) &=& -15(x^4(1-x)^4+y^4(1-y)^4)+\frac{22}{21}.
    \end{array}
    \right.
\end{equation}
\end{example}

We test our method for different $\epsilon_p$. First, we set $m=400,\ n_e=20,\ r_1=r_2=r_3=1,\ \lambda=100$, for $\epsilon^2_p=1$, the results of RaNN and SO-RaNN are shown in Figure \ref{pnp-benchmark1-p1} (RaNN) and Figure \ref{pnp-benchmark1-p2} (SO-RaNN). The comparison between Figures \ref{pnp-benchmark1-p1} and \ref{pnp-benchmark1-p2} shows that the correction steps improve positivity and mass behavior in this test. We then change $\epsilon_p$, and divide $I=[0,1]$ into $I_l,\ l=1,2,...,10$, set $m=400,\ n_e=20,\ r_1=r_2=r_3=1,\ \lambda=100$, the results are plotted in Figure \ref{pnp-benchmark1-p3} ($\epsilon^2_p=0.16$), Figure \ref{pnp-benchmark1-p4} ($\epsilon^2_p=0.09$), and Figure \ref{pnp-benchmark1-p5} ($\epsilon^2_p=100$). These results illustrate the intended value-level positivity correction and selected-time mass matching. The free-energy curves, computed from the final numerical solutions, are observed to decay for several values of \(\epsilon_p\).

\begin{figure}[!htbp] 		
	\centering
	\includegraphics[scale=0.6]{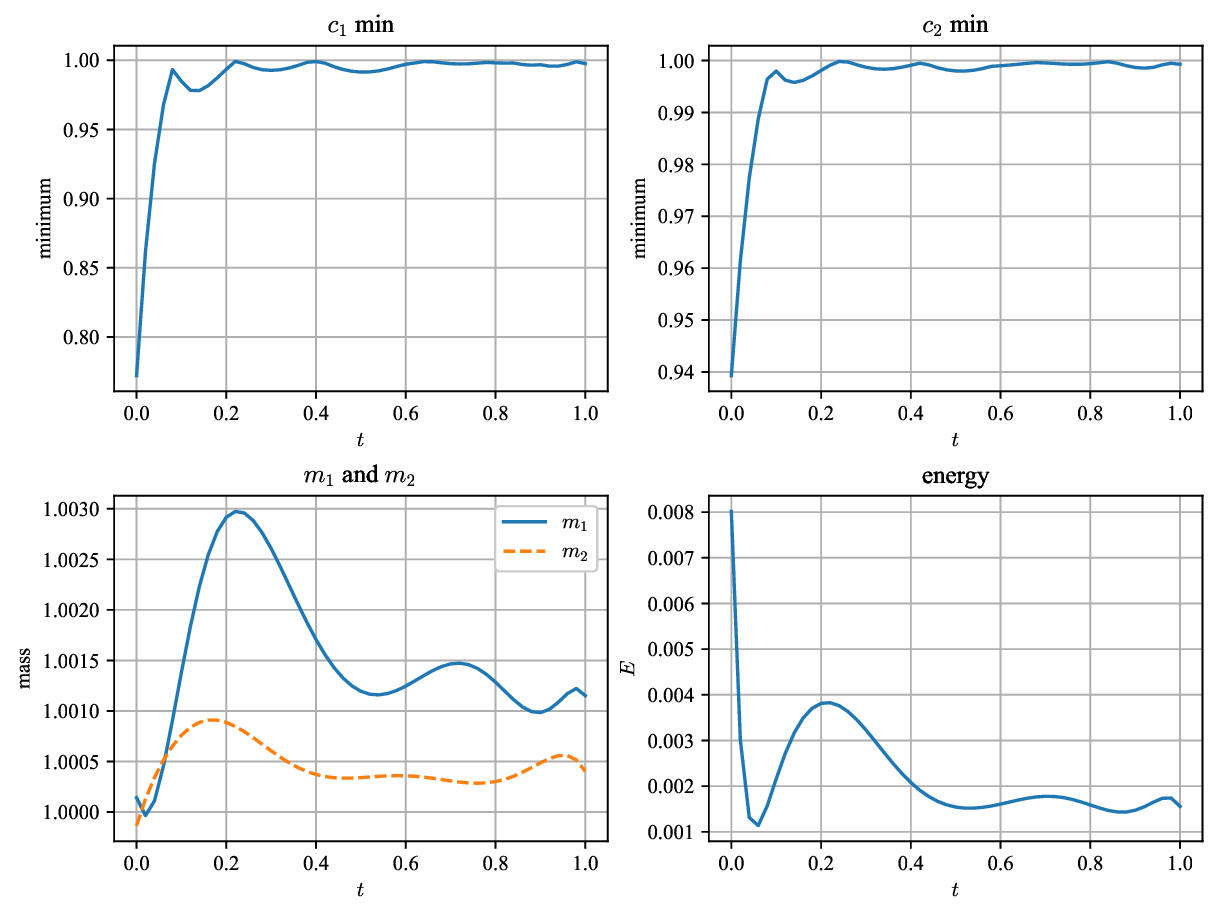}
	\caption{For $\epsilon_p^2=1$, the figure reports the computed free-energy functional, the minimum values, and the masses of \(c_1,c_2\) obtained by the SO-RaNN method, where $\ m=400,\ n_e=20,\ r_1=r_2=r_3=1,\ \lambda=100$ in Example \ref{pnp-benchmark1}.}
	\label{pnp-benchmark1-p1}
\end{figure}

\begin{figure}[!htbp] 		
	\centering
	\includegraphics[scale=0.6]{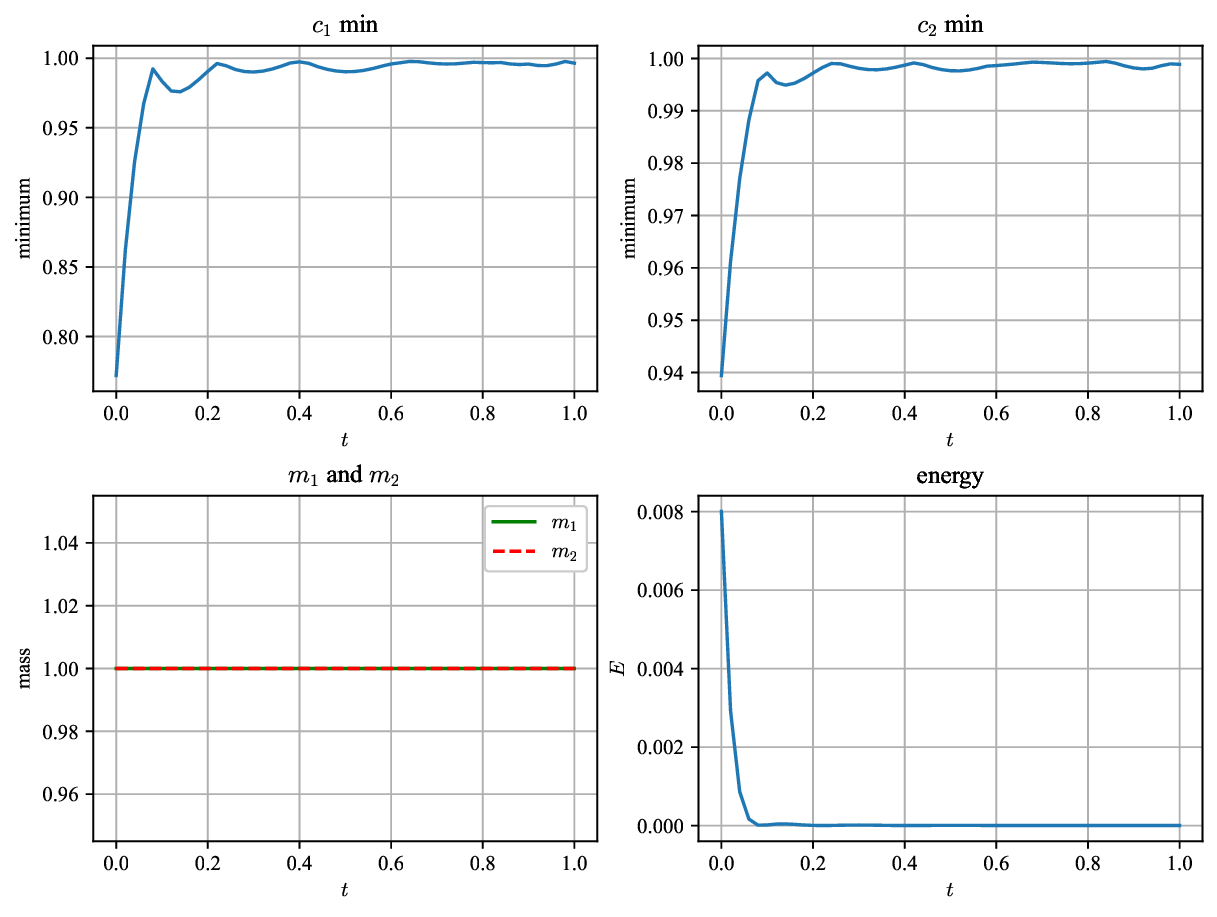}
	\caption{For $\epsilon_p^2=1$, the figure reports the computed free-energy functional, the minimum values, and the masses of \(c_1,c_2\) obtained by the SO-RaNN method, where $\ m=400,\ n_e=20,\ r_1=r_2=r_3=1,\ \lambda=100$ in Example \ref{pnp-benchmark1}.}
	\label{pnp-benchmark1-p2}
\end{figure}

\begin{figure}[!htbp] 		
	\centering
	\includegraphics[scale=0.6]{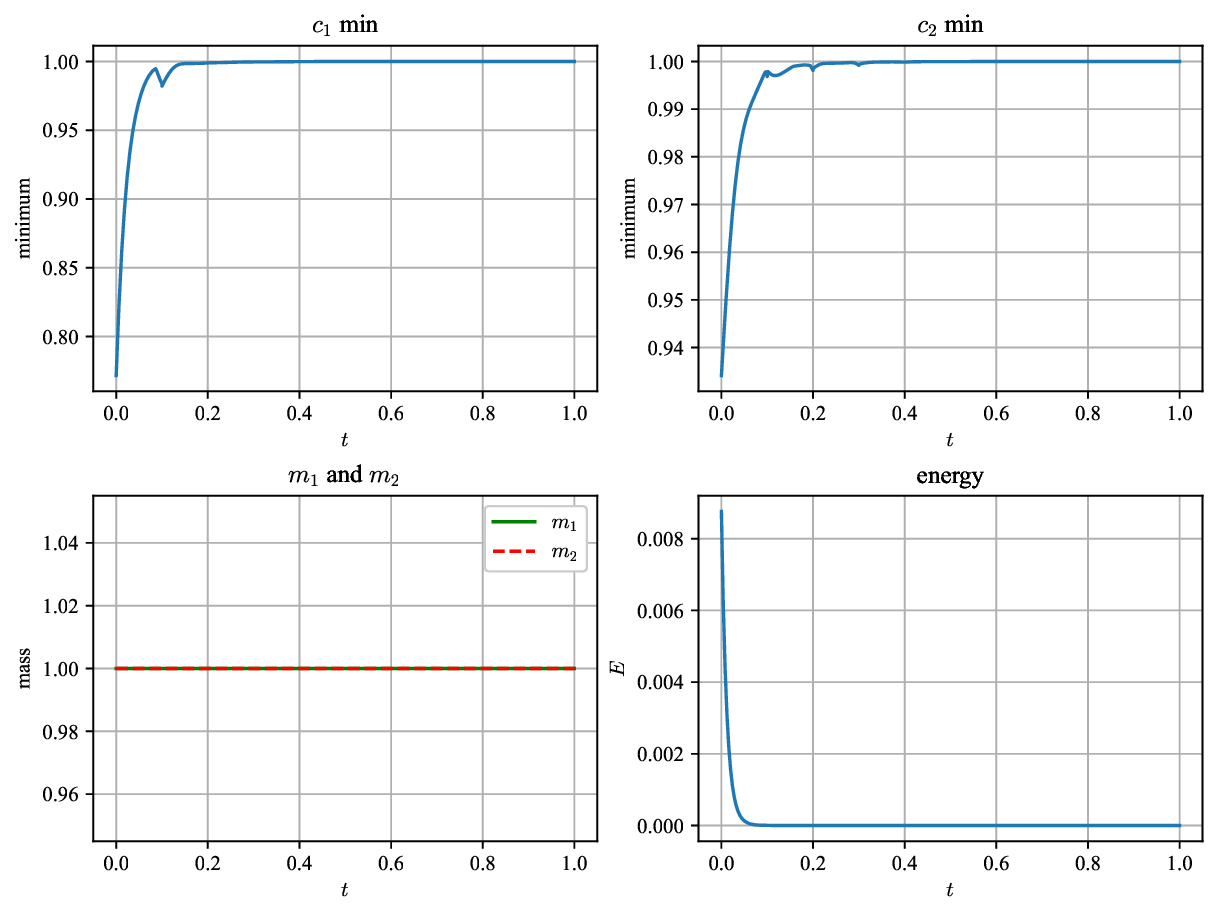}
	\caption{For $\epsilon_p^2=0.16$, the figure reports the computed free-energy functional, the minimum values, and the masses of \(c_1,c_2\) obtained by the SO-RaNN method (with time blocks), where $\ m=400,\ n_e=20,\ r_1=r_2=r_3=1,\ \lambda=100$ in Example \ref{pnp-benchmark1}.}
	\label{pnp-benchmark1-p3}
\end{figure}

\begin{figure}[!htbp] 		
	\centering
	\includegraphics[scale=0.6]{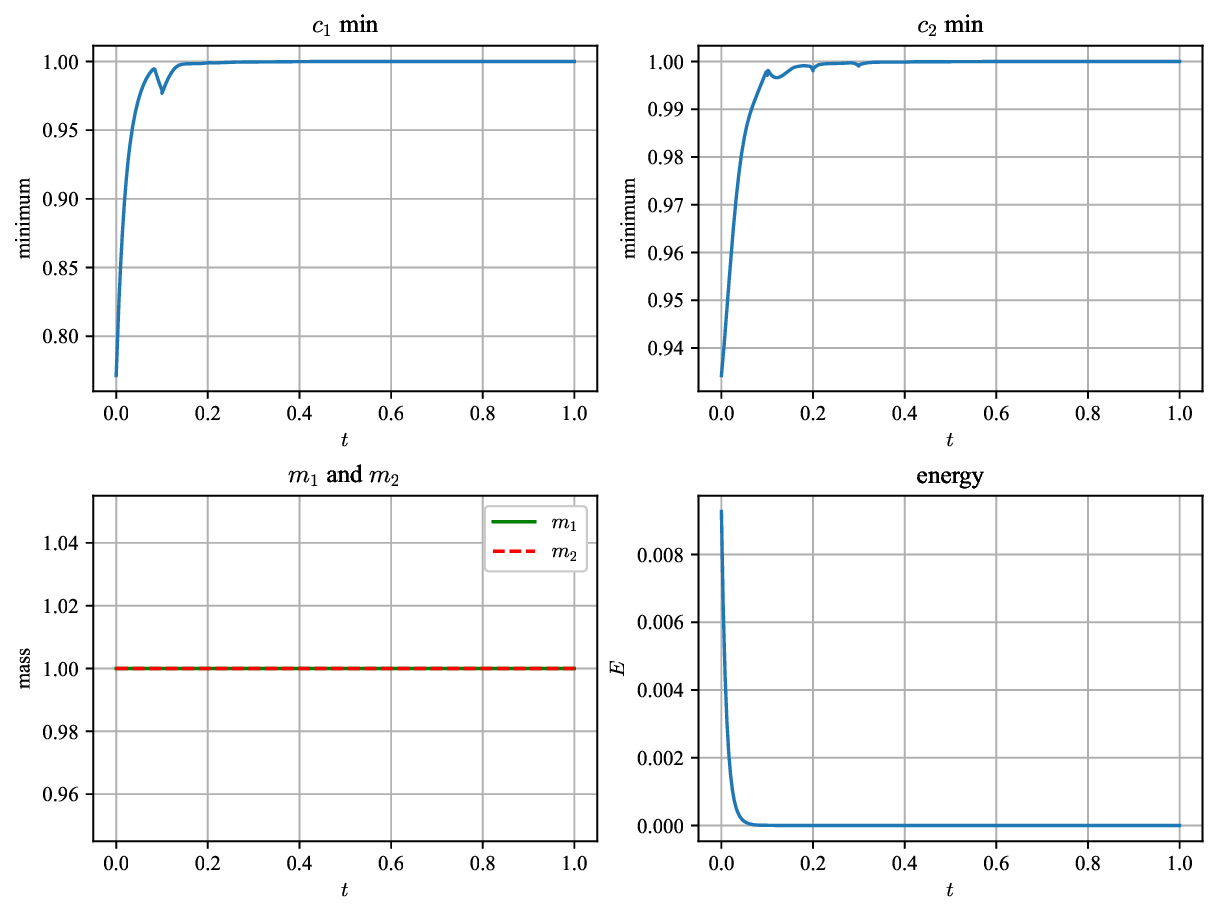}
	\caption{For $\epsilon_p^2=0.09$, the figure reports the computed free-energy functional, the minimum values, and the masses of \(c_1,c_2\) obtained by the SO-RaNN method (with time blocks), where $\ m=400,\ n_e=20,\ r_1=r_2=r_3=1,\ \lambda=100$ in Example \ref{pnp-benchmark1}.}
	\label{pnp-benchmark1-p4}
\end{figure}

\begin{figure}[!htbp] 		
	\centering
	\includegraphics[scale=0.6]{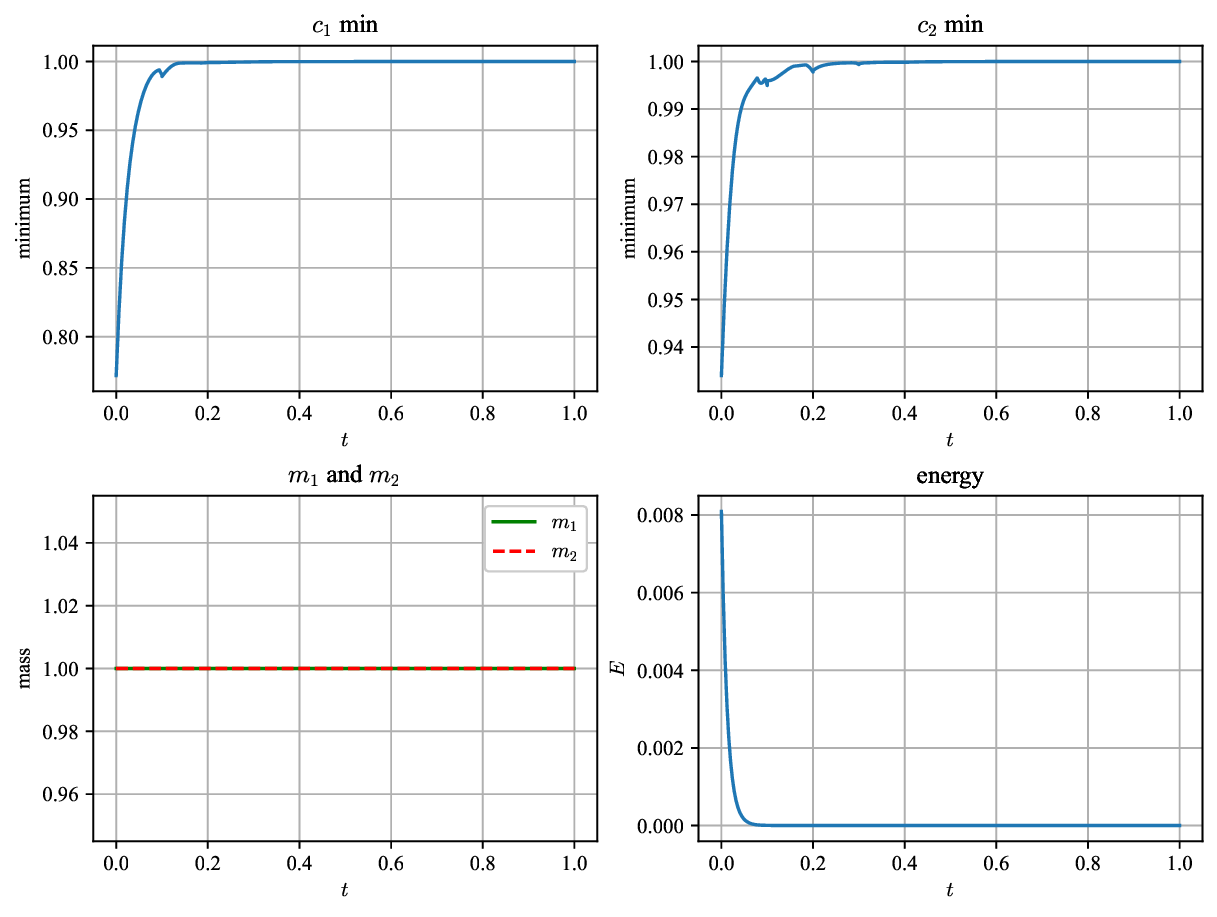}
	\caption{For $\epsilon_p^2=100$, the figure reports the computed free-energy functional, the minimum values, and the masses of \(c_1,c_2\) obtained by the SO-RaNN method (with time blocks), where $\ m=400,\ n_e=20,\ r_1=r_2=r_3=1,\ \lambda=100$ in Example \ref{pnp-benchmark1}.}
	\label{pnp-benchmark1-p5}
\end{figure}

\begin{example}[Two-component PNP Equations, benchmark test, discontinuous initial value]\label{pnp-benchmark2}
In this example, we consider the equations \eqref{PNP-model-2} in the domain $\Omega\times I=[-2,2]^2\times[0,1]$, $D_1=D_2=1$, $z_1=1,\ z_2=-1$, $\epsilon_p=1$, and use the same initial conditions as \cite{Tong2024,Bonilla2025}:
\[
c_1(x,y,0) = \left\{
\begin{array}{lll}
1 & \text{if} & \left(x-\frac{1}{2}\right)^2 + \left(y-\frac{1}{2}\right)^2 \leq \frac{1}{4}, \\[6pt]
\frac{1}{2} & \text{if} & \left(x+\frac{1}{2}\right)^2 + \left(y+\frac{1}{2}\right)^2 \leq \frac{1}{4}, \\[6pt]
0 & & \text{otherwise}.
\end{array}
\right.
\]
\[
c_2(x,y,0) = \left\{
\begin{array}{lll}
\frac{1}{2} & \text{if} & \left(x-\frac{1}{2}\right)^2 + \left(y-\frac{1}{2}\right)^2 \leq \frac{1}{4}, \\[6pt]
1 & \text{if} & \left(x+\frac{1}{2}\right)^2 + \left(y+\frac{1}{2}\right)^2 \leq \frac{1}{4}, \\[6pt]
0 & & \text{otherwise}.
\end{array}
\right.
\]
\end{example}

The initial condition is discontinuous, which is challenging for NN-based methods. We divide $I=[0,1]$ into 10 blocks uniformly, and set $m=800,\ n_e=20,\ r_1=r_2=r_3=2$. Figure \ref{pnp-benchmark2-p1} depicts the evolution of $c_1,\ c_2,\ \phi$ along with time ($t=0,\ 0.04,\ 0.2,\ 1$), the results in Figure \ref{pnp-benchmark2-p1} are consistent with the results in \cite{Bonilla2025} (subsection 6.1). Figure~\ref{pnp-benchmark2-p2} shows that the final numerical solution remains positive and that the prescribed masses are recovered at the selected correction instants. The computed free-energy functional, evaluated from the final corrected concentrations and the final gauge-fixed potential, is observed to be nonincreasing in this test. We also observe that the minima of $c_1$ and $c_2$ are close to zero, which further motivates enforcing positivity at the value level when evaluating logarithmic free-energy-related quantities.

\begin{figure}[!htbp] 		
	\centering
	\includegraphics[scale=0.8]{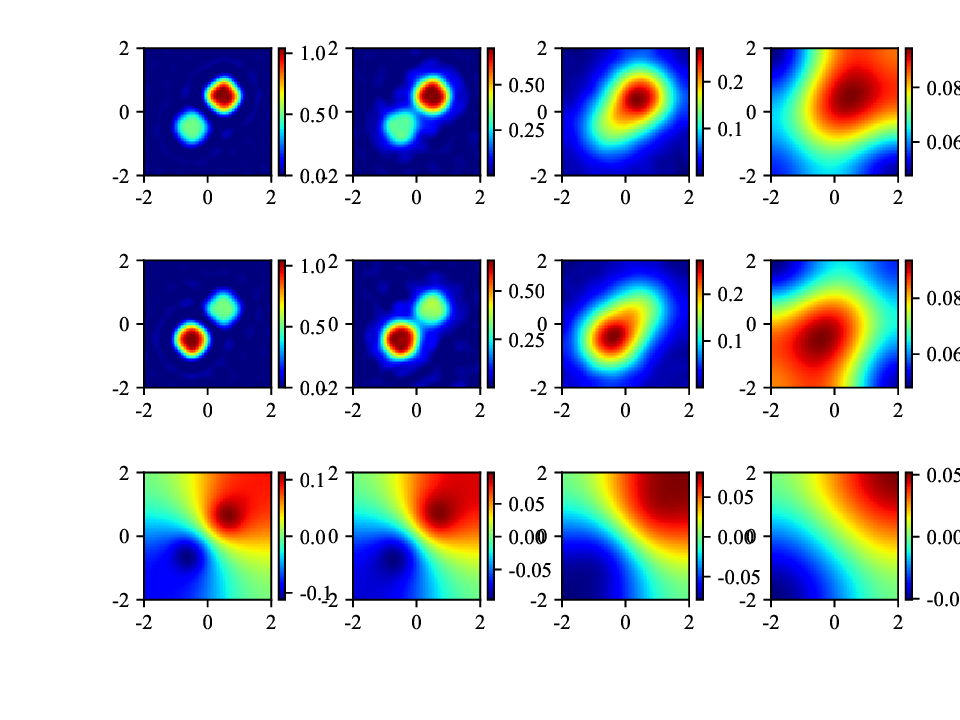}
	\caption{$c_1$ (the left row), $c_2$ (the middle row) and $\phi$ (the right row) at $t=0$ (the first column), $t=0.04$ (the second column), $t=0.2$ (the third column) and $t=1$ (the fourth column), where $\ m=800,\ n_e=20,\ r_1=r_2=r_3=2,\ \lambda=100$ in Example \ref{pnp-benchmark2}.}
	\label{pnp-benchmark2-p1}
\end{figure}

\begin{figure}[!htbp] 		
	\centering
	\includegraphics[scale=0.6]{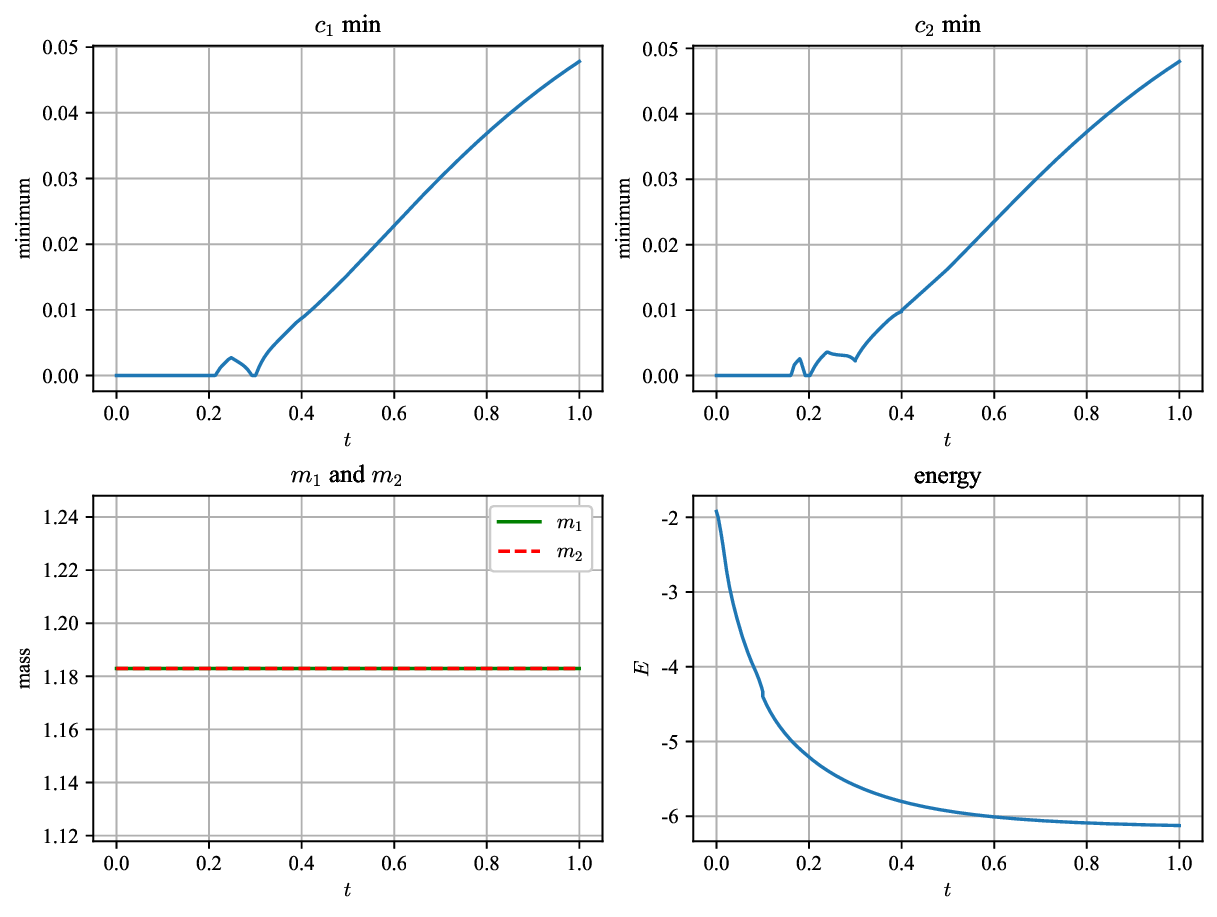}
	\caption{The figure reports the computed free-energy functional, the minimum values, and the masses of \(c_1,c_2\) obtained by the SO-RaNN method (with time blocks), where $\ m=800,\ n_e=20,\ r_1=r_2=r_3=2,\ \lambda=100$ in Example \ref{pnp-benchmark2}.}
	\label{pnp-benchmark2-p2}
\end{figure}

\begin{example}[Three-component PNP Equations, benchmark test]\label{pnp-benchmark3}
In this example, we consider the equations \eqref{PNP-model} with three ions in the domain $\Omega\times I=[0,1]^3$, $D_1=D_2=D_3=1$, $z_1=1,\ z_2=-1,\ z_3=2$, $\epsilon_p=1$, the initial conditions are given by
\begin{equation}
    \left\{
    \begin{array}{rrll}
        c_1(x,y,0) &=& 3x^2-2x^3+3y^2-2y^3,  \\
        c_2(x,y,0) &=& 3(3x^2-2x^3)(3y^2-2y^3)+\frac{9}{4}, \\
        c_3(x,y,0) &=& x^2(1-x)^2+y^2(1-y)^2+\frac{14}{15}.
    \end{array}
    \right.
\end{equation}
\end{example}

Figure \ref{pnp-benchmark3-p1} shows the computed profiles obtained with \(m=400\), \(n_e=20\), \(r_1=r_2=r_3=r_4=1\), \(\lambda=100\), and 10 uniform time blocks for \(I=[0,1]\). The free-energy curve reported in the figure is computed from the final numerical solution, namely the corrected concentrations together with the final gauge-fixed potential. Its behavior is consistent with the expected dissipative trend of the multi-ion PNP dynamics in this benchmark.

\begin{figure}[!htbp] 		
	\centering
	\includegraphics[scale=0.6]{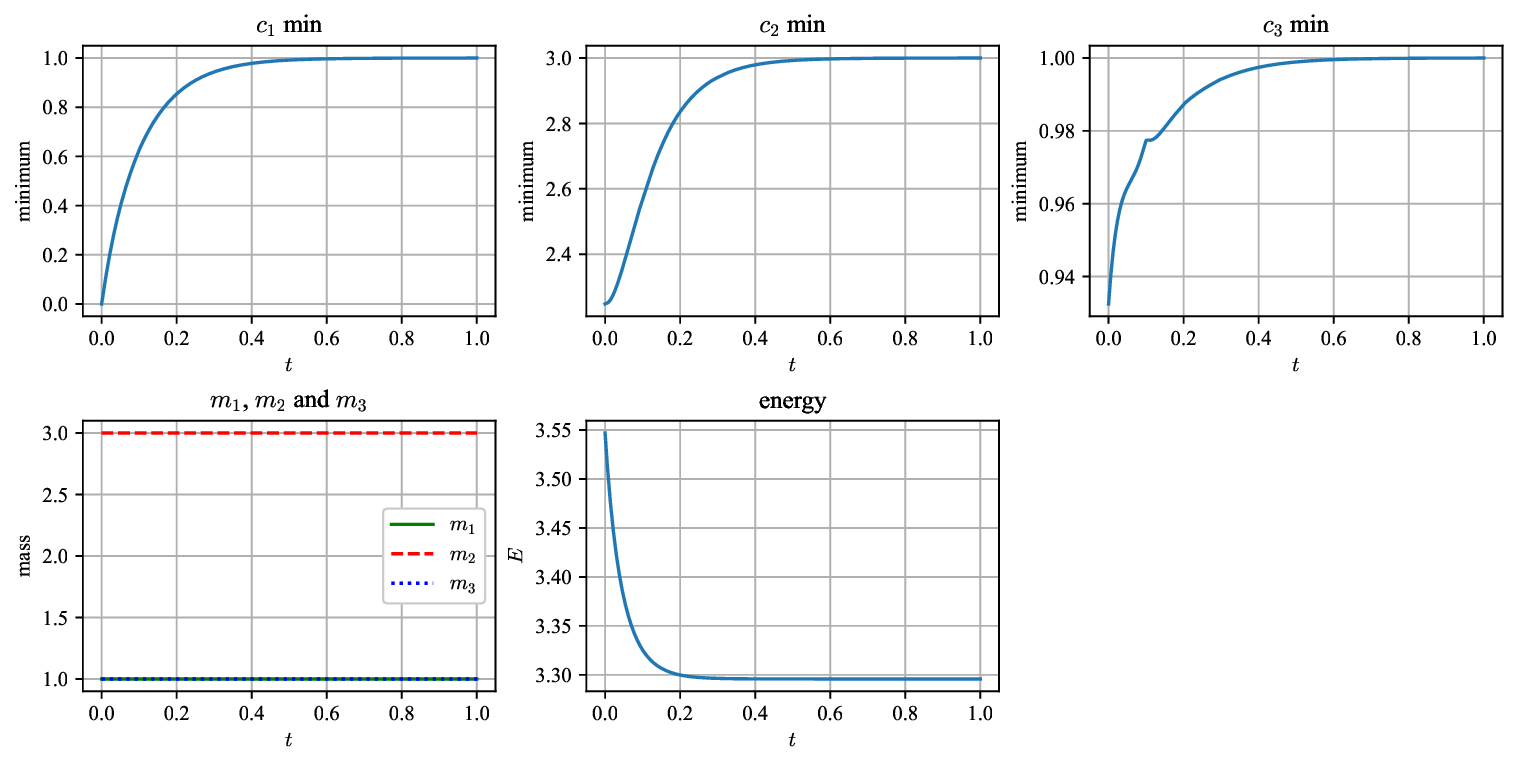}
	\caption{The figure reports the computed free-energy functional, the minimum values, and the masses of \(c_1,c_2,c_3\) obtained by the SO-RaNN method (with time blocks), where $\ m=400,\ n_e=20,\ r_1=r_2=r_3=r_4=1,\ \lambda=100$ in Example \ref{pnp-benchmark3}.}
	\label{pnp-benchmark3-p1}
\end{figure}

\subsection{PNP-NS system}
\begin{example}[Two-component PNP-NS Equations, accuracy test]\label{pnpns-accuracy}
In this example, we use the following augmented equations with manufactured analytic solutions as the test problem
\begin{subequations}\label{PNPNS-f}
    \begin{align}
        \frac{\partial c_i}{\partial t} + (\mathbf{u}\cdot \nabla)c_i &= D_i\nabla \cdot(\nabla c_i + z_i c_i \nabla \phi) + f_i,\ i=1,...,N,\ {\rm in}\ \Omega\times I, \\
        -\epsilon_p^2 \Delta \phi &= \sum\limits_{i=1}^{N}z_i c_i + g,\ {\rm in}\ \Omega\times I,\\
        \frac{\partial \mathbf{u}}{\partial t} + (\mathbf{u}\cdot \nabla)\mathbf{u} - \nu \Delta \mathbf{u} + \nabla p&= - \left(\sum\limits_{i=1}^{N}z_i c_i\right)\nabla \phi + \mathbf{f}_{\rm NS},\ {\rm in}\ \Omega\times I, \\
        \nabla \cdot \mathbf{u} &= 0,\ {\rm in}\ \Omega\times I,
    \end{align}
\end{subequations}
where $\Omega\times I=[-1,1]^2\times [0,0.1]$, $D_1=D_2=1$, $z_1=1,\ z_2=-1$, $\epsilon_p=1$, exact solutions are as follows
\begin{equation}
    \left\{
    \begin{array}{rrll}
        c_1 &=& {\rm cos}(\pi x){\rm cos}(\pi y){\rm sin}^2(t) + 1.1,  \\
        c_2 &=& -{\rm cos}(\pi x){\rm cos}(\pi y){\rm sin}^2(t) + 1.1, \\
        \phi &=& \frac{1}{\pi^2}{\rm cos}(\pi x){\rm cos}(\pi y){\rm sin}^2(t), \\
        \mathbf{u} &=& (\pi{\rm sin}^2(\pi x){\rm sin}(2\pi y){\rm sin}^2(t),-\pi{\rm sin}(2\pi x){\rm sin}^2(\pi y){\rm sin}^2(t)), \\
        p &=& {\rm sin}(\pi x){\rm sin}(\pi y){\rm sin}^2(t).
    \end{array}
    \right.
\end{equation}
The source terms $f_1,\ f_2,\ g,\ \mathbf{f}_{\rm NS}$, boundary and initial conditions can be calculated with the exact solution.
\end{example}

Since the manufactured source terms change the species masses, the correction target is the exact time-dependent mass of the manufactured solution at each selected correction instant.

We fix $n_e=20,\ r_1=r_2=r_3=r_4=r_5=4,\ \lambda=100$, and summarize the results for $\rm Re=10 (\nu=\frac{1}{\rm Re})$ and $\rm Re=1000$ with different $m$ in the Table \ref{pnpns-accuracy-table1} and Table \ref{pnpns-accuracy-table2}. The numerical errors decrease as the number of neurons increases, and the spatial divergence-free condition is achieved up to numerical precision. This manufactured example suggests that the proposed method can achieve accurate approximations for the tested PNP--NS configurations, while the SP-RaNN ansatz enforces the divergence-free constraint to numerical precision.

\begin{table}[H]
\centering
\begin{tabular}{cccccccccc}
\toprule
$m$&$e_{c_1}$&$e_{c_2}$&$e_{\phi}$&$e_{u_1}$&$e_{u_2}$&$e_{p}$&$e_{\nabla \cdot \mathbf{u}}$&$t$ (CPU)&ite\\
\midrule
200 &7.46E-04	&6.30E-04	&4.16E-04	&5.41E-03	&4.81E-03	&1.20E-01	&1.88E-15	&1.7572715	&4\\
400 &2.93E-05	&2.74E-05	&4.24E-05	&5.71E-04	&5.65E-04	&1.59E-02	&8.57E-15	&4.1277681	&4\\
800 &7.91E-07	&5.12E-07	&1.36E-06	&1.55E-05	&1.91E-05	&5.10E-04	&1.38E-14	&11.0776446	&4\\
1600 &2.31E-08	&2.83E-08	&6.27E-09	&1.53E-06	&1.39E-06	&7.30E-05	&4.18E-14	&35.8386845	&4\\
\bottomrule
\end{tabular}
\caption{$L^2$ errors at $t=0.1$ for $\rm Re=10$ with different $m$, where $\ n_e=20,\ r_1=r_2=r_3=r_4=r_5=4,\ \lambda=100$ in Example \ref{pnpns-accuracy}.}
\label{pnpns-accuracy-table1}
\end{table}

\begin{table}[H]
\centering
\begin{tabular}{cccccccccc}
\toprule
$m$&$e_{c_1}$&$e_{c_2}$&$e_{\phi}$&$e_{u_1}$&$e_{u_2}$&$e_{p}$&$e_{\nabla \cdot \mathbf{u}}$&$t$ (CPU)&ite\\
\midrule
200 &7.46E-04	&6.30E-04	&4.16E-04	&5.67E-03	&5.94E-03	&1.05E-01	&1.61E-15	&1.7559143&4\\
400 &2.93E-05	&2.74E-05	&4.24E-05	&7.65E-04	&7.89E-04	&1.42E-02	&6.36E-15	&4.1023153&4\\
800 &7.91E-07	&5.12E-07	&1.36E-06	&1.63E-05	&1.86E-05	&5.57E-04	&1.54E-14	&10.986527&4\\
1600 &2.31E-08	&2.83E-08	&6.30E-09	&4.89E-07	&4.32E-07	&2.21E-05	&3.63E-14	&35.7213455&4\\
\bottomrule
\end{tabular}
\caption{$L^2$ errors at $t=0.1$ for $\rm Re=1000$ with different $m$, where $\ n_e=20,\ r_1=r_2=r_3=r_4=r_5=4,\ \lambda=100$ in Example \ref{pnpns-accuracy}.}
\label{pnpns-accuracy-table2}
\end{table}

\begin{example}[Two-component PNP-NS Equations, benchmark test]\label{pnpns-benchmark1}
In this example, we consider the equations \eqref{PNPNS-model-2} with $\rm Re=100$ in the domain $\Omega\times I=[-1,1]^2\times [0,1]$, $D_1=D_2=1$, $z_1=1,\ z_2=-1$, $\epsilon_p=1$, the initial conditions are given by
\begin{equation}
    \left\{
    \begin{array}{rrll}
        c_1(x,y,0) &=& {\rm cos}(\pi x){\rm cos}(\pi y) + 1.1,  \\
        c_2(x,y,0) &=& -{\rm cos}(\pi x){\rm cos}(\pi y) + 1.1, \\
        \mathbf{u}(x,y,0) &=& (\pi{\rm sin}^2(\pi x){\rm sin}(2\pi y),-\pi{\rm sin}(2\pi x){\rm sin}^2(\pi y)).
    \end{array}
    \right.
\end{equation}
\end{example}

We divide $I=[0,1]$ into 10 blocks uniformly, and fix $m=400,\ n_e=20,\ r_1=r_2=r_3=r_4=r_5=4$, the results are plotted in Figure \ref{pnpns-benchmark1-p1} and Figure \ref{pnpns-benchmark1-p2}. Figure \ref{pnpns-benchmark1-p1} presents the snapshots at $t=0,\ 0.1,\ 0.6,\ 1$ of $c_1,\ c_2,\ u_1,\ u_2$, which is consistent with \cite{Zhou2023} (Figs. 5.6-5.9). Figure \ref{pnpns-benchmark1-p2} shows the evolution of the free-energy functional computed from the final numerical solution, including the corrected concentrations, the final gauge-fixed potential, and the velocity field. In this test, the free-energy curve is observed to be nonincreasing, while positivity, mass matching, and the divergence-free constraint are maintained to the reported numerical accuracy.

\begin{figure}[!htbp] 		
	\centering
	\includegraphics[scale=0.8]{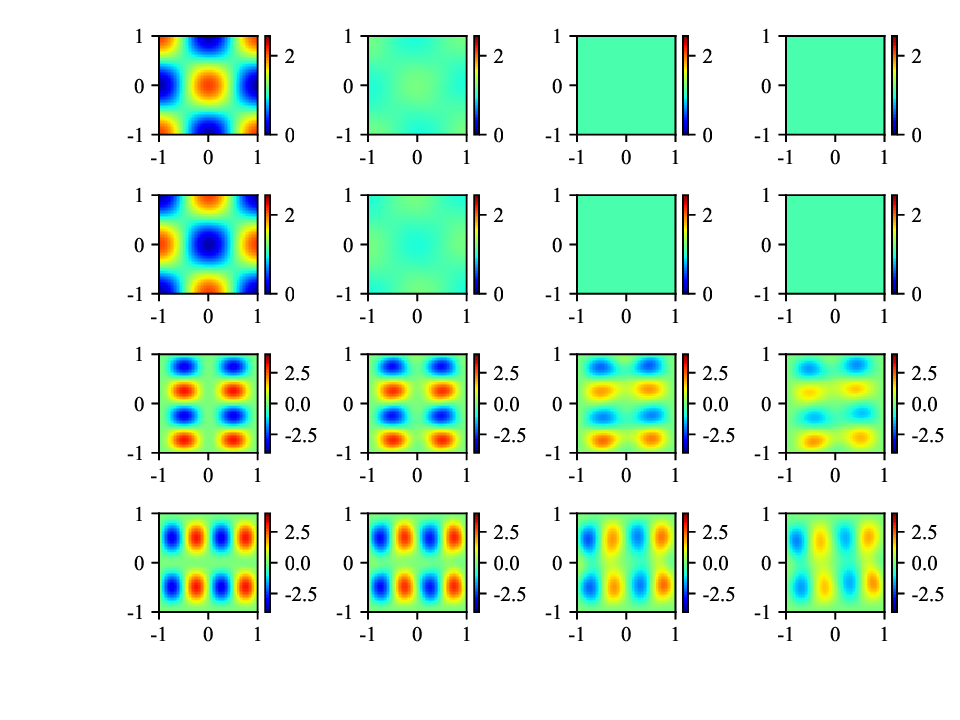}
	\caption{$c_1$ (the first row), $c_2$ (the second row), $u_1$ (the third row) and $u_2$ (the fourth row) at $t=0$ (the first column), $t=0.1$ (the second column), $t=0.6$ (the third column) and $t=1$ (the fourth column), where $\ m=400,\ n_e=20,\ r_1=r_2=r_3=r_4=r_5=4,\ \lambda=100$ in Example \ref{pnpns-benchmark1}.}
	\label{pnpns-benchmark1-p1}
\end{figure}

\begin{figure}[!htbp] 		
	\centering
	\includegraphics[scale=0.6]{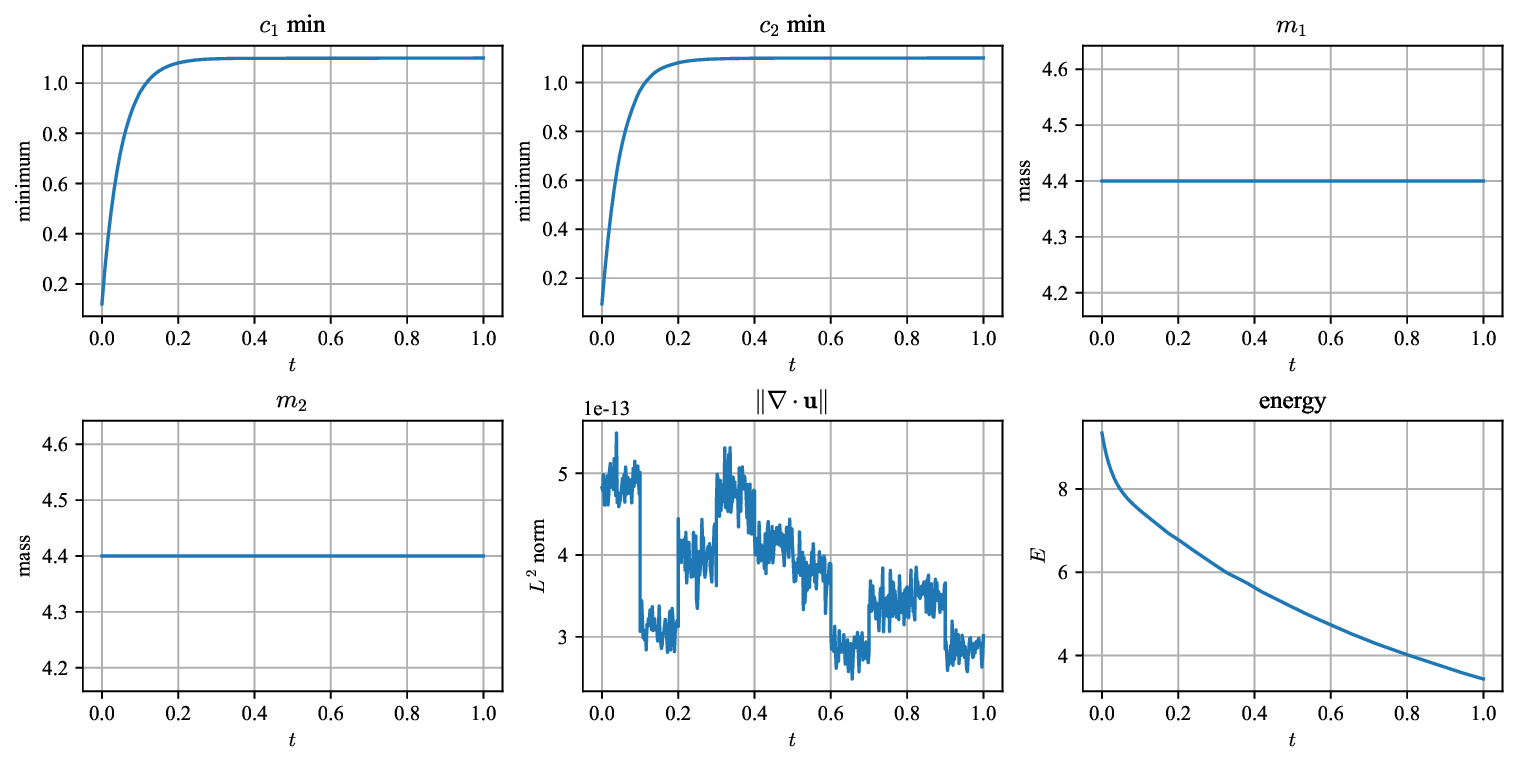}
	\caption{the figure reports the computed free-energy functional, the minimum values, and the masses of \(c_1,c_2\), and the $\Vert \nabla \cdot \mathbf{u}_{\rho} \Vert_0$ obtained by the SO-RaNN method (with time blocks), where $\ m=400,\ n_e=20,\ r_1=r_2=r_3=r_4=r_5=4,\ \lambda=100$ in Example \ref{pnpns-benchmark1}.}
	\label{pnpns-benchmark1-p2}
\end{figure}

\section{Conclusions}
\label{sec5}
In this paper, we proposed a structure-oriented randomized neural network method for the PNP and PNP-NS systems. The method combines raw RaNN solvers for decoupled linearized subproblems with positivity cut-off, selected-time mass correction, and an SAV-based auxiliary-variable correction. A final gauge-fixed Poisson least-squares update is used to compute the reported potential. For the PNP-NS system, an SP-RaNN representation is used for the velocity field, so that incompressibility is satisfied pointwise by construction.

On the theoretical side, we derived residual-based error estimates for the raw, uncorrected RaNN solvers of the linearized Nernst--Planck subproblems and for compatible Neumann Poisson subproblems, and formulated a conditional local-in-time convergence result for the raw outer Picard iteration of the PNP system together with its inexact-Picard interpretation for the implemented raw RaNN subsolvers. We also analyzed the correction steps separately: the positivity cut-off is stable at the value level, the mass correction gives exact mass matching at selected correction instants, and the ideal SAV update yields monotonicity of the SAV auxiliary variable. For the PNP-NS system, we used a divergence-free trial space for the velocity field, established an approximation result for the SP-RaNN ansatz, and gave a conditional residual-based error statement for the corresponding linearized Oseen-type subproblem.

Several directions are worth further investigation. First, a more complete analysis for the PNP-NS system, including the coupled electrohydrodynamic outer iteration, would be an important extension of the present work. Second, the proposed framework can be extended to other coupled PDE systems with positivity, conservation, dissipation, or incompressibility structures. Third, it would also be interesting to incorporate more physical structures directly into the neural ansatz, thereby reducing the reliance on post-processing corrections and enforcing more structures intrinsically at the approximation level.


\begin{thebibliography}{99}
    \bibitem{Arendt2016}
    W. Arendt and S. Monniaux, Maximal regularity for non-autonomous Robin boundary conditions, \emph{Mathematische Nachrichten} {\bf 289} (2016), 1325-1340. 

    \bibitem{Barron1993}
    A. R. Barron, Universal approximation bounds for superpositions of a sigmoidal function, \emph{IEEE Transactions on Information Theory} {\bf 39} (1993), 930-945.
    
    \bibitem{Bazant2004}
    M. Z. Bazant, K. Thornton, A. Ajdari, Diffuse-charge dynamics in electrochemical systems, \emph{Physical Review E} {\bf 70} (2004), 021506. 

    \bibitem{Biler1994}
    P. Biler, W. Hebisch, and T. Nadzieja, The Debye system: existence and large time behavior of solutions, \emph{Nonlinear Analysis: Theory, Methods \& Applications} {\bf 23} (1994), 1189–1209.

    \bibitem{Bonilla2025}
    J. Bonilla and J. V. Guti\'errez-Santacreu, Physics-based stabilized finite element approximations of the Poisson–Nernst–Planck equations, \emph{Computer Methods in Applied Mechanics and Engineering} {\bf 443} (2025), 118035.

    \bibitem{Chen2022}
    J. Chen, X. Chi, W. E and Z. Yang, Bridging Traditional and Machine Learning-based Algorithms for Solving PDEs: The Random Feature Method, \emph{Journal of Machine Learning} {\bf 1} (2022), 268-298.
    
    \bibitem{Chen1995}
    T. Chen and H. Chen, Universal approximation to nonlinear operators by neural networks with arbitrary activation functions and its application to dynamical systems, \emph{IEEE Transactions on Neural Networks} {\bf 6} (1995), 911–917.
          
    \bibitem{Choi2006}
    H.-W. Choi, and M. Paraschivoiu, Advanced hybrid-flux approach for output bounds of electro-osmotic flows: adaptive refinement and direct equilibrating strategies, \emph{Microfluidics and Nanofluidics} {\bf 2} (2006), 154–170.

    \bibitem{Dang2023}
    H. Dang and F. Wang, Local Randomized Neural Networks with Hybridized Discontinuous Petrov-Galerkin Methods for Stokes-Darcy Flows, \emph{Physics of Fluids} {\bf 36} (2024), 087138.
    
    \bibitem{Dang2024}
    H. Dang, F. Wang and S. Jiang, Adaptive Growing Randomized Neural Networks for Solving Partial Differential Equations, (2024), arXiv: 2408.17225v3.

    \bibitem{Dehghan2023}
    M. Dehghan, Z. Gharibi and R. Ruiz-Baier, Optimal Error Estimates of Coupled and Divergence-Free Virtual Element Methods for the Poisson–Nernst–Planck/Navier–Stokes Equations and Applications in Electrochemical Systems, \emph{Journal of Scientific Computing} {\bf 94} (2023), 72.

    \bibitem{Ding2020}
    J. Ding, Z. Wang and S. Zhou, Structure-preserving and efficient numerical methods for ion transport, \emph{Journal of Computational Physics} {\bf 418} (2020) 109597.

    \bibitem{Dong2024}
    L. Dong, D. He, Y. Qin and Z. Zhang, A positivity-preserving, linear, energy stable and convergent numerical scheme for the Poisson–Nernst–Planck (PNP) system, \emph{Journal of Computational and Applied Mathematics} {\bf 444} (2024), 115784.

    \bibitem{Dong2021}
    S. Dong and Z. Li, Local extreme learning machines and domain decomposition for solving linear and nonlinear partial differential equations, \emph{Computer Methods in Applied Mechanics and Engineering} {\bf 387} (2021), 114129.
    
    \bibitem{Dong2023}
    S. Dong and Y. Wang. A method for computing inverse parametric PDE problems with random-weight neural networks. \emph{Journal of Computational Physics} {\bf 489} (2023), 112263.
    
    \bibitem{E2017}
    W. E and B. Yu, The deep Ritz method: A deep learning-based numerical algorithm for solving variational problems, \emph{Communications in Mathematics and Statistics} {\bf 6} (2018), 1–12.

    \bibitem{Eisenberg1998}
    B. Eisenberg, Ionic channels in biological membranes-electrostatic analysis of a natural nanotube, \emph{Contemporary Physics} {\bf 39} (1998), 447–466.

    \bibitem{Fu2022}
    G. Fu and Z. Xu, High-order space–time finite element methods for the Poisson–Nernst–Planck equations: positivity and unconditional energy stability, \emph{Computer Methods in Applied Mechanics and Engineering} {\bf 395} (2022) 115031.

    \bibitem{Gajewski1986}
    H. Gajewski, K. G\"{o}ger, On the basic equations for carrier transport in semiconductors, \emph{Journal of Mathematical Analysis and Applications} {\bf 113} (1986), 12–35.
    
    \bibitem{He2019}
    D. He, K. Pan and X. Yue, A Positivity Preserving and Free Energy Dissipative Difference Scheme for the Poisson–Nernst–Planck System, \emph{Journal of Scientific Computing} {\bf 81} (2019), 436-458.

    \bibitem{Hu2020}
    J. Hu and X. Huang, A fully discrete positivity-preserving and energy-dissipative finite difference scheme for Poisson–Nernst–Planck equations, \emph{Numerische Mathematik} {\bf 145} (2020), 77-115.

    \bibitem{Huang2021}
    F. Huang and J. Shen, Bound/positivity preserving and energy stable scalar auxiliary variable schemes for dissipative systems: applications to Keller-Segel and Poisson-Nernst-Planck equations, \emph{SIAM Journal on Scientific Computing} {\bf 43} (2021), A1832--A1857.

    \bibitem{Huang2006}
    G. Huang, Q. Zhu and C. K. Siew, Extreme learning machine: theory and applications, \emph{Neurocomputing} {\bf 70} (2006), 489–501.
    
    \bibitem{Igelnik1995}
    B. Igelnik and Y.H. Pao, Stochastic choice of basis functions in adaptive function approximation and the functional-link net, \emph{IEEE Transactions on Neural Networks} {\bf 6} (1995), 1320–1329.
    
    \bibitem{Jerome1996}
    J. W. Jerome, Analysis of Charge Transport: A Mathematical Study of Semiconductor Devices, Springer, Berlin, 1996.

    \bibitem{Jin2021}
    X. Jin, S. Cai, H. Li and G. E. Karniadakis, NSFnets (Navier-Stokes flow nets): Physics-informed neural networks for the incompressible Navier-Stokes equations, \emph{Journal of Computational Physics} {\bf 426} (2021), 109951.

    \bibitem{Kajiwara2022}
    N. Kajiwara, Maximal $L_p-L_q$ regularity for the Stokes equations with various boundary conditions in the half space, arXiv: 2201.05306.

    \bibitem{Li2025}
    Y. Li and F. Wang, Local randomized neural networks with finite difference methods for interface problems, \emph{Journal of Computational Physics} {\bf 529} (2025), 113847.

    \bibitem{Li2026}
    Y. Li, F. Wang and L. Li, Structure-preserving Randomized Neural Networks for Incompressible Magnetohydrodynamics Equations, (2026), arXiv: 2603.01102.

    \bibitem{Lions1972}
    J. L. Lions and E. Magenes, Non-Homogeneous Boundary Value Problems and Applications, Vol. I, Springer-Verlag, Berlin--Heidelberg--New York, 1972.

    \bibitem{Liu2014}
    H. Liu and Z. Wang, A free energy satisfying finite difference method for Poisson–Nernst–Planck equations, \emph{Journal of Computational Physics} {\bf 268} (2014), 362–376.

    \bibitem{Lu2021}
    L. Lu, X. Meng, Z. Mao, and G. E. Karniadakis, DeepXDE: A Deep Learning Library for Solving Differential Equations, \emph{SIAM Review} {\bf 63} (2021), 208-228.

    \bibitem{Markowich1990}
    P. Markowich, C. Ringhofer, and C. Schmeiser, Semiconductor Equations, Springer, New York, 1990.

    \bibitem{Moseley2023}
    B. Moseley, A. Markham and T. Nissen-Meyer, Finite basis physics-informed neural networks (FBPINNs): a scalable domain decomposition approach for solving differential equations, \emph{Advances in Computational Mathematics} {\bf 49} (2023), 62.

    \bibitem{Nernst1889}
    W. Nernst, Die elektromotorische wirksamkeit der jonen, \emph{Zeitschrift Fur Physikalische Chemie} {\bf 4} (1889), 129--181.
    
    \bibitem{Pan2024}
    M. Pan, S. Liu, W. Zhu, F. Jiao and D. He, A linear, second-order accurate, positivity-preserving and unconditionally energy stable scheme for the Navier–Stokes–Poisson–Nernst–Planck system, \emph{Communications in Nonlinear Science and Numerical Simulation} {\bf 131} (2024), 107873.

    \bibitem{Pao1994}
    Y. Pao, G. Park and D. Sobajic, Learning and generalization characteristics of the random vector functional-link net, \emph{Neurocomputing} {\bf 6} (1994), 163–180.

    \bibitem{Planck1890}
    M. Planck, Ueber die erregung von electricität und Wärme in electrolyten, \emph{Annalen der Physik und Chemie} {\bf 275} (1890), 161-186.
    
    \bibitem{George2019}
    M. Raissi, P. Perdikaris and G. E. Karniadakis, Physics-informed neural networks: A deep learning framework for solving forward and inverse problems involving nonlinear partial differential equations, \emph{Journal of Computational Physics} {\bf 378} (2019), 686–707.

    \bibitem{Raymond2007}
    J.-P. Raymond, Stokes and Navier–Stokes equations with nonhomogeneous boundary conditions, \emph{Annales de l'Institut Henri Poincaré C, Analyse non linéaire} {\bf 24} (2007), 921-951.
    
    \bibitem{Schmuck2009}
    M. Schmuck, Analysis of the Navier-Stokes-Nernst-Planck-Poisson system, \emph{Mathematical Models and Methods in Applied Sciences} {\bf 19} (2009), 993-1014.

    \bibitem{Shang2023}
    Y. Shang, F. Wang and J. Sun, Randomized neural network with petrov–galerkin methods for solving linear and nonlinear partial differential equations, \emph{Communications in Nonlinear Science and Numerical Simulation} {\bf 127} (2023), 107518.
    
    \bibitem{Shang2025}
    Y. Shang, A. Heinlein, S. Mishra and Fei Wang, Overlapping Schwarz Preconditioners for Randomized Neural Networks with Domain Decomposition, \emph{Computer Methods in Applied Mechanics and Engineering} {\bf 442} (2025), 118011.
    
    \bibitem{Shen2018}
    J. Shen, J. Xu and J. Yang, The scalar auxiliary variable (SAV) approach for gradient flows, \emph{Journal of Computational Physics} {\bf 353} (2018), 407–416.

    \bibitem{Shen2021}
    J. Shen and J. Xu, Unconditionally positivity preserving and energy dissipative schemes for Poisson–Nernst–Planck equations, \emph{Numerische Mathematik} {\bf 148} (2021), 671-697.

    \bibitem{Sheng2021}
    H. Sheng and C. Yang, PFNN: A penalty-free neural network method for solving a class of second-order boundary-value problems on complex geometries, \emph{Journal of Computational Physics} {\bf 428} (2021), 110085.

    \bibitem{Shibata2007}
    Y. Shibata and R. Shimada, On the Stokes equation with Robin boundary condition, \emph{Advanced Studies in Pure Mathematics} {\bf 47} (2007), 341-348.

    \bibitem{Sirignano2018}
    J. Sirignano and K. Spiliopoulos, DGM: A deep learning algorithm for solving partial differential equations, \emph{Journal of Computational Physics} {\bf 375} (2018), 1339–1364.
    
    \bibitem{Sun2024}
    J. Sun, S. Dong and F. Wang, Local randomized neural networks with discontinuous galerkin methods for partial differential equations, \emph{Journal of Computational and Applied Mathematics} {\bf 445} (2024), 115830.
    
    \bibitem{Sun2025}
    J. Sun and F. Wang, Local Randomized Neural Networks with Discontinuous Galerkin Methods for KdV-type and Burgers Equations, \emph{Communications in Nonlinear Science and Numerical Simulation} {\bf 150} (2025), 108957.

    \bibitem{Tong2024}
    F. Tong and Y. Cai, Positivity preserving and mass conservative projection method for the Poisson-Nernst-Planck equation, \emph{SIAM Journal on Numerical Analysis} {\bf 62} (2024), 2004-2024.

    \bibitem{Wang2024}
    Y. Wang and S. Dong, An extreme learning machine-based method for computational PDEs in higher dimensions, \emph{Computer Methods in Applied Mechanics and Engineering} {\bf 418} (2024), 116578.

    \bibitem{Xu2025}
    Z. Xu and Z. Sheng, Subspace method based on neural networks for solving the partial differential equation, \emph{Computers \& Mathematics with Applications} {\bf 195} (2025), 109-138.

    \bibitem{Yang2026}
    Y. Yang and F. Wang, Adaptive-distribution randomized neural networks for PDEs: A low-dimensional distribution-learning framework, (2026), arXiv:2604.23999.
    
    \bibitem{Yu2025}
    Z. Yu, J. Shen, C. Wang and Q. Cheng, A Decoupled Structure Preserving Scheme for the Poisson-Nernst-Planck Navier-Stokes Equations and its Error Analysis, \emph{Journal of Scientific Computing} {\bf 104} (2025), 105.

    \bibitem{Zhang2024}
    Z. Zhang, F. Bao, L. Ju and G. Zhang, Transferable neural networks for partial differential equations, \emph{Journal of Scientific Computing} {\bf 99} (2024), 2.
    
    \bibitem{Zhou2023}
    X. Zhou and C. Xu, Efficient time-stepping schemes for the Navier-Stokes-Nernst-Planck-Poisson equations, \emph{Computer Physics Communications} {\bf 289} (2023), 108763.
\end{thebibliography}
\end{document}